\documentclass[11pt,a4paper]{article}
\pdfoutput=1

\usepackage[T1]{fontenc}
\usepackage[utf8]{inputenc}
\usepackage{amsmath}
\usepackage[english]{babel}
\usepackage{amsfonts}
\usepackage{amssymb}
\usepackage{graphicx}
\usepackage{xurl}
\usepackage{amssymb,amsmath,mathtools,proof,amsthm}
\usepackage{makeidx}
\usepackage{lscape}
\usepackage{latexsym}
\usepackage{hyperref}
\usepackage[margin=1in]{geometry}
\usepackage[dvipsnames]{xcolor}
\usepackage[protrusion=true,expansion=true]{microtype}
\usepackage{hyperref}
\usepackage{orcidlink}
\usepackage{caption}
\usepackage{multicol}

\theoremstyle{definition}
\newtheorem{definition}{Definition}
\newtheorem{theorem}{Theorem}
\newtheorem{lemma}{Lemma}

\newtheorem{corollary}{Corollary}
\newtheorem{notation}{Notation}

\newcommand{\NN}{\mathbb{N}}

\newcommand{\TCF}{\textsf{TCF}}

\newcommand{\BB}{\mathbb{B}}
\newcommand{\true}{\operatorname{tt}}
\newcommand{\false}{\operatorname{ff}}
\newcommand{\andnc}{\wedge^{\operatorname{nc}}}
\newcommand{\andb}{\wedge^b}
\newcommand{\exd}{\exists}
\newcommand{\exl}{\exists}
\newcommand{\exbn}[1]{\exists^{<#1}}
\newcommand{\exbp}[1]{\exists^{\leq #1}}
\newcommand{\exnc}{\exists^{\operatorname{nc}}}
\newcommand{\orr}{\vee}
\newcommand{\ord}{\vee}
\newcommand{\oru}{\vee}
\newcommand{\orb}{\vee^{\operatorname{b}}}
\newcommand{\andr}{\wedge}
\newcommand{\Prime}[1]{\textbf{P}(#1)}
\newcommand{\Primes}[2]{\textbf{P}_{#2}(#1)}
\newcommand{\Pms}{\operatorname{Pms}}
\newcommand{\lf}{\operatorname{LF}}
\newcommand{\Suc}{\operatorname{S}}
\newcommand{\SZero}{\operatorname{S}_0}
\newcommand{\SOne}{\operatorname{S}_1}
\newcommand{\SPos}{\operatorname{S}_{0/1}}
\newcommand{\PP}{\mathbb{P}}
\newcommand{\IsSquare}[1]{\operatorname{IsSq}(#1)}
\newcommand{\falsum}{\textbf{F}}
\newcommand{\Total}{\textbf{T}}
\newcommand{\PosToNat}{\operatorname{PosToNat}}
\newcommand{\NatToPos}{\operatorname{NatToPos}}
\newcommand{\realizer}[1]{{#1}^{\textbf{r}}}
\newcommand{\et}[1]{\operatorname{et}(#1)}
\newcommand{\sep}{, }
\newcommand{\Href}[2]{\texorpdfstring{\href{#1}{\texttt{#2}}}{#2}}

\title{Verified Program Extraction in Number Theory: The Fundamental Theorem of Arithmetic and Relatives\footnote{The research for this document was funded by the Austrian Science Fund (FWF) \textbf{10.55776/ESP576}.\\
We would like to thank \emph{Helmut Schwichtenberg} for valuable advice on writing the Minlog code and for integrating it into the official Minlog version, as well as \emph{Iosif Petrakis} for proofreading the paper.\\
I am also grateful to the anonymous reviewers for their helpful comments,
which significantly improved and extended this article.}}

\author{
Franziskus Wiesnet \orcidlink{0000-0003-3870-6984}\\
{\normalsize \href{mailto:franziskus.wiesnet@tuwien.ac.at}{franziskus.wiesnet@tuwien.ac.at}}\\
{\normalsize Vienna University of Technology}\\
{\normalsize Wiedner Hauptstraße 8-10, 1040 Vienna, Austria}
}
\date{}
\begin{document}
\maketitle
\begin{abstract}
This article revisits standard theorems from elementary number theory from a constructive, algorithmic, and proof-theoretic perspective, framed within the theory of computable functionals $\TCF$. Key examples include Bézout’s identity, the fundamental theorem of arithmetic, and Fermat’s factorisation method. All definitions and theorems are fully formalised in the proof assistant Minlog, laying the foundation for a comprehensive formal framework for number theory within Minlog.

While formalisation guarantees correctness, the primary emphasis is on the computational content of proofs. Leveraging Minlog's built-in program extraction, we obtain executable terms and export them as Haskell code.

The efficiency of the extracted programs plays a central role. We show how performance considerations influence even the original formulation of theorems and proofs. In particular, we compare formalisations based on binary encodings of natural numbers with those using the traditional unary (successor-based) representation.

We present several core proofs in detail and reflect on the challenges that arise from formalisation in contrast to informal reasoning. The complete formalisation is available online and linked throughout. Minlog's tactic scripts are designed to follow the structure of natural-language proofs, allowing each derivation step to be traced precisely and thereby bridging the gap between formal and classical mathematical reasoning.\vspace{1mm}

\end{abstract}
\section{Introduction}
\label{Sec:Intro}

\subsection{Overview and Motivation}
This article is based on the formalisation of statements from elementary arithmetic in Minlog.
To date, Minlog has primarily been applied to the formalisation of analysis. The only existing work addressing number theory in Minlog  -- specifically, the greatest common divisor -- was carried out by Helmut Schwichtenberg and Ulrich Berger \cite{berger1996greatest}, and it mainly focused on program extraction from classical proofs.
However, given the functions that Minlog already provides, applying it to elementary number theory is a very natural next step:

Minlog is a proof assistant specifically designed for formal program extraction from implemented proofs, yielding extracted terms that can be executed as Haskell programs. Elementary number theory is particularly well suited as a case study because many of its central theorems have an inherently algorithmic character and are naturally stated in an existence-oriented form (e.g.~greatest common divisors, Bézout coefficients, and factorisation), making them amenable to program extraction.
More generally, Minlog supports two complementary verification paradigms. On the one hand, one can obtain certified programs by extracting them from constructive proofs. On the other hand, one can implement an algorithm directly and subsequently verify it by proving the desired properties. We will combine both approaches in this article. In this way, the extracted terms remain explicit enough for the computational content to stay transparent.

Moreover, Minlog's underlying theory -- most notably the ability to mark predicates as computationally relevant or computationally irrelevant -- provides a precise view of which parts of a proof contribute to the computational content. This makes it possible to organise proofs in a way that keeps the extracted term as efficient as possible. We will return to these issues repeatedly throughout the paper.
A closely related point is the representation of natural numbers. In particular, Minlog provides positive binary numbers as a built-in data type; consequently, extracted terms directly operate on the binary representation, which is crucial for practical performance. In many proofs we will even switch between different number representations, which is straightforward in Minlog. Some parts are carried out in unary form with zero and successor to simplify reasoning, while others are carried out in binary form to obtain efficient extracted code.
\subsection{Comparison with Other Proof Assistants and Novelties of this Article}
\label{Sec:Novelties}
Table \ref{tab:PAs} lists the key theorems and algorithms in this article and indicates where they have been implemented in other proof assistants. A dash (--) means that, at present, no implementation is known. Of course, the existence of such an implementation can never be ruled out with certainty.
The choice of proof assistants in Table \ref{tab:PAs}  was made based on their prominence and their similarity to Minlog.

\begin{center}\captionsetup{hypcap=false}
\small
\begin{tabular}{|l|l|c|c|c|c|c|c|}
\hline
Object & Minlog & Agda & Isabelle & Lean &  Mizar & Naproche & Rocq \\
\hline \hline
Euclidean Algorithm & Definition \ref{Def:NatGcd}  & \cite{alexandru2025intrinsically} & \cite{eberl2025theory} & \cite{mathlibcommunity2026extended} & \cite{trybulec1993euclids} &\cite{koepke2023formalizing} & \cite{rpdcc2025numtheory}\\
\hline
Stein's Algorithm & Definition \ref{Def:PosGcd}  & -- & \cite{paulson2023verifying} & -- & -- & -- & \cite{rpdcc2025stein}\\
\hline
Bézout's Identity& Theorems \ref{Thm:NatGcdToLinComb}, \ref{Thm:PosGcdToLinComb} & \cite[13.26]{rijke2026agda} & \cite{tabacznyj2025theoryb} & \cite{mathlibcommunity2026extended} & \cite{okazaki2012extended} &\cite{koepke2023formalizing} & \cite{rpdcc2025numtheory}\\
\hline
Infinitude of Primes & Theorem \ref{Thm:PrimesToNewPrimes} & \cite[13.100]{rijke2026agda}& \cite{wenzel2025theory} & \cite{mathlibcommunity2026inifite} & \cite{chmur1991lattice} &\cite{delon2021isabelle/naproche,koepke2023formalizing} & \cite{mcdt2020library}\\
\hline
Euclid's Lemma & Lemma \ref{Lem:PrimeToIrred}  & \cite[13.25]{rijke2026agda} & \cite{tabacznyj2025theorya} & \cite{mathlibcommunity2026prime}& \cite{chmur1991lattice} &\cite{koepke2025formalizing} & \cite{rpdcc2025numtheory}\\
\hline
FTA & Theorems \ref{Thm:ExPrimeFac}, \ref{Thm:PrimeFactorisationsToPms} &\cite[13.77]{rijke2026agda} & \cite{tabacznyj2025theorya} & \cite{mathlibcommunity2026factors} & \cite{kornilowicz2004fundamental} & -- & \cite{mcdt2020library}\\
\hline
Fermat Factorisation & Theorem \ref{Thm:Fermat}& -- & -- & -- & -- & -- & --\\
\hline
\end{tabular}
\captionof{table}{Main implementations in Minlog from this article compared with other proof assistants}
\label{tab:PAs}
\captionsetup{hypcap=true}
\end{center}
\noindent
There are many other proof assistants in which these theorems and definitions have been formalised. For implementations in additional proof assistants, we refer to Wiedijk’s 100 Theorems \cite{wiedijk2025formalizing}. In particular, it provides an overview of the greatest common divisor algorithm, Bézout’s theorem, the infinitude of primes, and the fundamental theorem of arithmetic (FTA).

Based on our selection, we can already observe the following: the typical theorems and algorithms (such as the Euclidean algorithm, its extension yielding Bézout’s identity, the infinitude of primes, Euclid's lemma, and in many cases also the FTA) have been implemented in most proof assistants. In contrast, Stein’s algorithm, which is an efficient variant of the Euclidean algorithm for binary integers, has been treated much less often. In general, natural numbers are most often treated in the unary representation with zero and successor, rather than in a binary representation. This suggests that, so far, implementations have tended to focus less on the efficiency of the algorithms and more on the correctness of the statements. This is also consistent with the fact that, so far, apart from trial division in the context of the fundamental theorem of arithmetic, no other factorisation algorithm for integers has been implemented in a proof assistant. With its implementation of Fermat’s factorisation method, this article takes a first step in this direction.

Another notable aspect of Table~\ref{tab:PAs} is that most references are web links. To the best of our knowledge, there is no peer-reviewed survey that systematically covers a broad suite of elementary number-theoretic theorems and algorithms within a single proof assistant and compares implementations across systems. The most sustained peer-reviewed line explicitly focused on elementary number theory appears in the Mizar ecosystem (see \cite{pak2024elementary} in addition to the sources in Table~\ref{tab:PAs}), whereas other systems are mainly documented through isolated case studies. In this article, we not only provide an overview but also discuss the efficiency of the extracted algorithms.

\subsection{Methodology of this Article}

\paragraph{Links to the Minlog implementation.}
All sources (Minlog development, extracted Haskell code, and test files)
are available on GitHub at \url{https://github.com/FranziskusWiesnet/MinlogArith/} and archived on Zenodo \cite{wiesnet2026minlogarith}.

The repository's \texttt{README.md} explains how to use the repository. The folder
\texttt{minlog} contains a snapshot of Minlog in which the relevant files are included
and working. This snapshot is provided for archival and reproducibility purposes.
In general, we recommend using the latest Minlog version from the Minlog website
\cite{miyamoto2024minlog}. The implementations discussed here are located
in \texttt{examples/arith} within the Minlog tree. Note that, as of February 2026, these
files are only available and working on Minlog's \texttt{dev} branch. Instructions for
switching to the \texttt{dev} branch (if still required) are also provided on the Minlog
website.

For all definitions and theorems in this article, the corresponding name used in the Minlog file is provided in \texttt{typewriter font}. These notations also serve as hyperlinks (permalinks) to the exact file and line in the GitHub repository.
When viewing the article in its original PDF format, clicking on any of these typewriter-font references will take the reader directly to the corresponding location.

\paragraph{Correctness of the statement.}
Formal implementation ensures the correctness of the presented theorems. This is particularly beneficial in cases involving numerous case distinctions, lengthy calculations, or subtle conditions for applying a lemma. In such situations, even a textbook-style proof can become difficult to follow and may inadvertently contain gaps. We will return to this advantage of formalisation after presenting such complex proofs.

\paragraph{Transparency of the proofs and definitions.}
Proofs constructed within a proof assistant are completely transparent, allowing every logical step to be inspected. As is common in textbook proofs, we omit minor steps in this paper. However, if a particular part of a proof or definition is unclear, the reader can consult the corresponding formal implementation for clarification.

One of the key advantages of Minlog is that its tactic scripts closely resemble textbook-style proofs. Moreover, large proofs are annotated and indented in the implementation to improve readability and facilitate comprehension.

This article also highlights that formally conducted proofs can sometimes be significantly more intricate than one might expect based on their textbook counterparts.\footnote{At this point, it is also worth noting that some formal proofs turn out to be surprisingly trivial, particularly when the statement has been carefully formulated in advance.} This added complexity is primarily due to the need for precise definitions and explicitly structured data types in formalisation. This contrast becomes particularly salient in the context of unary versus binary representations of natural numbers. For instance, our formalisation integrates structural properties of binary numbers into the proofs -- an approach rarely taken in traditional mathematical texts, where numbers are typically treated as abstract objects.

\paragraph{Generation of an extracted program.}
Minlog provides a mechanism to extract programs from proofs, as long as the proven theorem contains computational content. These programs are initially represented as terms in the language of the underlying theory \TCF. Minlog provides a built-in command to convert these terms into executable Haskell programs. Executing the extracted Haskell programs is considerably more efficient than running the corresponding terms within Minlog itself. 

The extracted Haskell files can also be found in the folder \texttt{test-files} of the GitHub repository as \href{https://github.com/FranziskusWiesnet/MinlogArith/blob/main/test-files/gcd_pos.hs}{\texttt{gcd{\_}pos.hs}}, \href{https://github.com/FranziskusWiesnet/MinlogArith/blob/main/test-files/fta_pos.hs}{\texttt{fta{\_}pos.hs}}, and \href{https://github.com/FranziskusWiesnet/MinlogArith/blob/main/test-files/factor_pos.hs}{\texttt{factor{\_}pos.hs}}.
We will discuss the application of the Haskell files after presenting some theorems with computational content in this article. The corresponding tests and their results are available in the test-files folder of the GitHub repository, in the files \href{https://github.com/FranziskusWiesnet/MinlogArith/blob/main/test-files/gcd_pos_test.txt}{\texttt{gcd\_pos\_test.txt}}, \href{https://github.com/FranziskusWiesnet/MinlogArith/blob/main/test-files/fta_pos_test.txt}{\texttt{fta\_pos\_test.txt}}, and \href{https://github.com/FranziskusWiesnet/MinlogArith/blob/main/test-files/factor_pos_test.txt}{\texttt{factor\_pos\_test.txt}}, respectively.

\paragraph{Efficiency of proofs and extracted programs.} 
This article does not formally define the efficiency of programs, but rather discusses it in terms of their runtime behaviour. For our purposes, an efficient proof is characterised by clarity and simplicity.
Therefore, we present two formalisation approaches and compare their respective trade-offs.

The first approach uses unary natural numbers, defined via zero and the successor function. This simple definition results in relatively straightforward proofs. However, the unary representation is inherently inefficient, as its length grows at least linearly with the number’s value. As a result, the runtime of many algorithms scales at least linearly with the number of successor operations, which becomes inefficient for large inputs.

The second formalisation uses positive binary numbers, represented with three constructors: 1, the 0-successor, and the 1-successor. This representation is far more efficient, since binary numbers are encoded and processed as lists of digits (0s and 1s). Consequently, algorithm runtimes scale often polynomially with the number of the input digits, which grows logarithmically with the numeric value. However, formalising binary numbers adds systematic overhead, since one must treat the two step cases (e.g.~appending 0 and appending 1) separately throughout definitions and proofs.
Moreover, for certain operations, such as bounded search from below, the binary representation offers no advantage over the unary representation of natural numbers. Because of that, we adopt a hybrid approach, leveraging the most suitable aspects of both representations, and motivate our choice of definitions based on their efficiency in specific contexts.

\paragraph{Runtime analyses.}
In Haskell's GHCi, the command \texttt{:set +s} enables timing statistics for executed expressions. After evaluating an expression, GHCi displays the corresponding runtime and memory usage. This feature was used to measure the performance of the generated Haskell programs.
Experiments were run on a computer with an Intel Core Ultra 5 125H (Meteor Lake-H) CPU (up to 4.5 GHz), 16 GB DDR5 RAM, running Linux.
While this system is modest compared with modern high-end hardware, it is nevertheless fully sufficient for our purposes.

When a polynomial runtime is expected, we have carried out a more detailed runtime analysis. A tabular and graphical presentation of the results can be found in Appendix~\ref{Sec:Runtime}. Using the computer algebra system SymPy~\cite{meurer2017sympy}, we determined a least-squares approximating polynomial of degree at most $9$ that satisfies $f(0)=0$. If the upper bound of degree 9 did not yield a meaningful result, we reduced the degree until the result became more meaningful. The code for generating the approximating polynomials and the plots is provided in the folder \texttt{test-files} of the GitHub repository in the file \href{https://github.com/FranziskusWiesnet/MinlogArith/blob/main/test-files/poly_approx.py}{\texttt{poly\_approx.py}}.
These values should be regarded as heuristic only. Because runtimes vary considerably, statistically reliable estimates would require orders of magnitude more runs and a more systematic statistical analysis, which is beyond the scope of this paper. For background on constrained least-squares polynomial approximation we refer, for example, to \cite{bjoerck1996numerical}. 

\section{Metatheoretical Background}
\subsection{The Proof Assistant Minlog}
Minlog was developed in the 1990s by members of the logic group at the Ludwig-Maximilians-University in Munich under the direction of Helmut Schwichtenberg  \cite{schwichtenberg1993proofs,schwichtenberg2006minlog}. Installation instructions and documentation are available on the official Minlog website, hosted by Ludwig-Maximilians-University \cite{miyamoto2024minlog}. Furthermore, several introductions to Minlog are available \cite{wiesnet2017konstruktive, wiesnet2018introduction,wiesnet2024minlog}. Notable researchers who made significant contributions to Minlog include Ulrich Berger, Kenji Miyamoto, and Monika Seisenberger  \cite{berger2011minlog, miyamoto2013phd, miyamoto2013program}. One of the first constructive works implemented in Minlog focused on the greatest common divisors of integers \cite{berger1996greatest}, which is related to the topic of this article. Recently, Minlog has predominantly been employed in the field of constructive analysis \cite{berger2016logic,koepp2023lookahead, miyamoto2014program,schwichtenberg2023logic,schwichtenberg2021logic, wiesnet2022limits}. However, a broad spectrum of proofs from various domains has already been implemented in Minlog \cite{berger2006program, ishihara2016embedding,schwichtenberg2019program, schwichtenberg2016higmans, schwichtenberg2017tiered}. A common trait among these implementations is Minlog’s dual emphasis on proof verification and program extraction.

Minlog is implemented in the functional programming language Scheme and supports the extraction of Haskell programs, thanks to work by Fredrik Nordvall-Forsberg.
The correctness of a proof formalised in Minlog can be verified automatically through a built-in command; in simple cases, Minlog can even autonomously search for proofs.
Its core philosophy is to interpret proofs as programs and to work with them accordingly \cite{benl1999formal}. Minlog provides all necessary tools for the formal extraction of programs from proofs. Furthermore, proofs can be normalised in Minlog. This process involves ensuring that proofs adhere to a standard or normalised form, enhancing their clarity and facilitating further analysis \cite{berger1998normalization}.

Terms in Minlog can also be transformed and evaluated. Transformations include standard $\beta$- and $\eta$-reduction of lambda expressions, as well as conversion rules for program constants.
It is important to note that, due to the conversion rules for program constants, terms in Minlog do not necessarily have a unique normal form. In our setting, however, every term has a normal form and can be evaluated, as we work exclusively with total objects, which will be introduced below.

\subsection{Theory of Computational Functionals}
Minlog is based on an extension of $\textsf{HA}^{\omega}$ known as the \emph{Theory of Computable Functionals} (\TCF), which serves as its meta-theory. Its origins go back to Dana Scott's seminal work on \emph{Logic of Computable Functionals} \cite{scott1993type} and Platek's PhD thesis \cite{platek1966foundations}.  It also builds on fundamental ideas  introduced by Kreisel \cite{kreisel1959interpretation} and Troelstra \cite{troelstra1973metamathematical}.
More recently, significant contributions to {\TCF} have been made by the Munich group, including Helmut Schwichtenberg, Ulrich Berger, Wilfried Buchholz, Basil Kar{\'a}dais, and Iosif Petrakis \cite{berger2002refined, huber2010towards, karadais2013towards, petrakis2013advances, schwichtenberg1991primitive}.
The semantics of \TCF{} builds on partial continuous functionals and information systems, following the \emph{Scott model} \cite{larsen1991using, scott1982domains}. Here we provide only an overview of the aspects relevant to this article. For a comprehensive and formal introduction, we refer to \cite{schwichtenberg2023logic, schwichtenberg2012proofs, wiesnet2017konstruktive, wiesnet2021computational}.

The logical framework for {\TCF} is the \emph{calculus of natural deduction} \cite{pelletier1999brief, prawitz2006natural, schwichtenberg2013minimal}, where the implication $\to$ and the universal quantifier $\forall$ are the only logical symbols. Other standard logical connectives such as $\wedge, \vee, \exists$, and even the falsum $\falsum$ (at least the version we use), are introduced as inductively defined predicates. Consequently, {\TCF} is a theory within minimal logic, omitting the law of excluded middle. The ex-falso axiom is a consequence of the axioms of {\TCF} and the definition of $\falsum$.
\subsubsection{Terms and Types in \TCF}
Terms in \TCF\ are based on typed $\lambda$-calculus, meaning every term is assigned a type. Types in \TCF\ are either type variables, function types, or algebras. Algebras can be seen as fixed points of their constructors. For example, the type of \emph{natural numbers} (\href{https://github.com/FranziskusWiesnet/MinlogArith/blob/1c63c302aa52e8972b594a3ef8753fa11a0435eb/minlog/lib/pos.scm#L38}{\texttt{nat}})
is defined as an algebra with the constructors $0:\NN$ and $\Suc: \NN \to \NN$. In short notation, we express this as $$\NN := \mu_{\xi}(0: \xi,\Suc : \xi\to\xi).$$
\paragraph{Positive binary numbers (\texorpdfstring{\href{https://github.com/FranziskusWiesnet/MinlogArith/blob/1c63c302aa52e8972b594a3ef8753fa11a0435eb/minlog/lib/pos.scm\#L38}{$\mathtt{pos}$}}{pos}).} Another important algebra we are dealing with in this article is the algebra of positive binary numbers
$$\PP := \mu_{\xi}(1: \xi,\SZero : \xi\to\xi, \SOne:\xi\to\xi)$$
with three constructors: the nullary constructor 
$1$ and two successor functions $\SZero$ and $\SOne$.  The successor functions append the corresponding digit to the binary representation of a positive number; in particular, $\SZero p = 2p$ and $\SOne p = 2p +1$.  Note that the representation using constructors is the reverse of the well-known binary representation, for example $\SZero\SOne 1$ represents $6$ which is usually written as $110$ in binary representation.

The reason for starting with \(1\) instead of \(0\) is primarily to ensure that each number is uniquely determined by its constructors. If we were to start with \(0\), then for example \(0\) and $\SZero 0$ would represent the same number, violating this uniqueness. This uniqueness is important to ensure that functions on positive numbers are automatically well-defined. However, the absence of $0$ introduces additional complexity in some cases (see, for example, Theorem \ref{Thm:PosGcdToLinComb}).

The third algebra we will use is \emph{Boolean algebra} $\BB$. It is simply defined by the two nullary constructors \(\true: \BB\) and \(\false: \BB\).

These three algebras form the basic types in this article. Additionally, we will use function types over these three algebras.
Since we do not want to explicitly specify the type every time we introduce a new variable, we will assign variables to a type based on their letters. This is also done in Minlog, and we follow the conventions of Minlog here, which is why some variables consist of two letters.
\begin{notation}
The following table shows the types assigned to each variable:

\begin{center}
\begin{tabular}{l@{\hskip 2cm}l@{\hskip 2cm}l}
$w: \BB$ & $l,m,n: \NN$ & $p,q,r: \PP$ \\
$ws: \NN \to \BB$ & $wf: \PP \to \BB$ & $ps,qs,rs: \NN \to \PP$
\end{tabular}
\end{center}
If other variables are used, their types are either irrelevant or will be specified explicitly. If we need multiple variables of the same type, we will assign them indices or other decorations, i.e. $wf_2'$ also has type $\PP\to \BB$.
\end{notation}

In addition to variables and constructors, the so-called \textit{defined constants} (also referred to as \textit{program constants} or simply \textit{constants}) are also terms in \TCF.
These are defined by their type and computation rules. A simple program constant is the logarithm  $\operatorname{PosLog}:\PP \to \NN$ on the positive numbers, in Minlog \href{https://github.com/FranziskusWiesnet/MinlogArith/blob/1c63c302aa52e8972b594a3ef8753fa11a0435eb/minlog/lib/pos.scm#L5097}{\texttt{PosLog}}.
It essentially returns the number of digits in the binary representation minus one. It is given by the following three computation rules
\begin{align*}
\operatorname{PosLog}\ 1 &:= 0\\
\operatorname{PosLog}\ (\SPos p) &:= \Suc(\operatorname{PosLog}\ p),
\end{align*}
where the last line represents two rules as $\SPos$ shall be $\SZero$ and $\SOne$. The basic arithmetic operations for both \(\NN\) and \(\PP\) are also defined as program constants through computation rules.

There are also versions of the conjunction, disjunction and negation that are given as program constants on the algebra $\BB$: $\andb$, $\orb$ and $\neg$. For our purposes, these program constants behave the same as the corresponding logical connectives. However, formally, they are defined by the well-known computation rules.

\paragraph{Conversion between natural numbers and positive binary numbers.}
Two important program constants are \(\PosToNat: \PP \to \NN\) and \(\NatToPos: \NN \to \PP\). As their names suggest, they convert a positive number to a natural number and vice versa. Note that \(\NatToPos(0) = 1\). Apart from this exception, the program constants behave as expected. While their precise definitions are somewhat more involved, they are not essential for our purposes here. We therefore refer the reader to the Minlog implementation for further details.

In Minlog, the program constants \href{https://github.com/FranziskusWiesnet/MinlogArith/blob/1c63c302aa52e8972b594a3ef8753fa11a0435eb/minlog/lib/pos.scm#L668}{\texttt{PosToNat}} and \href{https://github.com/FranziskusWiesnet/MinlogArith/blob/1c63c302aa52e8972b594a3ef8753fa11a0435eb/minlog/lib/pos.scm#L705}{\texttt{NatToPos}} are essential and must always be used explicitly. In contrast, in Haskell the types \texttt{nat} and \texttt{pos} are identical. As a result, $\PosToNat$ is simply the identity function, and $\NatToPos$ is also the identity function, with the exception that \(0\) is mapped to \(1\).

In this article, we will use these two functions explicitly only when necessary for clarity; otherwise, we will leave them implicit.

As already mentioned, many theorems come in two versions. To enhance readability, we will provide only the version for positive binary numbers whenever the alternative version is equivalent for our purposes. Note that a positive number \(p\) can also be understood as a natural number greater than \(0\). For example, \(\forall_p A(p)\) can also be seen as \(\forall_n(0 < n \rightarrow A(\NatToPos\ n))\). 

\subsubsection{Formulas in TCF}
In {\TCF}, predicates are defined (co-)inductively. For the purposes of this article, only a few inductively defined predicates are important. Each inductively defined predicate $I$ is specified by introduction axioms (also called clauses) of the form
$$I^+_i:\ \forall_{\vec{x}_i}((A_{ij}(I))_{j<n_i}\to I\vec{t}_i),$$
where $(C_i)_{i<n}\to D$ is a short form for $C_0\to\dots\to C_n\to D$ and $I$ can at most occur strictly positive in each $A_i(I)$. $\vec{t}_i$ denotes a (possibly empty) list of terms filling all arguments of $I$, and $\vec{x}_i$ is the list of free variables such that $I^+_i$ has no free variables. If $\vec{x}_i$ is empty, the universal quantification is omitted.

An inductively defined predicate can be viewed as the smallest (with respect to inclusion) predicate that satisfies its clauses. This interpretation is given by the elimination axiom of an inductively defined predicate. It states that any predicate satisfying the clauses is a superset of the predicate:
$$
I^-(X) :\ \left(\forall_{\vec{x}_i}\left(\left(A_{ij}(I\cap X)\right)_{j<n_i}\to X\vec{t}_i\right)\right)_{i<k}\to \forall_{\vec{x}}(I\vec{x}\to X\vec{x})
$$
Here one even has the strengthened premise with $A_{ij}(I\cap X)$ instead of $A_{ij}(X)$.
For our cases, however, this strengthening makes no difference, so we consider only the last case.
Note that in Minlog's notation, the premise $I\vec{x}$ appears at the beginning of the formula, in particular  $(A_i)_{i<k}\to \forall_{\vec{x}}(I\vec{x}\to X\vec{x})$ becomes $\forall_{\vec{x}}(I\vec{x}\to (A_i)_{i<k}\to X\vec{x})$, where the  $A_i$s as above in $I^-(X)$. These two versions are equivalent, as all $A_i$s are closed formulas.

In this article we only need a few examples of inductively defined predicates, which we present in the following. For our purposes, understanding these examples is entirely sufficient; however, anyone interested in the general underlying theory in more detail should consult \cite[Section 7.1.2]{schwichtenberg2012proofs}.
 
\paragraph{Leibniz equality and the falsum.}
A simple example of an inductively defined predicate is the Leibniz equality $\equiv$ on any fixed type $\alpha$. It is given by the only clause $\forall_{x:\alpha}x\equiv x$. The elimination axiom is $$\operatorname{Eq}^-(X): \forall_{x,y:\alpha}(x\equiv y \to \forall_xXxx \to Xxy),$$
where $X$ can be any predicate that takes two terms of type $\alpha$ as arguments.
By using the Leibniz equality on booleans, a boolean term $w$ can be identified with the formula $w\equiv \true$. In this way, we can use Boolean terms as propositions in formulas. The falsum is in particular defined by $\falsum := \false \equiv \true$, and the ex falso principle can then be proven for all relevant formulas in this article \cite[Satz 1.4.7]{wiesnet2017konstruktive}.

A particular Boolean term is decidable equality, denoted by $=$. It is given as a Boolean valued program constant by its computation rules. For natural numbers, it is specified by the rules:
\begin{align*}
\begin{matrix}
0 &=& 0& &:=& &\true\\
Sn &=& 0& &:=& &\false\\
0 &=& Sn& &:=& &\false \\
Sn &=& Sm& &:=& &n = m
\end{matrix}
\end{align*}
Similarly, it is defined on the positive integers and the booleans. For function types, there is no decidable equality. In our case there is no difference between decidable equality and Leibniz equality, if both exist. This holds especially because we consider only total objects, which are discussed in the next paragraph.
\paragraph{Totality.}
A key feature of \TCF\ is that a closed term of an algebraic type need not
normalise to a finite constructor expression of that type.
For instance, a term $n:\NN$ need not reduce to a numeral of the form
$\Suc \dots\Suc 0$ with finitely many $\Suc$. Operationally, such a term may diverge or unfold indefinitely (e.g.\ it may keep producing successors forever).
Hence, terms in \TCF\ are partial in general: from $n:\NN$ alone we cannot
assume that $n$ denotes a \emph{total} natural number. In many cases this behaviour is even desirable, and allowing non-total objects makes the theory significantly richer; see, for example, \cite{berger2016logic, koepp2023lookahead,schwichtenberg2023logic, schwichtenberg2025logic, schwichtenberg2021logic, wiesnet2022limits}.

However, we can no longer prove statements of the form $\forall_tA(t)$ by induction or case distinction on the term $t$ as we do not know, in general, how $t$ is constructed. In order to do induction after all, we will use the totality predicate $\Total$ of \TCF. Informally speaking, $\Total_{\tau} t$ for some term $t:\tau$ means that $t$ is a finite constructor expression of $\tau$ if $\tau$ is an algebra, or $t$ maps total objects to total objects if $\tau$ is a function type. Formally, the totality is defined by recursion over the type. 

On natural numbers, $\Total_{\NN} n$ is defined by the clauses $\Total_{\NN} 0$ and $\forall_n( \Total_{\NN} n \to \Total_{\NN}(\Suc n))$. On positive binary numbers $\Total_{\PP} p$ is defined by the three clauses $\Total_{\PP} 1$ and $\forall_p( \Total_{\PP} p \to \Total_{\PP}(\SPos p))$. On the booleans $\Total_{\BB} w$ is defined by $\Total_{\BB}\true$ and $\Total_{\BB} \false$.
The elimination axiom of $\Total_{\BB} w$ is then case distinction by $w=\true $ and $w=\false$, the elimination axiom of $\Total_{\NN} n$ is induction over natural numbers, and the elimination axiom of $\Total_{\PP} p$ is similar to induction but one has to consider two successor cases, one for $\SOne$ and one for $\SZero$, i.e.,
\begin{align*}
\Total_{\PP}^-(X):\ \ \forall_p\left(\Total_{\PP} p \to X1 \to \forall_q(Xq\to X(\SZero q)) \to \forall_q(Xq\to X(\SOne q)) \to Xp\right),
\end{align*}
where $X$ is any predicate on the positive binary numbers.
For the rest of this article, we implicitly assume that each quantified variable is total, except when we write a hat on it. In particular, $\forall_nA(n)$ is an abbreviation for $\forall_{\hat{n}}(\Total_{\NN} \hat{n} \to A(\hat{n}))$, and analogously for the other versions of the totality. This abbreviation is also used in Minlog.
Usually we just write $\Total x$ instead of $\Total_{\tau}x$ when the type will be clear from the context or is not important.
Since we implicitly assume that all boolean terms are total, we can use case distinctions on them. In this sense, the law of excluded middle holds for boolean terms.

From \cite[Satz 1.5.10]{wiesnet2017konstruktive} it follows that for total objects in $\BB$, $\NN$ and $\PP$ the decidable equality implies the Leibniz equality, and by $\operatorname{Eq}^-\left(\{x,y\mid x = y\}\right)$ it follows directly that the Leibniz equality implies decidable equality on total objects of any type. Therefore, we use the simpler decidable equality.

\paragraph{Logical connectives.}
In the next section, we will briefly discuss the computational content of formulas. Therefore, in this section, we will already outline what we expect as the computational content of the logical connectives and quantifiers.

For any two formulas $A$ and $B$, their conjunction is given as an inductively defined predicate with the single clause
\begin{align*}
\wedge^+&: A\to B \to A\wedge B
\end{align*}
Its elimination axiom is then given by
\begin{align*}
\wedge^-&: A\wedge B \to (A\to B\to C)\to C
\end{align*}
for any formula $C$. Note, that in this case the predicate $A\wedge B$ does not take any terms as arguments and is therefore also a formula.
For the computational content, we expect the computational content of $A$ and the computational content of $B$, if $A$ and $B$ have computational content.

The disjunction is given by two introduction rules and an elimination axiom, as follows:
\begin{align*}
\vee^+_0&: A\to A\vee B, \qquad \vee^+_1: B\to A\vee B,\\
\vee^-&: A\vee B \to (A\to C) \to (B\to C) \to C
\end{align*}
As computational content we expect a marker indicating which side holds and, whenever that side has computational content, the corresponding data.

The existential quantifier is also inductively defined by one clause with its introduction and elimination axiom as follows:
\begin{align*}
\exists^+&:\forall_x(A(x)\to \exists_xA(x))\\
\exists^-&:\exists_xA(x)\to \forall_x(A(x)\to C)\to C
\end{align*}
Note that $x$ must not be a free variable in $C$.
For the computational content, note that $\forall_x(A(x)\to C)$ is short for $\forall_{\hat x}(\Total \hat x\to A(\hat x)\to C)$. Hence, $\exists_xA(x)$ can be seen as $\exists_{\hat x}(\Total \hat x \wedge A(\hat x))$.
Therefore, the computational content of  $\exists_xA(x)$ is the computational content of $\Total x$ and $A(x)$, where the computational content of $\Total x$ can be identified with $x$ itself. Hence, the computational content of $\exists_xA(x)$ can be seen as a term $t$ together with the computational content of $A(t)$, if $A(t)$ has computational content; if $A(t)$ has no computational content, the computational content of $\exists_xA(x)$ is just the term $t$.

\subsubsection{Program Extraction from Proofs} 
\label{Sec:ProgExtr}
In this section we give an overview of the process of program extraction from proofs in \TCF. For formal definitions and proofs we refer to \cite{schwichtenberg2023computational,schwichtenberg2012proofs,wiesnet2017konstruktive}. In this short section we do not give formal definitions as they are quite
complex and we will use the proof assistant Minlog in any case to carry out the program
extraction.

\paragraph{Computationally relevant formulas.}
The computational content arises from (co-)inductively defined predicates. When defining such a predicate or a predicate variable, one must specify whether it is computationally relevant (cr) or non-computational (nc). For instance, Leibniz equality is defined as non-computational. Thus, Boolean terms as formulas carry no computational content.

 The standard version of the totality $\Total$ is cr, while the alternative version, denoted by \(\Total^{nc}\), is nc. The same distinction applies to the logical connectives \(\wedge\), \(\vee\), and \(\exists\), whose constructive content was discussed in the previous section. Correspondingly, there are also nc versions \(\wedge^{nc}\), \(\vee^{nc}\), and \(\exists^{nc}\), which do not carry any computational content.

A formula $C$ is computationally relevant if its final conclusion $\operatorname{FC}(C)$ is of the form \(A\vec{t}\), where \(A\) is a computationally relevant predicate. The final conclusion is defined by the rules $\operatorname{FC}(A\to B) := \operatorname{FC}(B)$, $\operatorname{FC}(\forall_{\hat x}A):=\operatorname{FC}(A)$ and $\operatorname{FC}(I\vec{x}) = I\vec{x}$, where $I$ is an (co-)inductively defined predicate.

The key difference between a non-computational and a computational inductively defined predicate lies in the competitor predicate: for the non-computational case, the competitor must also be non-computational.

It is important to note that universal and existential quantifiers, by themselves, do not inherently carry computational content. However, when used in conjunction with the totality predicate and the previously introduced abbreviations, they do.
Therefore, the formula \(\exl_x A(x)\) carries at least the computational information of \(x\), since it is an abbreviation for 
$
\exists_{\hat{x}} \bigl( \Total \hat{x} \wedge A(\hat{x})\bigr).
$
The computational content of this expression is a pair consisting of the content of \(\Total \hat{x}\) and that of \(A(\hat{x})\). In our case, the computational content of \(\Total x\) can be identified with \(x\) itself. Hence, the computational content of \(\exl_x A(x)\) provides at least a term \(t\) such that \(A(t)\) holds. Similarly, the computational content of \(\forall_x A(x)\) is a function that takes an arbitrary term \(t\) of the same type as \(x\) and returns the computational content of \(A(t)\).
If we explicitly wish to suppress computational content, we usually use the nc existence quantifier \(\exists^{nc}\) and write $\forall^{nc}_xA(x)$ for $\forall_{\hat x}(\Total^{nc}\hat x\to A(\hat x))$.
\paragraph{Procedure of formal program extraction.}
In a first step, given a computationally relevant formula \(A\), one defines its \emph{type} \(\tau(A)\) and a corresponding \emph{realiser predicate} \(\realizer{A}\). Formally, this is done by structural recursion on the formula \(A\).
The realiser predicate is a predicate that takes a term of type \(\tau(A)\) and asserts that this term is a realiser of the formula, i.e., it satisfies the computational requirements specified by \(A\).

In a second step, the \emph{extracted term} \(\et{M}\) of the formalised proof \(M\) of \(A\) is computed. This extracted term is a term in \TCF, has the type \(\tau(A)\), and is defined by structural recursion over the proof \(M\). It represents the algorithm extracted from the formal proof.
In our case, we will discuss this term after presenting an interesting proof formalised in Minlog, typically expressed as a Haskell term.

In the final step of program extraction, a proof is constructed showing that the extracted term is indeed a realiser of the realiser predicate, i.e., \(\realizer{A} \et{M}\). This is known as the \emph{soundness theorem}, and it holds for all computationally relevant formulas \(A\) and proofs \(M\) that do not themselves contain the realiser predicate. Since the realiser predicate is designed specifically to evaluate computational content, this is no real restriction to any usual proof in normal mathematics.
For our purposes, it suffices to know that such a proof exists. Thus, the extracted term is correct and behaves as expected. 

\section{Basic Definitions}
In this section, we cover fundamental definitions that are not directly related to number theory but are very general. In particular, we address the different definitions for natural numbers and for positive binary numbers.
\subsection{Bounded Existence Quantifier}
An important Boolean-valued program constant for this article is the bounded existence quantifier. On the natural numbers, it is defined as follows:

\begin{definition}[\texorpdfstring{\href{https://github.com/FranziskusWiesnet/MinlogArith/blob/1c63c302aa52e8972b594a3ef8753fa11a0435eb/minlog/lib/nat.scm\#L3743}{\texttt{ExBNat}}}{ExBNat}]
For a sequence of booleans $ws:\NN\to \BB$, we define the bounded existence quantifier on natural numbers by the following rules:\footnote{\label{Footnote}Note that, instead of using a case distinction in the definition, one could also use the Boolean disjunction (i.e.~$\exbn{\Suc n} ws :=(ws\ n)  \orb\exbn{n} ws$). While this makes no difference on paper, Minlog uses a case distinction to ensure that the condition term is evaluated first, and only the corresponding branch is computed based on its value. This improves efficiency compared to a Boolean disjunction, as it avoids evaluating both branches unnecessarily.}
\begin{align*}
\exbn{0} ws &:= \false \\
 \exbn{\Suc n} ws &:= \begin{cases}\true & \text{ if}\quad ws\ n\\ \exbn{n} ws & \text{ otherwise.}  \end{cases}
\end{align*}
As an abbreviation we use the notation $\exbn{n}_i (ws\ i) := \exbn{n} ws$. Formally, the bounded existence quantifier on natural numbers is a defined constant $\exists^< :\NN \to (\NN \to \BB) \to \BB$.
\end{definition}
\begin{lemma} \label{Lem:ExBNatProp}
The universal closures of the following statements hold:
\begin{equation*}
\begin{array}{ll}
\href{https://github.com/FranziskusWiesnet/MinlogArith/blob/1c63c302aa52e8972b594a3ef8753fa11a0435eb/minlog/lib/nat.scm#L3770}{\mathtt{ExBNatIntro}}:
& m<n \to ws\ m \to \exbn{n}_i (ws\ i)\\
\href{https://github.com/FranziskusWiesnet/MinlogArith/blob/1c63c302aa52e8972b594a3ef8753fa11a0435eb/minlog/lib/nat.scm#L3886}{\mathtt{ExBNatElim}}:
& \exbn{n}_i (ws\ i) \to \forall_{m<n}^{nc}(ws\ m \to X) \to X\\
\href{https://github.com/FranziskusWiesnet/MinlogArith/blob/1c63c302aa52e8972b594a3ef8753fa11a0435eb/minlog/lib/nat.scm#L3801}{\mathtt{ExBNatToExNc}}:
& \exbn{n}_i (ws\ i) \to \exnc_{m<n}(ws\ m)
\end{array}
\end{equation*}
Here, $X$ can be an arbitrary formula, as long as it does not lead to a collision of free variables. 
\end{lemma}
\begin{proof}
The first two statements follow by induction on $n$ and the last statement is a special case of the second statement. For details, we refer to the corresponding statement in the Minlog library.
\end{proof}
Due to the nature of positive binary numbers, the bounded existential quantifier is defined somewhat differently in this context. In particular, it is easier here to test up to and including the upper bound. Therefore, $\leq$ becomes part of the notation. This allows for an easy distinction between the notation for natural numbers and positive binary numbers.
\begin{definition}[\texorpdfstring{\href{https://github.com/FranziskusWiesnet/MinlogArith/blob/1c63c302aa52e8972b594a3ef8753fa11a0435eb/minlog/lib/pos.scm\#L6943}{\texttt{ExBPos}}}{ExBPos}]
For a given function $wf:\PP\to \BB$, we define the bounded existence quantifier $\exists^{\leq}$ on positive binary numbers by the following rules:\footref{Footnote}
\begin{align*}
\exbp{1}wf &:= wf\ 1\\
\exbp{\SZero p}wf &:= \begin{cases}
\true &\text{ if }\exbp{p}wf\\
\exbp{p}\lambda_i(wf(p+i)) &\text{ otherwise,}
\end{cases}\\
\exbp{\SOne p}wf &:= \begin{cases}
\true &\text{ if }wf(\SOne p)\\
\exbp{\SZero p}wf &\text{ otherwise.}
\end{cases}
\end{align*}
We use the notation $\exbp{n}_i (ws\ i) := \exbp{n} ws$. Formally, $\exists^{\leq}: \PP \to (\PP \to \BB) \to \BB$ is a defined constant.
\end{definition}
\begin{lemma} \label{Lem:ExBPosProp}
The universal closures of the following statements hold:
\begin{equation*}
\begin{array}{ll}
\href{https://github.com/FranziskusWiesnet/MinlogArith/blob/1c63c302aa52e8972b594a3ef8753fa11a0435eb/minlog/lib/pos.scm#L6997}{\mathtt{ExBPosIntro}}: & p\leq q \to wf\ p \to \exbp{q}_i (wf\ i)\\
\href{https://github.com/FranziskusWiesnet/MinlogArith/blob/1c63c302aa52e8972b594a3ef8753fa11a0435eb/minlog/lib/pos.scm#L7085}{\mathtt{ExBPosElim}}: & \exbp{p}_i (wf\ i) \to \forall_{q\leq p}^{nc}(wf\ q \to X) \to X\\
\href{https://github.com/FranziskusWiesnet/MinlogArith/blob/1c63c302aa52e8972b594a3ef8753fa11a0435eb/minlog/lib/pos.scm#L7172}{\mathtt{ExBPosToExNc}}: & \exbp{p}_i (wf\ i) \to \exnc_{q\leq p}(wf\ q)
\end{array}
\end{equation*}
Here, $X$ can be any formula, provided it does not cause a collision of free variables.
\end{lemma}
\begin{proof}
The first statement follows by induction on $q$, the second statement follows by induction on $p$, and the last statement is a special case of the second one. For details, we refer to the corresponding Minlog files.

\end{proof}
\subsection{Least Number and Monotone Maximum}
Throughout this article, we frequently need to compute a number satisfying a certain property, where it already exists in a weak sense and is bounded. In this case, a bounded search can be applied. It is defined as follows:
\begin{definition}[\texorpdfstring{\href{https://github.com/FranziskusWiesnet/MinlogArith/blob/1c63c302aa52e8972b594a3ef8753fa11a0435eb/minlog/lib/nat.scm\#L1949}{\texttt{NatLeast}}}{NatLeast}, \texorpdfstring{\href{https://github.com/FranziskusWiesnet/MinlogArith/blob/1c63c302aa52e8972b594a3ef8753fa11a0435eb/minlog/lib/nat.scm\#L1988}{\texttt{NatLeastUp}}}{NatLeastUp}]
For a sequence of booleans $ws:\NN\to \BB$, the bounded $\mu$-operator is given by the following rules:
\begin{align*}
\mu_{<0} ws &:= 0\\
\mu_{<\Suc n} ws &:= \begin{cases} 
0 &\text { if}\quad ws\ 0\\
\Suc \left(\mu_{<n}\left(\lambda_m\left(ws \left(\Suc m\right)\right)\right)\right)&\text{ otherwise.}
\end{cases}
\end{align*}
We use the notations $\mu_{i<n} (ws\ i) := \mu_{<n} ws$ and define the general $\mu$-operator by
\begin{align*}
\mu_{m\leq i<n} (ws\ i):=\begin{cases}\left(\mu_{i<n-m}(ws(i+m))\right) + m & \text{ if } m\leq n\\ 0 &\text{ otherwise.}\end{cases}
\end{align*}
\end{definition}
As can be seen directly from the definition, $\mu_{<n}(ws)$ indeed returns the smallest natural number satisfying the property $ws$, if such a number exists and is less than $n$; otherwise, it returns $n$.

Since each individual case is checked exhaustively until a corresponding 
$i$ is found or the bound is reached, the computational use of the bounded existential quantifier or the $\mu$-operator is associated with a high runtime. We will see that the bounded existential quantifier appears in theorem statements but does not contribute to the computational content. Therefore, the $\mu$-operator constitutes the main bottleneck in terms of computational complexity. However, it should be noted that the smaller the corresponding number 
$i$ for which $ps(i)$ holds, the more efficient the $\mu$-operator $\mu_{<n}ps$ becomes.

The following lemma shows characteristic properties of the $\mu$-operator.
Since it is obvious that
\begin{align*}
\forall_{n,ws}(\mu_{i<n}(ws\ i)=\mu_{0\leq i <n}(ws\ i) )\qquad 
(\href{https://github.com/FranziskusWiesnet/MinlogArith/blob/1c63c302aa52e8972b594a3ef8753fa11a0435eb/minlog/lib/nat.scm\#L5710}{\mathtt{NatLeastUpZero}}),
\end{align*}
the lemma is formulated only for the general $\mu$-operator if there are two versions, and can simply be specialised to the normal $\mu$-operator. The proofs are straightforward and are mainly done by induction. For details, we refer to the corresponding Minlog implementation and do not give them here.
\begin{lemma}
\label{Lem:LeastNatProp}
The universal closures of the following statements hold:
\begin{equation*}
\begin{array}{ll}
\href{https://github.com/FranziskusWiesnet/MinlogArith/blob/1c63c302aa52e8972b594a3ef8753fa11a0435eb/minlog/lib/nat.scm#L5676}{\mathtt{NatLeastUpLtElim}}: & n_0\leq \mu_{n_0\leq i < n_1}(ws\ i) \to\\ & \mu_{n_0\leq i < n_1}(ws\ i)< n_1 \to ws(\mu_{n_0\leq i < n_1}(ws\ i))\vspace{1mm}\\
\href{https://github.com/FranziskusWiesnet/MinlogArith/blob/1c63c302aa52e8972b594a3ef8753fa11a0435eb/minlog/lib/nat.scm#L5646}{\mathtt{NatLeastUpLeIntro}}: & n_0\leq m \to ws\ m \to \mu_{n_0\leq i < n_1}(ws\ i)\leq m\\
\href{https://github.com/FranziskusWiesnet/MinlogArith/blob/1c63c302aa52e8972b594a3ef8753fa11a0435eb/minlog/lib/nat.scm#L5635}{\mathtt{NatLeastUpLBound}}: & n_0 \leq n_1 \to n_0 \leq \mu_{n_0\leq i < n_1}(ws\ i)\\
\href{https://github.com/FranziskusWiesnet/MinlogArith/blob/1c63c302aa52e8972b594a3ef8753fa11a0435eb/minlog/lib/nat.scm#L5615}{\mathtt{NatLeastUpBound}}: & \mu_{n_0\leq i < n_1}(ws\ i)\leq n_1\\
\href{https://github.com/FranziskusWiesnet/MinlogArith/blob/1c63c302aa52e8972b594a3ef8753fa11a0435eb/minlog/lib/nat.scm#L4175}{\mathtt{PropNatLeast}}: & m\leq n \to ws\ m \to ws(\mu_{i<n}(ws\ i) )
\end{array}
\end{equation*}
\end{lemma}

The $\mu$-operator allows us to turn the bounded existential quantifier into a computational existential quantifier. Initially, the bounded existential quantifier yields only a non-computational existence, as shown in Lemma~\ref{Lem:ExBNatProp} (\texttt{ExBNatToExNc}). In the following lemma, we show how to obtain computational existence using the $\mu$-operator. We present this proof in particular to demonstrate the different treatment of computational and non-computational terms.

\begin{lemma}
\label{Lem:ExBPosToEx}
\begin{align*}
\forall_{ws}&(\exbn{n}_l(ws\ l) \to \exists_l (ws\ l))
\end{align*}
\end{lemma}

\begin{proof}
Assume $\exbn{n}_l(ws\,l)$ for a given $ws$ (with $\Total(ws)$ implicitly and computationally relevant), and apply Lemma~\ref{Lem:ExBNatProp} (\texttt{ExBNatElim}) with $X := \exl_l(ws\,l)$.
It remains to show
\[
\forall^{nc}_{l<n}\bigl(ws\,l \to X\bigr).
\]
Let $l_0$  be given non-computationally such that $l_0<n$ and $ws\,l_0$ hold. We need to show
$\exists_l(ws\,l)$. For this, we define
\[
l := \mu_{<n}(ws),
\]
which can be used computationally as $ws$ is computationally relevant.
By \texttt{PropNatLeast} from Lemma~\ref{Lem:LeastNatProp}, we obtain
\[
\forall_m\bigl(m \le n \to ws\,m \to ws\,l\bigr).
\]
Note that the final conclusion $ws\,l$, and hence the whole formula, is
non-computational, as $ws\,l$ is formally a Leibniz equality. Therefore, we may instantiate $m$ with $l_0$.
Since $l_0<n$ implies $l_0 \le n$ and we have $ws\,l_0$, we conclude $ws\,l$ and
thus $\exists_l(ws\,l)$. 
\end{proof}
As one can easily see directly from the proof, the computational content of the proof above is a function that takes a sequence $ws$ and returns $\mu_{<n}(ws)$. Thus, the $\mu$-operator directly yields a computational meaning even when the existence is initially non-computational. However, at least $ws$ must be given computationally.\smallskip

A similar definition of $\mu$ is also possible for positive binary numbers. However, this would not take into account the efficient representation of positive binary numbers. Moreover, the above definition can also be used for positive binary numbers by converting them to natural numbers, using the $\mu$-operator and then converting the result back to positive binary numbers.

However, the representation of positive binary numbers is particularly well suited for creating an interval nesting. For this, we need a Boolean valued function $wf:\PP\to \BB$ on positive binary numbers that is true for small numbers, i.e.~$wf(1)$, and false for large numbers, i.e.~$ \neg wf(2^n)$ for some given $n:\NN$, and once it is false, it remains false. We express this by $p<q \to wf(q)\to wf(p)$. Then, we can efficiently determine the largest positive binary number $p$ with $wf(p)$ using an interval nesting approach:

\begin{definition}[\texorpdfstring{\href{https://github.com/FranziskusWiesnet/MinlogArith/blob/1c63c302aa52e8972b594a3ef8753fa11a0435eb/minlog/lib/pos.scm\#L7572}{\texttt{PosMonMax}}}{PosMonMax}]
For a Boolean valued function $wf: \PP \to \BB$, we define the function $\nu_{aux}(wf,\cdot,\cdot):\NN \to \PP \to \PP$ by the rules
\begin{align*}
\nu_{aux}(wf,0,q) &:= q,\\
\nu_{aux}(wf,\Suc n,1) &:=
\begin{cases}
\nu_{aux}(wf,n, 2^n)& \text{ if } wf(2^n)\\
\nu_{aux}(wf,n, 1)& \text{ otherwise,}
\end{cases}\\
\nu_{aux}(wf,\Suc n,\SPos q) &:=
\begin{cases}
\nu_{aux}(wf,n, \SPos q + 2^n)& \text{ if } wf(\SPos q + 2^n)\\
\nu_{aux}(wf,n, \SPos q)& \text{ otherwise.}
\end{cases}
\end{align*}
The (monotone) maximum $\nu:(\PP\to\BB)\to \NN\to \PP$ is then defined by
\begin{align*}
\nu(wf,n) := \nu_{aux}(wf,n,1).
\end{align*}
\end{definition}
Note that the case distinction in $\nu_{aux}$ based on the third argument is necessary because we do not have $0$ available: the third argument serves as an accumulator, which should initially be set to 0 and then incremented or left unchanged depending on $wf$. However, since we do not have 0 available, we initialise the accumulator with 1 and must explicitly handle the case where it is 1 separately.

Using the definition one can then prove the following statements about $\nu$. For a detailed proof, we refer to the corresponding Minlog implementation.
\begin{lemma}[\texorpdfstring{\href{https://github.com/FranziskusWiesnet/MinlogArith/blob/1c63c302aa52e8972b594a3ef8753fa11a0435eb/minlog/lib/pos.scm\#L7586}{\texttt{PosMonMaxProp}}}{PosMonMaxProp},  \texorpdfstring{\href{https://github.com/FranziskusWiesnet/MinlogArith/blob/1c63c302aa52e8972b594a3ef8753fa11a0435eb/minlog/lib/pos.scm\#L7608}{\texttt{PosMonMaxNegProp}}}{PosMonMaxNegProp}, \texorpdfstring{\href{https://github.com/FranziskusWiesnet/MinlogArith/blob/1c63c302aa52e8972b594a3ef8753fa11a0435eb/minlog/lib/pos.scm\#L7622}{\texttt{PosMonMaxChar}}}{PosMonMaxChar}]
\label{Lem:NuProp}
For given $wf:\PP\to \BB$ and $n:\NN$ with $wf(1)$ and $\neg wf(2^n)$ we assume that $wf$ is monotone in the sense that $wf(q)\to wf(p)$ for all $p<q$.

Then, $wf(\nu(wf,n))$ and $\neg wf(q)$ for all $q>\nu(wf,n)$.
Furthermore, these two properties are characteristic for $\nu$, i.e.:
\begin{align*}
\forall_p\left(wf(p) \to \forall_q\left(p<q\to \neg wf(q)\right) \to p=\nu(wf,n)\right).
\end{align*}
\end{lemma}
\subsection{Square root}
In this section, we provide the definition of the integer square root function on the natural numbers and the positive binary numbers. We will need it later for divisibility and Fermat factorisation. An efficient definition is very important to keep the runtime as low as possible. Whereas the definition on natural numbers uses the $\mu$-operator and is therefore straightforward but not very efficient, the definition on positive binary numbers uses the $\nu$-operator and is thus particularly efficient.
\begin{definition}[\texorpdfstring{\href{https://github.com/FranziskusWiesnet/MinlogArith/blob/1c63c302aa52e8972b594a3ef8753fa11a0435eb/minlog/lib/nat.scm\#L8937}{\texttt{NatSqrt}}}{NatSqrt}]
\begin{align*}
\lceil \sqrt{n} \rceil := \mu_{<\Suc n}(\lambda_m.\ n\leq m\cdot m)
\end{align*}
\end{definition}

This definition actually corresponds to the rounded-up square root, which justifies the notation. This can also be seen by the properties of the following lemma:

\begin{lemma}
\label{Lem:NatSqrtProp}
The universal closures of the following statements hold:
\begin{equation*}
\begin{array}{ll}
\href{https://github.com/FranziskusWiesnet/MinlogArith/blob/1c63c302aa52e8972b594a3ef8753fa11a0435eb/minlog/lib/nat.scm#L8949}{\mathtt{NatLtSqrtToSquareLt}}: & m<\lceil \sqrt{n} \rceil \to m\cdot m < n\\
\href{https://github.com/FranziskusWiesnet/MinlogArith/blob/1c63c302aa52e8972b594a3ef8753fa11a0435eb/minlog/lib/nat.scm#L8963}{\mathtt{NatSqrtLeToLeSquare}}: & \lceil \sqrt{n} \rceil \leq m \to n\leq m\cdot m\\
\href{https://github.com/FranziskusWiesnet/MinlogArith/blob/1c63c302aa52e8972b594a3ef8753fa11a0435eb/minlog/lib/nat.scm#L9029}{\mathtt{NatSqrtBound}}: & \lceil \sqrt{n} \rceil \leq n
\end{array}
\end{equation*}
\end{lemma}
\begin{proof} These properties follow directly by the properties of the $\mu$-operator in Lemma \ref{Lem:LeastNatProp}. For the exact proofs, we refer to the corresponding Minlog implementation.\end{proof}

While the definition of the square root on the natural numbers itself is very simple, it suffers from extremely poor runtime due to the use of the $\mu$-operator. 
By using positive binary numbers and their interval nesting, as introduced in the previous section, we can significantly improve the efficiency as follows.

\begin{definition}[\texorpdfstring{\href{https://github.com/FranziskusWiesnet/MinlogArith/blob/1c63c302aa52e8972b594a3ef8753fa11a0435eb/minlog/lib/pos.scm\#L7648}{\texttt{PosSqrt}}}{PosSqrt}]
\begin{align*}
\lfloor\sqrt{p} \rfloor :=  \nu \left(\lambda_q(q\cdot q\leq p), \left\lfloor \frac{\texttt{PosLog}\ p}{2}\right\rfloor + 1 \right)
\end{align*}
\end{definition}

Note that here, we actually obtain the rounded-down square root, which is the case due to the properties of the $\nu$-operator. Therefore, the notation $\lfloor\sqrt{p} \rfloor$ is justified.
Tests for the square root algorithm are documented in \texttt{factor\_pos\_test.txt} and summarised in Appendix~\ref{Sec:AppSqrt}. Given an input with $n$ digits, the $\nu$-operator performs exactly $\lfloor n/2 \rfloor+1$ steps. In each step, two numbers with at most $\lfloor n/2 \rfloor +1$ digits are multiplied and compared with the input number. The runtime of multiplication is at most quadratic in the number of digits, and comparing two numbers is linear in the number of digits. Hence, the runtime of the square root procedure is (roughly) cubic in the number of digits of the input. A cubic runtime is also confirmed by the measured data. However, multiplication in Haskell can be subquadratic for large integers \cite{fsf2020gnu}, so the cubic bound for our algorithm is only an upper bound. 

The following lemma shows the two standard properties of the rounded-down square root and states that these properties are characteristic of it:

\begin{lemma}
\label{Lem:PosSqrtChar}
The universal closures of the following statements hold:
\begin{equation*}
\begin{array}{ll}
\href{https://github.com/FranziskusWiesnet/MinlogArith/blob/1c63c302aa52e8972b594a3ef8753fa11a0435eb/minlog/lib/pos.scm#L7715}{\mathtt{PosSquareSqrtUpBound}}: & \lfloor \sqrt{p} \rfloor \cdot \lfloor \sqrt{p} \rfloor \leq p  \\
\href{https://github.com/FranziskusWiesnet/MinlogArith/blob/1c63c302aa52e8972b594a3ef8753fa11a0435eb/minlog/lib/pos.scm#L7725}{\mathtt{PosSquareSqrtLowBound}}: & \lfloor \sqrt{p} \rfloor< q \to p < q\cdot q\\
\href{https://github.com/FranziskusWiesnet/MinlogArith/blob/1c63c302aa52e8972b594a3ef8753fa11a0435eb/minlog/lib/pos.scm#L7766}{\mathtt{PosSqrtChar}}: & q\cdot q \leq p \to \forall_r(q<r \to p< r \cdot r)\to q = \lfloor\sqrt p \rfloor
\end{array}
\end{equation*}
\end{lemma}
\begin{proof}
The properties follow directly by the properties of the $\nu$-operator in Lemma \ref{Lem:NuProp}. We again refer to the Minlog code for details. 
\end{proof}
As we have seen, the above definition of the square root is already quite fast, since it is based on interval bisection. However, in the special case of the square root, the definition can be reformulated slightly so that the runtime becomes more efficient. The following algorithm is adapted from \cite{trostle2014fast}\footnote{Thanks to the anonymous reviewer who pointed me to this work and to the algorithm.} and has already been implemented in Nuprl. 
\begin{definition}[\texorpdfstring{\href{https://github.com/FranziskusWiesnet/MinlogArith/blob/1c63c302aa52e8972b594a3ef8753fa11a0435eb/minlog/examples/arith/factor_pos.scm\#L17}{\texttt{FastSqrt}}}{FastSqrt}]
\label{Def:FastSqrt}
We define the program constant $\mathtt{FastSqrt}:\PP\to \PP$ by the computation rules
\begin{align*}
\mathtt{FastSqrt}\ 1 &:= 1\\
\mathtt{FastSqrt}\ 2 &:= 1\\
\mathtt{FastSqrt}\ 3 &:= 1\\
\mathtt{FastSqrt}\ \Suc_{\alpha} \Suc_{\beta} &:= \lambda_q\left(
\begin{cases} 
\SOne q \quad\text{ if } \left(\Suc_{\alpha} \Suc_{\beta} p\right) < (\SOne q)\cdot (\SOne q)\\
\SZero q \quad\text{ else} 
\end{cases}
\right)(\mathtt{FastSqrt}\ p)
\end{align*}
where $\alpha,\beta\in\{0,1\}$, and therefore last line formally represents four rules.
\end{definition}
The above definitions of \texttt{FastSqrt} and \texttt{PosSqrt} have in common that they require at most $\left\lfloor(\texttt{PosLog}\ p)/2\right\rfloor + 1$ steps to compute, and that each step involves one multiplication and one comparison.
However, the difference is that in \texttt{FastSqrt} the number of digits of the argument is reduced by two at each step, and the number of digits of the result is reduced by one. Hence, \texttt{FastSqrt} and \texttt{PosSqrt} exhibit the same polynomial behaviour, but \texttt{FastSqrt} is asymptotically faster than \texttt{PosSqrt}.

That said, \texttt{FastSqrt} has a problem on large inputs: it starts from the input argument, shifts it by two binary digits, and first needs the result of \texttt{FastSqrt} on the shifted argument, and so on. As a consequence, all steps of \texttt{FastSqrt} must be fully unfolded before the first step can be completed. For large numbers, this leads to enormous memory allocation. While \texttt{PosSqrt} produced a result for an input of 200\,000 decimal digits without any issues (albeit after a long runtime), computations using  \texttt{FastSqrt} were consistently killed by the kernel already for inputs of 100\,000 decimal digits. Even for smaller inputs, the runtimes of \texttt{FastSqrt} are fluctuating. For inputs between 70\,000 and 90\,000 decimal digits, the process was sometimes killed before completion, whereas in other runs (for instance after restarting the PC) we did obtain a result. In \texttt{factor\_pos\_test.txt} we documented several tests, and in Appendix A.2 we summarised the runtimes up to 50\,000 digits in steps of 5\,000 in tabular form. A polynomial fit up to degree 9 does not yield a meaningful result. A fit up to degree 6 gives a quadratic approximation. However, the plot shows that the values fluctuate strongly. A reliable polynomial approximation is therefore not possible. Nevertheless, the data suggest that, asymptotically, \texttt{FastSqrt} is almost twice as fast as \texttt{PosSqrt}.

The following lemma provides important properties of FastSqrt, in particular the equality of \texttt{FastSqrt}  and \texttt{PosSqrt}.
\begin{lemma}
The universal closures of the following statements hold:
\begin{equation*}
\begin{array}{ll}
\href{https://github.com/FranziskusWiesnet/MinlogArith/blob/1c63c302aa52e8972b594a3ef8753fa11a0435eb/minlog/examples/arith/factor_pos.scm#L113}{\mathtt{PosSquareFastSqrtUpBound}}: & (\mathtt{FastSqrt}\ p) \cdot (\mathtt{FastSqrt}\ p) \leq p  \\
\href{https://github.com/FranziskusWiesnet/MinlogArith/blob/1c63c302aa52e8972b594a3ef8753fa11a0435eb/minlog/examples/arith/factor_pos.scm#L179}{\mathtt{PosSquareFastSqrtLowBound}}: & \mathtt{FastSqrt}\ p< q \to p < q\cdot q\\
\href{https://github.com/FranziskusWiesnet/MinlogArith/blob/1c63c302aa52e8972b594a3ef8753fa11a0435eb/minlog/examples/arith/factor_pos.scm#L378}{\mathtt{FastSqrtEqPosSqrt}}: & \mathtt{FastSqrt}\ p = \lfloor \sqrt{p} \rfloor
\end{array}
\end{equation*}
\end{lemma}
\begin{proof}
For \texttt{FastSquareSqrtUpBound}, one proves 
\begin{align*}
(\mathtt{FastSqrt}\  p) \cdot (\mathtt{FastSqrt}\  p) \leq \phantom{\SZero } p&\quad \wedge \\
(\mathtt{FastSqrt}\ \SZero p) \cdot (\mathtt{FastSqrt}\ \SZero p) \leq \SZero p&\quad \wedge \\
 (\mathtt{FastSqrt}\ \SOne p) \cdot (\mathtt{FastSqrt}\ \SOne p) \leq \SOne p & 
\end{align*}
by induction on $p$ and by a case distinction according to the definition of $\mathtt{FastSqrt}$. The statement \texttt{FastSquareSqrtLowBound} is proved by a similar idea.  \texttt{FastSqrtEqPosSqrt} follows then directly from \texttt{PosSqrtChar} of Lemma \ref{Lem:PosSqrtChar}. For details, we refer to the Minlog file. 
\end{proof}
Since \texttt{PosSqrt} produces more consistent results, we will mainly use it whenever it is not runtime critical. However, in Section \ref{Sec:Fermat} we will apply the square-root function iteratively, so we will rely on \texttt{FastSqrt} there.

\section{Divisibility and Greatest Common Divisor}
\subsection{Divisibility on Natural Numbers}
We will define the divisibility relation on natural numbers directly using the bounded existential quantifier. We will use the same notation for positive binary numbers and natural numbers. The context will make it clear whether we are referring to divisibility on natural numbers or on positive binary numbers. Sometimes, we will consider statements for both cases, allowing us to cover both with this notation.

\begin{definition}[\texorpdfstring{\href{https://github.com/FranziskusWiesnet/MinlogArith/blob/1c63c302aa52e8972b594a3ef8753fa11a0435eb/minlog/examples/arith/gcd_nat.scm\#L11}{\texttt{NatDiv}}}{NatDiv}]
\begin{align*}
n\mid m \quad :=\quad \exbn{\Suc m}_l l\cdot n=m
\end{align*}
\end{definition}

The following lemma shows that the divisibility relation has the usual properties of a factorisation into a product.
\begin{lemma}
\label{Lem:NatProdToDiv}
The universal closures of the following statements hold:
\begin{equation*}
\begin{array}{ll}
\href{https://github.com/FranziskusWiesnet/MinlogArith/blob/1c63c302aa52e8972b594a3ef8753fa11a0435eb/minlog/examples/arith/gcd_nat.scm#L24}{\mathtt{NatProdToDiv}}: & l\cdot m=n \to m\mid n\\
\href{https://github.com/FranziskusWiesnet/MinlogArith/blob/1c63c302aa52e8972b594a3ef8753fa11a0435eb/minlog/examples/arith/gcd_nat.scm#L63}{\mathtt{NatDivToProdNc}}: & m \mid n \to \exnc_l\ l\cdot m = n\\
\href{https://github.com/FranziskusWiesnet/MinlogArith/blob/1c63c302aa52e8972b594a3ef8753fa11a0435eb/minlog/examples/arith/gcd_nat.scm#L49}{\mathtt{NatDivToProd}}: & m \mid n \to \exl_l\ l\cdot m = n
\end{array}
\end{equation*}
\end{lemma}
\begin{proof}
\texttt{NatProdToDiv} follows directly from \texttt{ExBNatIntro} from Lemma \ref{Lem:ExBNatProp},
\texttt{NatDivToProdNc} follows directly by the definition of $m\mid n$ and \texttt{ExBNatToExNc} from Lemma \ref{Lem:ExBNatProp}, and \texttt{NatDivToProd} follows directly by the definition of $m\mid n$ and Lemma \ref{Lem:ExBPosToEx}. 
\end{proof}
Only the last formula has computational content.
As discussed after Lemma~\ref{Lem:ExBPosToEx}, the computational content of the last formula comes from the $\mu$-operator, which searches through the natural numbers from below, one by one. Hence, the computational content of this formula is quite inefficient. In our proofs, we will (fortunately) only need to use the first and the second formula.
\subsection{Greatest Common Divisor}
\label{Sec:Gcd}
The greatest common divisor is often defined by its characteristic properties: it divides both numbers, and any number dividing both also divides the greatest common divisor. Proving the existence of such a greatest common divisor then yields an algorithm for computing it, namely the standard Euclidean algorithm. This would also be a nice application of program extraction from proofs in Minlog; in fact, something similar has already been carried out by Berger and Schwichtenberg in the context of Friedman’s $A$-translation \cite{berger1996greatest}.

Here we take the opposite approach: we first give an explicit algorithm and then prove that it satisfies the defining properties. One of the advantages of working with proof assistants in constructive mathematics, and in particular with Minlog, is that both approaches are available, hence one can choose whichever fits best. In this case we want to emphasise the algorithms, and therefore we first present them explicitly.

\begin{definition}[\texorpdfstring{\href{https://github.com/FranziskusWiesnet/MinlogArith/blob/1c63c302aa52e8972b594a3ef8753fa11a0435eb/minlog/examples/arith/gcd_nat.scm\#L243}{\texttt{NatGcd}}}{NatGcd}]
\label{Def:NatGcd}
The greatest common divisor on the natural numbers is defined by the following rules:
\begin{align*}
\gcd(0,n) &:= n\\
\gcd(m,0) &:= m\\
\gcd(\Suc m, \Suc n) &:= \begin{cases} \gcd(\Suc m, n-m) & \text{if }m<n \\ \gcd(m-n,\Suc n)& \text{otherwise } \end{cases}
\end{align*}
\end{definition}

Some readers might criticise that we use recursion using only $n-m$ or $m-n$, rather than on the remainder $n \bmod m$ or $m \bmod n$, as it is common in standard presentations of the Euclidean algorithm. Equivalently, one would use recursion on $n-l\cdot m$ (or $m-l\cdot n$), where $l$ is the largest integer such that the result remains non-negative.
However, this would require computing $l$ (or defining $\bmod$) explicitly, along with corresponding computation rules. On natural numbers, this is not more efficient than the computation rules we provide here.

As we use the unary representation of the natural numbers, the efficiency of this algorithm is very poor. Already the running time to compare two natural numbers is proportional to the minimum of the two numbers. The same holds for subtraction. Both operations occur in the recursive step. Moreover, computing $\gcd(n,1)$ in this way requires exactly $n+1$ calls. Hence, the runtime of the algorithm is at least proportional to the maximum of the two arguments.

Unlike in the case of the natural numbers, division with remainder can, in principle, be defined efficiently on positive binary numbers using the $\nu$-operator, i.e. $$\lfloor p / q \rfloor:= \nu\left(\lambda_{r}(rq\leq p), \Suc (\operatorname{poslog} p\ - \operatorname{poslog}q)\right)$$
and
\begin{align*}
\dot{\gcd}(p,q) &
:= \begin{cases} 
\dot{\gcd}(p, q - \lfloor p / q \rfloor \cdot p) & \text{if }p<q\text{ and }  q \neq \lfloor p / q \rfloor \cdot p \\ 
p & \text{if }p<q\text{ and }  q = \lfloor p / q \rfloor \cdot p \\ 
\dot{\gcd}(p- \lfloor q / p \rfloor \cdot q,q)& \text{if }q<p\text{ and }  p \neq \lfloor q / p \rfloor \cdot q \\ 
q & \text{otherwise}
\end{cases}
\end{align*}
In fact, the algorithm $\dot{\gcd}$ is defined as \href{https://github.com/FranziskusWiesnet/MinlogArith/blob/1c63c302aa52e8972b594a3ef8753fa11a0435eb/minlog/examples/arith/gcd_pos.scm#L1066}{\texttt{EuclidGcd}} in Minlog, hence we refer to the corresponding definition in the Minlog file \href{https://github.com/FranziskusWiesnet/MinlogArith/blob/1c63c302aa52e8972b594a3ef8753fa11a0435eb/minlog/examples/arith/gcd_pos.scm#L1050}{\texttt{gcd\_pos.scm}} for details. 
Here we see that to compute $\lfloor p / q \rfloor$ we need at most $\Suc (\operatorname{poslog} p-\operatorname{poslog} q)$ steps. In each step a multiplication with $q$ is executed. Hence, the runtime to compute $\lfloor p / q \rfloor$ is at most quadratic in the numbers of digits of $p$ and $q$. In the worst case of each recursive call $\lfloor p / q \rfloor = 1$ or $\lfloor q / p \rfloor = 1$ and therefore, the numbers decrease only as slowly as consecutive Fibonacci numbers. Since Fibonacci numbers grow exponentially, the number of recursive calls is linear in the number of digits. Therefore, the runtime of this version of the euclidean algorithm is at most cubic in the numbers of digits of the arguments, or formally
$$T_{\text{posEuclid}}(p,q)\in \mathcal{O}\left(\operatorname{max}\left(\operatorname{poslog}(p),\operatorname{poslog}(q)\right)^3\right).$$

This is also confirmed by our measured values in Appendix \ref{Sec:AppEuclid}.
Nonetheless, this is not even necessary as by using the representation of positive binary numbers, we can provide a significantly more efficient definition than the Euclidean algorithm. This algorithm is known in the literature as \emph{Stein's algorithm} (named after Josef Stein \cite{stein1967computational}), the \emph{binary Euclidean algorithm} or simply the \emph{binary GCD algorithm}.

\begin{definition}[\texorpdfstring{\href{https://github.com/FranziskusWiesnet/MinlogArith/blob/1c63c302aa52e8972b594a3ef8753fa11a0435eb/minlog/examples/arith/gcd_pos.scm\#L138}{\texttt{PosGcd}}}{PosGcd}]
\label{Def:PosGcd}
The greatest common divisor on the positive binary numbers is defined by the following rules:
\begin{align*}
\gcd(1,p) &:= 1\\
\gcd(\SZero p, 1) &:= 1\\
\gcd(\SZero p, \SZero q) &:=\SZero(\gcd(p,q))\\
\gcd(\SZero p, \SOne q) &:= \gcd(p,\SOne q)\\
\gcd(\SOne p, 1) &:= 1\\
\gcd(\SOne p, \SZero q) &:= \gcd(\SOne p, q)\\
\gcd(\SOne p, \SOne q) &:= \begin{cases} \gcd(\SOne p,q-p) &\text{if }p<q \\ \gcd(p-q,\SOne q) &\text{if }q<p\\ \SOne p &\text{otherwise.} \end{cases}
\end{align*}
\end{definition}

According to the definition, each computation rule eliminates at least one digit from the arguments, and only in the last case, two numbers have to be compared and, if necessary, subtracted. The running time for this can be assumed to be linear in the number of binary digits and the digit length is reduced in both arguments by one. In the other cases at least one digit is directly removed, which can be done in constant time. Therefore, Stein’s algorithm runs in at most quadratic time in the maximum of the digit lengths of its two arguments, or formally in Landau notation  
$$T_{\text{stein}}(p,q)\in \mathcal{O}\left(\max\left(\operatorname{poslog}(p),\operatorname{poslog}(q)\right)^2\right),$$
where $T_{\text{stein}}(p,q)$ denotes the runtime of Stein's algorithm with input $p,q$.

This observation also aligns very well with the measured data documented in the file \href{https://github.com/FranziskusWiesnet/MinlogArith/blob/1c63c302aa52e8972b594a3ef8753fa11a0435eb/test-files/gcd_pos_test.txt#L86}{\texttt{gcd\_pos\_test.txt}}, and summarised in Appendix \ref{Sec:AppStein}.
Here we also see that Stein's algorithm operates in a completely different league than the Euclidean algorithm: while the Euclidean algorithm on binary numbers already takes more than four minutes for 5,000 digits per argument, Stein’s algorithm needs less than a second even for 10\,000 digits per argument.

The next lemma shows that the two definitions of the greatest common divisor on the natural numbers and on the positive binary numbers are compatible with each other. This is an important property that any definition of the greatest common divisor on both types must satisfy.
\begin{lemma}[\texorpdfstring{\href{https://github.com/FranziskusWiesnet/MinlogArith/blob/1c63c302aa52e8972b594a3ef8753fa11a0435eb/minlog/examples/arith/gcd_pos.scm\#L303}{\texttt{PosGcdToNatGcd}}}{PosGcdToNatGcd},
\texorpdfstring{\href{https://github.com/FranziskusWiesnet/MinlogArith/blob/1c63c302aa52e8972b594a3ef8753fa11a0435eb/minlog/examples/arith/gcd_pos.scm\#L512}{\texttt{NatGcdToPosGcd}}}{NatGcdToPosGcd}]
\label{Lem:NatPosGcd}
The two definitions above are compatible with the embedding of the positive binary numbers into the natural numbers. In particular,
\begin{align*}
\forall_{p,q}\, \operatorname{PosToNat}(\gcd(p, q)) = \gcd(\operatorname{PosToNat} p, \operatorname{PosToNat} q)
\end{align*}
and
\begin{align*}
\forall_{n,m > 0}\, \gcd(n,m) = \operatorname{PosToNat}(\gcd(\operatorname{NatToPos} n, \operatorname{NatToPos} m)).
\end{align*}
\end{lemma}

\begin{proof}
It suffices to prove the first formula, as the second follows directly from it. To do so, we prove the equivalent statement:
\begin{align*}
\forall_{n,p,q}\big(p+q < n \to \operatorname{PosToNat}(\gcd(p, q)) = \gcd(\operatorname{PosToNat} p, \operatorname{PosToNat} q)\big).
\end{align*}
This is done by induction on $n$, followed by a case distinction on $p$ and $q$. In each case, either $p$ or $q$ is 1, and the statement holds trivially, or the induction hypothesis can be applied directly -- except in the case where $p = \SOne p'$ and $q = \SOne q'$. In this case, we first do a case distinction on whether $p < q$, then apply the induction hypothesis to $p$ and $\SZero(q' - p')$, or to $\SZero(p' - q')$ and $q$, respectively.
The detailed execution of the proof is lengthy but does not yield additional insights. Therefore, we refer to the Minlog code for further details.
\end{proof}

Note that the symbol $\gcd$ is overloaded here and is used both for natural numbers and for positive binary numbers, and the corresponding definitions (Definition \ref{Def:NatGcd} and Definition \ref{Def:PosGcd}) are quite different. Hence, a formal proof of this statement is by no means trivial. By contrast, the Euclidean algorithm on positive binary numbers
$\dot\gcd$ is only defined for the purpose of comparing runtimes after Definition~\ref{Def:NatGcd} at the end of Section \ref{Sec:Bezout}, and therefore plays no role in the theorems of this article.
\subsection{Divisibility on Positive Binary Numbers}
\label{Sec:PosDiv}
In contrast to the definition of divisibility on natural numbers, we do not define divisibility on positive numbers using the bounded existential quantifier, as this would be far too inefficient. Instead, we use the greatest common divisor, since its definition on the positive binary numbers is very efficient.
\begin{definition}[\texorpdfstring{\href{https://github.com/FranziskusWiesnet/MinlogArith/blob/1c63c302aa52e8972b594a3ef8753fa11a0435eb/minlog/examples/arith/gcd_pos.scm\#L1424}{\texttt{PosDiv}}}{PosDiv}]
\begin{align*}
p\mid q\quad  :=\quad  \gcd(p,q) = p
\end{align*}
\end{definition}

As in the case of natural numbers, we also obtain the typical product representation of divisibility here. In this case, however, there would be no difference between the proof using the nc existential quantifier and the one using the cr existential quantifier. Therefore, we only provide the stronger version with the cr existential quantifier:

\begin{lemma}[\texorpdfstring{\href{https://github.com/FranziskusWiesnet/MinlogArith/blob/1c63c302aa52e8972b594a3ef8753fa11a0435eb/minlog/examples/arith/gcd_pos.scm\#L1458}{\texttt{PosProdToDiv}}}{PosProdToDiv},
\texorpdfstring{\href{https://github.com/FranziskusWiesnet/MinlogArith/blob/1c63c302aa52e8972b594a3ef8753fa11a0435eb/minlog/examples/arith/gcd_pos.scm\#L1486}{\texttt{PosDivToProd}}}{PosDivToProd}]
\begin{align*}
&\forall_{p,q,r}\left(r\cdot p = q \to p\mid q\right)\\
&\forall_{p,q}\left(p \mid q \to \exl_r\ r\cdot p =q\right)
\end{align*}
\end{lemma}

\begin{proof}
Since the first formula expresses a non-computational equality, we can prove it by reducing it to the case of natural numbers, in particular by using Lemma \ref{Lem:NatPosGcd} and Lemma \ref{Lem:NatProdToDiv}.

For the second formula, we prove the equivalent statement: 
\begin{align*} 
\forall_{l,p,q}\left(p+q<l \to p \mid q \to \exl_r\ r\cdot p =q\right)
\end{align*} 
by induction on $l$ and followed by a case distinction on $p$ and $q$. This case distinction allows us to apply the computation rules from Definition \ref{Def:PosGcd} and subsequently use the induction hypothesis on the updated arguments of $\gcd$, which are in sum smaller than $l$. (Therefore we prove the equivalent and not the original statement.) In the case where $p=\SOne p'$ and $q=\SOne q'$, an additional case distinction regarding whether $p'<q'$ is required before applying the induction hypothesis.
\end{proof}

\subsection{Properties of Divisibility}
In this section, we present the following lemmas on divisibility and the greatest common divisor. Their proofs are very straightforward and simple. Therefore, we will not provide them here but refer to the corresponding formalisations in Minlog.

\begin{lemma}\label{Lem:DivPreOrd}
Divisibility on the positive and natural numbers defines a partial order. That is, the universal closures of the following statements hold:
\begin{equation*}
\begin{array}{ll}
    \href{https://github.com/FranziskusWiesnet/MinlogArith/blob/1c63c302aa52e8972b594a3ef8753fa11a0435eb/minlog/examples/arith/gcd_nat.scm#L77}{\mathtt{NatDivRefl}},\ \href{https://github.com/FranziskusWiesnet/MinlogArith/blob/1c63c302aa52e8972b594a3ef8753fa11a0435eb/minlog/examples/arith/gcd_pos.scm#L1635}{\mathtt{PosDivRefl}}: & p\mid p \\
    \href{https://github.com/FranziskusWiesnet/MinlogArith/blob/1c63c302aa52e8972b594a3ef8753fa11a0435eb/minlog/examples/arith/gcd_nat.scm#L89}{\mathtt{NatDivTrans}},\ \href{https://github.com/FranziskusWiesnet/MinlogArith/blob/1c63c302aa52e8972b594a3ef8753fa11a0435eb/minlog/examples/arith/gcd_pos.scm#L1610}{\mathtt{PosDivTrans}}: & p_0\mid p_1 \to p_1\mid p_2\to p_0\mid p_2 \\
    \href{https://github.com/FranziskusWiesnet/MinlogArith/blob/1c63c302aa52e8972b594a3ef8753fa11a0435eb/minlog/examples/arith/gcd_nat.scm#L141}{\mathtt{NatDivAntiSym}},\ \href{https://github.com/FranziskusWiesnet/MinlogArith/blob/1c63c302aa52e8972b594a3ef8753fa11a0435eb/minlog/examples/arith/gcd_pos.scm#L1624}{\mathtt{PosDivAntiSym}}: & p\mid q \to q\mid p \to p= q
\end{array}
\end{equation*}
\end{lemma}

\begin{lemma}\label{Lem:DivProp}
The universal closures of the following statements hold:
\begin{equation*}
\begin{array}{ll}
    \href{https://github.com/FranziskusWiesnet/MinlogArith/blob/1c63c302aa52e8972b594a3ef8753fa11a0435eb/minlog/examples/arith/gcd_nat.scm#L169}{\mathtt{NatDivPlus}},\ \href{https://github.com/FranziskusWiesnet/MinlogArith/blob/1c63c302aa52e8972b594a3ef8753fa11a0435eb/minlog/examples/arith/gcd_pos.scm#L1650}{\mathtt{PosDivPlus}}: & p\mid q_0 \to p\mid q_1 \to p\mid q_0+q_1\\
    \href{https://github.com/FranziskusWiesnet/MinlogArith/blob/1c63c302aa52e8972b594a3ef8753fa11a0435eb/minlog/examples/arith/gcd_nat.scm#L196}{\mathtt{NatDivPlusRev}},\ \href{https://github.com/FranziskusWiesnet/MinlogArith/blob/1c63c302aa52e8972b594a3ef8753fa11a0435eb/minlog/examples/arith/gcd_pos.scm#L1677}{\mathtt{PosDivPlusRev}}: & p\mid q_0 \to p\mid q_0+q_1 \to p\mid q_1\\
    \href{https://github.com/FranziskusWiesnet/MinlogArith/blob/1c63c302aa52e8972b594a3ef8753fa11a0435eb/minlog/examples/arith/gcd_nat.scm#L184}{\mathtt{NatDivTimes}},\ \href{https://github.com/FranziskusWiesnet/MinlogArith/blob/1c63c302aa52e8972b594a3ef8753fa11a0435eb/minlog/examples/arith/gcd_pos.scm#L1665}{\mathtt{PosDivTimes}}: & p\mid q_0  \to p\mid q\cdot q_0
\end{array}
\end{equation*}
\end{lemma}

\begin{lemma}\label{Lem:GcdChar}
The universal closures of the following statements hold:
\begin{equation*}
\begin{array}{ll}
    \href{https://github.com/FranziskusWiesnet/MinlogArith/blob/1c63c302aa52e8972b594a3ef8753fa11a0435eb/minlog/examples/arith/gcd_nat.scm#L394}{\mathtt{NatGcdDiv0}},\ \href{https://github.com/FranziskusWiesnet/MinlogArith/blob/1c63c302aa52e8972b594a3ef8753fa11a0435eb/minlog/examples/arith/gcd_pos.scm#L1703}{\mathtt{PosGcdDiv0}}: & \gcd(p,q)\mid p\\
    \href{https://github.com/FranziskusWiesnet/MinlogArith/blob/1c63c302aa52e8972b594a3ef8753fa11a0435eb/minlog/examples/arith/gcd_nat.scm#L450}{\mathtt{NatGcdDiv1}},\ \href{https://github.com/FranziskusWiesnet/MinlogArith/blob/1c63c302aa52e8972b594a3ef8753fa11a0435eb/minlog/examples/arith/gcd_pos.scm#L1721}{\mathtt{PosGcdDiv1}}: & \gcd(p,q)\mid q\\
    \href{https://github.com/FranziskusWiesnet/MinlogArith/blob/1c63c302aa52e8972b594a3ef8753fa11a0435eb/minlog/examples/arith/gcd_nat.scm#L458}{\mathtt{NatDivGcd}},\ \href{https://github.com/FranziskusWiesnet/MinlogArith/blob/1c63c302aa52e8972b594a3ef8753fa11a0435eb/minlog/examples/arith/gcd_pos.scm#L1729}{\mathtt{PosDivGcd}}: & q\mid p_0 \to q\mid p_1 \to q \mid \gcd(p_0,p_1)
\end{array}
\end{equation*}
\end{lemma}
The last lemma also represents a characterisation of the greatest common divisor. That is, if a binary function $f$ on the positive binary numbers (or, respectively, on the natural numbers) satisfies these three statements instead of the gcd, then it is identical to the gcd. This can easily be proved by showing $f(x,y)\mid \gcd(x,y)$ and $\gcd(x,y)\mid f(x,y)$ for given $x$ and $y$.
\subsection{Bézout's identity}
\label{Sec:Bezout}
\renewcommand{\subparagraph}[1]{\vspace{1mm}\noindent\textit{#1}.}
In the standard literature (i.e.~\cite[Corollar 4.11.]{forster2015algorithmische} and \cite[Satz 3.9]{muellerstach2011elementare}), Bézout's identity is known as the statement that the greatest common divisor is a linear combination of the two arguments. That is, for integers $a,b$ there are integers $u,v$ with $\gcd(a,b) = u\cdot a + v\cdot b$. There are several methods to determine these two coefficients. The most well-known method is the extended Euclidean algorithm. However, there are also extended versions of Stein's algorithm or other efficient algorithms to determine these coefficients (see e.g.~\cite[Section 14.4.3]{menezes2018handbook} or \cite{barkema2019extending,cohen1993course}).

However, some of these methods necessarily require negative numbers. In this article, however, we aim to avoid negative numbers in order to keep the data types for an implementation as simple as possible. For this reason, we also need to reformulate the statement of Bézout's identity. We do this by moving the negative terms to the other side of the equation and using case distinction:
\begin{theorem}[\Href{https://github.com/FranziskusWiesnet/MinlogArith/blob/1c63c302aa52e8972b594a3ef8753fa11a0435eb/minlog/examples/arith/gcd_nat.scm\#L514}{NatGcdToLinComb}] \label{Thm:NatGcdToLinComb}
\begin{align*}
\forall_{n,m}\exd_{l_0}\exl_{l_1}(\gcd(n,m)+l_0\cdot n = l_1\cdot m \oru \gcd(n,m)+l_0\cdot m=l_1\cdot n)
\end{align*}
\end{theorem}
\begin{proof}
The proof is similar to, and even simpler than, the proof of the next theorem. Therefore, we omit the details here and refer to the Minlog code.
\end{proof}

For positive binary numbers, we also do not have a 0 available. Therefore, we need to consider this case separately. If one of the coefficients were 0, the greatest common divisor would be exactly one of the two arguments. This occurs when one argument is a multiple of the other. Therefore, this also becomes part of the case distinction.
\begin{theorem}[\Href{https://github.com/FranziskusWiesnet/MinlogArith/blob/1c63c302aa52e8972b594a3ef8753fa11a0435eb/minlog/examples/arith/gcd_pos.scm\#L581}{PosGcdToLinComb}]
\label{Thm:PosGcdToLinComb}
\begin{align*}
\forall_{p_0,p_1}\bigl(\quad&
\exl_{q}\ q\cdot p_0 = p_1  \\
\ord\quad & \exl_{q}\ q\cdot p_1 = p_0\\ 
\ord\quad & \exd_{q_0} \exl_{q_1}\ \gcd(p_0,p_1)+q_0\cdot p_0 = q_1\cdot p_1 \\
\ord\quad & \exd_{q_0} \exl_{q_1}\ \gcd(p_0,p_1)+q_1\cdot p_1 = q_0\cdot p_0\bigr)
\end{align*}
\end{theorem}

The proof we present here is based on Stein's algorithm of $\gcd$. Essentially, we do induction over the cases in Definition \ref{Def:PosGcd}. The extracted term can therefore be understood as a version of the extended Stein's algorithm. However, this is not the standard extended version of Stein's algorithm known from the literature, for example in \cite{cohen1993course,menezes2018handbook}. Therefore, it is not that efficient. However, we present this proof because it yields a new extension of Stein's algorithm that completely avoids the use of negative numbers.

In addition to the proof presented here, there is also a proof in Minlog that leads to the extended Euclidean algorithm as the standard extracted term (see \href{https://github.com/FranziskusWiesnet/MinlogArith/blob/1c63c302aa52e8972b594a3ef8753fa11a0435eb/minlog/examples/arith/gcd_pos.scm#L1212}{\texttt{PosGcdToLinCombEuclid}}).
\begin{proof}
We prove the equivalent statement
\begin{align*}
\forall_{l,p_0,p_1}\bigl(&\   p_0+p_1 <l \to \\
&\ \phantom{\ord}\ \exl_{q}\ q\cdot p_0 = p_1  \\
&\ord\ \exl_{q}\ q\cdot p_1 = p_0\\ 
&\ord\  \exd_{q_0} \exl_{q_1}\ \gcd(p_0,p_1)+q_0\cdot p_0 = q_1\cdot p_1 \\
&\ord\  \exd_{q_0} \exl_{q_1}\ \gcd(p_0,p_1)+q_1\cdot p_1 = q_0\cdot p_0\ \bigr)
\end{align*}
by induction on $l$. For $l=0$, $p_0+p_1<l$ is not possible and hence, there is nothing to show.

For the induction step, let $l$ with $p_0+p_1<\Suc l$ be given. We apply a case distinction on $p_0$ and $p_1$. In particular, $p_0=1$, $p_0=\SZero p'_0$ or $p_0=\SOne p'_0$ and analogously for $p_1$. 

\subparagraph{Case 1: $p_0=1$ or $p_1=1$} In this case either $\exl_{q} q\cdot p_0 = p_1$ or $\exl_{q} q\cdot p_1 = p_0$ and we are done.

\subparagraph{Case 2: $p_0=\SZero p'_0 =2\cdot p'_0$ and $p_1=\SZero p'_1=2\cdot p'_1$} Here, we have $\gcd(p_0,p_1)=2\cdot\gcd(p'_0,p'_1)$. The statement follows directly from the induction hypothesis applied to $p'_0$ and $p'_1$, with each resulting equation multiplied by 2.

\subparagraph{Case 3: $p_0=\SZero p'_0 =2\cdot p'_0$ and $p_1=\SOne p'_1=2\cdot p'_1 +1$ or vice versa} We consider, without loss of generality, only the first case, as the second case is analogous. Then, we have $\gcd(p_0,p_1)=\gcd(p'_0,p_1)$. Hence, we apply the induction hypothesis to $p'_0$ and $p_1$ and get 
\begin{align*}
& \exl_{q}\ q\cdot p'_0 = p_1 \\
\ord\quad & \exl_{q}\ q\cdot p_1 = p'_0\\ 
\ord\quad &  \exd_{q_0} \exl_{q_1}\ \gcd(p'_0,p_1)+q_0\cdot p'_0 = q_1\cdot p_1 \\
\ord\quad &  \exd_{q_0} \exl_{q_1}\ \gcd(p'_0,p_1)+q_1\cdot p_1 = q_0\cdot p'_0.
\end{align*}
We consider all four cases one by one:

\subparagraph{Case 3.1: $q\cdot p'_0 = p_1$ for some $q$} Then $\gcd(p_0,p_1)=\gcd(p'_0,p_1)=\gcd(p'_0,q\cdot p'_0)=p'_0$, therefore $\gcd(p_0,p_1)+p'_1\cdot p_0=p'_0+p'_1\cdot 2\cdot p'_0 = p'_0(2\cdot p'_1+1)=p'_0\cdot p_1$. This implies the third statement of the disjunctions.

\subparagraph{Case 3.2: $q\cdot p_1 = p'_0$ for some $q$} Here, $2\cdot q \cdot p_1 = p_0$ and we are done.

\subparagraph{Case 3.3: $\gcd(p'_0,p_1)+q_0\cdot p'_0 = q_1\cdot p_1$ for some $q_0,q_1$} We calculate 
\begin{align*}
\gcd(p_0,p_1) +(p'_1 +1)\cdot q_0 \cdot p_0 &= \gcd(p'_0,p_1)+q_0\cdot p'_0+ q_0\cdot p'_0  + p'_1\cdot q_0\cdot p_0 \\ &= q_1\cdot p_1 + q_0\cdot p'_0  + p'_1\cdot q_0\cdot p_0 \\
&=q_1\cdot p_1 + q_0\cdot p'_0  + p'_1\cdot q_0\cdot 2\cdot p'_0\\
&=q_1\cdot p_1 + (2\cdot p'_1 +1)\cdot q_0\cdot p'_0\\
&=(q_1+q_0\cdot p'_0)\cdot p_1,
\end{align*}
which proves the third part of the disjunction.

\subparagraph{Case 3.4: $\gcd(p'_0,p_1)+q_1\cdot p_1 = q_0\cdot p'_0$ for some $q_0,q_1$} We get
\begin{align*}
\gcd(p_0,p_1)+ (q_1+q_0\cdot p'_0)\cdot p_1 &=\gcd(p'_0,p_1)+q_1\cdot p_1 +q_0\cdot p'_0\cdot p_1\\
&=q_0\cdot p'_0 + q_0\cdot p'_0 \cdot p_1\\
&= q_0\cdot p'_0 + q_0\cdot p'_0 \cdot (2\cdot p'_1 +1)\\
&= q_0\cdot 2\cdot p'_0 + q_0\cdot 2\cdot p'_0 \cdot p'_1 \\
&=(q_0+q_0\cdot p'_1)\cdot p_0
\end{align*}
showing the fourth part of the disjunction.

\subparagraph{Case 4: $p_0=\SOne p'_0$ and $p_1=\SOne p'_1$} If $p_0=p_1$ we are obviously done. Hence, without loss of generality, we assume $p_0<p_1$, and therefore $\gcd(p_0,p_1)=\gcd(p_0,p'_1-p'_0)=\gcd(p_0,\SZero (p'_1-p'_0))=\gcd(p_0,\SZero p'_1-\SZero p'_0)$. Applying the induction hypothesis to $p_0$, $\SZero p'_1-\SZero p'_0$, we get
\begin{align*}
&\exl_{q}\ q\cdot p_0 = \SZero p'_1-\SZero p'_0 \\
\ord\quad & \exl_{q}\ q\cdot (\SZero p'_1-\SZero p'_0) = p_0 \\ 
\ord\quad &  \exd_{q_0} \exl_{q_1}\ \gcd(p_0,\SZero p'_1-\SZero p'_0)+q_0\cdot p_0 = q_1\cdot (\SZero p'_1-\SZero p'_0) \\
\ord\quad &  \exd_{q_0} \exl_{q_1}\ \gcd(p_0,\SZero p'_1-\SZero p'_0)+q_1\cdot (\SZero p'_1-\SZero p'_0) = q_0\cdot p_0.
\end{align*}
Again, we consider each case one by one.

\subparagraph{Case 4.1: $q\cdot p_0 = \SZero p'_1-\SZero p'_0$ for some $q$} We have directly $$(q+1)\cdot p_0 =(\SZero p'_1-\SZero p'_0)+\SOne p'_0=\SZero p'_1+1 = p_1,$$
which proves the first part of the disjunction.

\subparagraph{Case 4.2: $q\cdot (\SZero p'_1-\SZero p'_0) = p_0$ for some $q$} Here we have $\SZero(q\cdot ( p'_1-p'_0))=\SOne p'_0$, which is not possible, hence there is nothing to show.

\subparagraph{Case 4.3: $\gcd(p_0,\SZero p'_1-\SZero p'_0)+q_0\cdot p_0 = q_1\cdot (\SZero p'_1-\SZero p'_0)$ for some $q_0,q_1$} Here, we get 
\begin{align*}
\gcd(p_0,p_1) +(q_0+q_1)\cdot p_0 = q_1\cdot (\SZero p'_1-\SZero p'_0) +q_1\cdot p_0= q_1\cdot(\SZero p'_1 +1)=q_1\cdot p_1,  
\end{align*}
which proves the third part of the disjunction.

\subparagraph{Case 4.4: $\gcd(p_0,\SZero p'_1-\SZero p'_0)+q_1\cdot (\SZero p'_1-\SZero p'_0) = q_0\cdot p_0$ for some $q_0,q_1$} We calculate
\begin{align*}
\gcd(p_0,p_1) + q_1\cdot p_1 &= \gcd(p_0,\SZero p'_1-\SZero p'_0) +q_1\cdot p_1\\
&=  \gcd(p_0,\SZero p'_1-\SZero p'_0)+q_1\cdot (p_1 -p_0) +q_1\cdot p_0\\
&= \gcd(p_0,\SZero p'_1-\SZero p'_0)+ q_1\cdot (\SZero p'_1 -\SZero p'_0)+q_1\cdot p_0\\
&= q_0\cdot p_0 +q_1\cdot p_0 = (q_0+q_1)\cdot p_0,
\end{align*}
which proves the fourth part of the equation.
\end{proof}

Here, we have presented the proof in full, as Bézout's identity is a central result in elementary number theory. Furthermore, the proof leads to a new algorithm, which we will examine in more detail in the following paragraph.
Since the proof involves numerous case distinctions, formalisation in a proof assistant is particularly beneficial. This guarantees that no cases are overlooked and that those cases often regarded as trivial are indeed trivial and free of any hidden subtleties.

\paragraph{The extracted term.}
Due to the numerous case distinctions, the extracted term is also quite lengthy even when written in Haskell notation. However, this illustrates one of the major advantages of automatic program extraction: we do not need to construct and replicate the entire term manually, but instead receive it directly from the computer. 
Furthermore, due to the correctness theorem of proof extraction, we can be confident that the extracted term behaves as intended.

Nevertheless, if we take a look at the extracted term in the Haskell program, we notice that while there are many case distinctions, there is only a single recursion operator at the beginning. This comes from the induction on $l$ in the proof. Since there is no nested induction, but rather a single induction that, as seen in the proof, does not revert to the predecessor but instead halves at least one of the arguments, we can already conclude from this limited information that the extracted term represents an efficient algorithm.

However, when we consider Cases 3.3 and 3.4 in the proof, we observe a flaw in the algorithm: In Case 3.3, for instance, the new coefficients are \((p'_1 + 1) \cdot q_0\) and \(q_1 + q_0 \cdot p'_0\), where \(q_0\) and \(q_1\) are the original coefficients, and \(p_0\) and \(p_1\) are the arguments of \(\gcd\).
This involves the multiplication of two numbers, which approximately corresponds to an addition in the number of their digits. Since such operations can occur at each step of the recursion -- which is called roughly as many times as the number of digits in the input -- the total number of digits in the output is approximately proportional to the square of the number of digits in the input.
In the Euclidean algorithm, however, the number of digits in the output is roughly proportional to the number of digits in the input.

Tests documented in \href{https://github.com/FranziskusWiesnet/MinlogArith/blob/1c63c302aa52e8972b594a3ef8753fa11a0435eb/test-files/gcd_pos_test.txt#L172}{\texttt{gcd\_pos\_test.txt}} and summarised in Appendix \ref{Sec:AppExtStein} of the generated Haskell program, using random decimal numbers with up to 1000 digits, have shown that the corresponding coefficients can be computed in way under a minute -- even though the resulting coefficients together have over 4.6 million digits. Since the number of digits grow so very quickly (i.e. quadratic), tests with a substantially larger number of digits are hardly to carry out. Therefore, we were not able to reliably determine an approximating polynomial for the runtime from the measured data.
The algorithm extracted from the proof which emulates the Euclidean algorithm, operates with approximately half the runtime and yields a more concise output. Nonetheless, both algorithms result in polynomial growth in terms of input size and are therefore at least quite efficient in runtime. However, since in each step the number of digits increases at most quadratically, and the number of steps -- just as for the standard Stein algorithm -- is linear in the number of digits, the runtime should grow cubically with the digit length of the arguments. Therefore, in the plot in Appendix \ref{Sec:AppExtStein} we included the approximating cubic polynomial fitted to the measured values. However, since the degree was restricted to 3, this does not necessarily mean that the runtime truly grows cubically.

For comparison, we have also implemented the extended Euclidean algorithm on positive binary numbers as \href{https://github.com/FranziskusWiesnet/MinlogArith/blob/1c63c302aa52e8972b594a3ef8753fa11a0435eb/minlog/examples/arith/gcd_pos.scm#L1212}{\texttt{PosGcdToLinCombEuclid}} in Minlog, and likewise recorded runtime measurements for it in \href{https://github.com/FranziskusWiesnet/MinlogArith/blob/1c63c302aa52e8972b594a3ef8753fa11a0435eb/test-files/gcd_pos_test.txt#L256}{\texttt{gcd\_pos\_test.txt}} and summarised them in Appendix \ref{Sec:AppExtEuclid}. Since we already observed that the runtime of the Euclidean algorithm grows cubically with the number of input digits, and since the number of digits of the coefficients grows only linearly with the number of input digits, it is reasonable to expect that the cubic growth persists, which is exactly what the measured data suggest.

For the purposes of this article, however, the computational use of Bézout’s identity is not required, and we therefore refrain from pursuing further improvements in this regard. What is essential in our context is the efficient computation of the greatest common divisor, which we have achieved using Stein’s algorithm. For future work, it would be worthwhile to investigate an extension of Stein’s algorithm within Minlog, as done on paper in \cite{barkema2019extending} and \cite[Algorithm 14.61]{menezes2018handbook}.

\section{Prime Numbers}
Prime numbers are an integral part of elementary number theory. In this section, we introduce the theory of prime numbers within {\TCF} and examine their computational aspects. We will see that even in the case of positive binary numbers, handling primes involves significant computational effort.
\subsection{Definition of Primes}
A number \(n > 1\) is called composite if it has a divisor greater than 1 and smaller than itself; otherwise, it is prime. This property is defined using the bounded existential quantifier. As a result, determining whether a number is composite or prime is relatively inefficient and typically requires many computational steps. 
In the case of positive binary numbers, we will at least take advantage of the efficient integer square root function to restrict the search to divisors up to the square root of \(n\).
\begin{definition}[\texorpdfstring{\href{https://github.com/FranziskusWiesnet/MinlogArith/blob/1c63c302aa52e8972b594a3ef8753fa11a0435eb/minlog/examples/arith/prime_nat.scm\#L12}{\texttt{NatComposed}}}{NatComposed},
\texorpdfstring{\href{https://github.com/FranziskusWiesnet/MinlogArith/blob/1c63c302aa52e8972b594a3ef8753fa11a0435eb/minlog/examples/arith/prime_pos.scm\#L15}{\texttt{PosComposed}}}{PosComposed},
\texorpdfstring{\href{https://github.com/FranziskusWiesnet/MinlogArith/blob/1c63c302aa52e8972b594a3ef8753fa11a0435eb/minlog/examples/arith/prime_nat.scm\#L23}{\texttt{NatPrime}}}{NatPrime},
\texorpdfstring{\href{https://github.com/FranziskusWiesnet/MinlogArith/blob/1c63c302aa52e8972b594a3ef8753fa11a0435eb/minlog/examples/arith/prime_pos.scm\#L26}{\texttt{PosPrime}}}{PosPrime}]
Let $n$ be a natural number. We define the program constant $\operatorname{NatComposed}: \NN \to \BB$ by
\[
\operatorname{NatComposed}\ n\quad :=\quad \exbn{n}_m \ \left(1 < m \andb m \mid n\right).
\]
Similarly, for a positive number $p$, we define $\operatorname{PosComposed}: \PP \to \BB$ by
\[
\operatorname{PosComposed}\ p\quad :=\quad \exbp{\lfloor \sqrt{p} \rfloor}_q \ \left(1 < q \andb q \mid p\right).
\]
A natural number $n$ is defined as prime by
\[
\operatorname{NatPrime}\ n\quad :=\quad \neg(\operatorname{NatComposed}\ n) \andb n > 1.
\]
A positive number $p$ is defined as prime by
\[
\operatorname{PosPrime}\ p\quad :=\quad \neg(\operatorname{PosComposed}\ p) \andb p > 1.
\]
When it is clear from the context, or if we refer to both, we simply write
$\Prime{p}$ for both $\operatorname{NatPrime}\ p$ and $\operatorname{PosPrime}\ p$. For a sequence $ps$ of positive binary numbers or natural numbers $\Primes{ps}{n}$ means that the first $n$ numbers of $ps$ are prime. Formally it is defined by the rules $\Primes{ps}{0} := \true$ and $\Primes{ps}{\Suc n} := \Primes{ps}{n} \andb \Prime{ps\ n}$.
\end{definition}

In the following sections of this article, we will often use another characterisation of a prime number, which is that it is greater than 1, and if it can be expressed as the product of two numbers, one of these numbers is 1. This is expressed in the following lemma:

\begin{lemma}[\Href{https://github.com/FranziskusWiesnet/MinlogArith/blob/1c63c302aa52e8972b594a3ef8753fa11a0435eb/minlog/examples/arith/prime_nat.scm\#L34}{NatPrimeProd},
\Href{https://github.com/FranziskusWiesnet/MinlogArith/blob/1c63c302aa52e8972b594a3ef8753fa11a0435eb/minlog/examples/arith/prime_nat.scm\#L85}{NatProdToPrime},
\Href{https://github.com/FranziskusWiesnet/MinlogArith/blob/1c63c302aa52e8972b594a3ef8753fa11a0435eb/minlog/examples/arith/prime_pos.scm\#L37}{PosPrimeProd}, \Href{https://github.com/FranziskusWiesnet/MinlogArith/blob/1c63c302aa52e8972b594a3ef8753fa11a0435eb/minlog/examples/arith/prime_pos.scm\#L105}{PosProdToPrime}]
\label{Lem:PrimeChar}
For every natural or positive number \(p\), the following statement is equivalent to \(p\) being prime:
\[
1 < p\ \ \andnc\ \forall_{q_0, q_1} \bigl(q_0 \cdot q_1 = p \to q_0 = 1 \;\orb \; q_1 = 1\bigr).
\]
\end{lemma}

\begin{proof}
First, assume \(1 < p\) and 
\[
\forall_{q_0, q_1} \bigl(q_0 \cdot q_1 = p \to q_0 = 1 \;\orb \; q_1 = 1\bigr).
\]
Our goal is to show \(\Prime{p}\). Since \(1 < p\) is given, it suffices to show that \(p\) is not composed.

Hence, we assume that there exists some \(q \leq \lfloor \sqrt{p} \rfloor\) with \(1 < q\) and \(q \mid p\).  
From \(q \mid p\), we obtain an \(r\) such that
\[
r \cdot q = p.
\]
By our assumption, either \(q = 1\) or \(r = 1\). 
If \(q = 1\), there follows directly a contradiction since \(q > 1\).
If \(r = 1\), then \(q = p\). However, since \(1 < p\), we have
    \[
    q \leq \lfloor \sqrt{p} \rfloor < p,
    \]
which contradicts \(q = p\).

Conversely, assume \(\Prime{p}\).  
By definition, \(1 < p\) holds.  
Let \(q_0, q_1\) be such that:
\[
q_0 \cdot q_1 = p.
\]
As the disjunction \(q_0 = 1 \orb q_1 = 1\) is decidable, we assume \(q_0 > 1\) and \(q_1 > 1\) and show a contradiction.  
From \(q_0, q_1 > 1\), it follows
$q_0, q_1 < p$, and from $q_0\cdot q_1=p$ we get $q_0 \leq \lfloor \sqrt{p} \rfloor$ or $q_1 \leq \lfloor \sqrt{p} \rfloor$ .
Without loss of generality, we assume
\[
q_0 \leq \lfloor \sqrt{p} \rfloor.
\]
Since \(q_0 \cdot q_1 = p\), it follows that \(q_0 \mid p\), and therefore $\exbp{\lfloor \sqrt{p} \rfloor}_q \bigl(1 < q \;\andb\; q \mid p\bigr)$.
This together with $\Prime{p}$ leads to a contradiction.
\end{proof}

Note that the statement in this lemma does not carry any computational content, as it concludes with a Boolean disjunction. However, computational content is not required in this case, since the formulas \(q_0 = 1\) and \(q_1 = 1\) can be verified directly.
\subsection{Smallest Divisor}
In this section, we present the simplest method for finding a proper divisor of a natural number or a positive binary number: a bounded search from below using the \(\mu\)-operator. 
It is important to note that this approach becomes highly inefficient, which is typical for factorisation algorithms. Nevertheless, in the following sections -- particularly when we aim to prove the fundamental theorem of arithmetic -- we will require the fact that every number is divisible by a prime.
\begin{definition}[\Href{https://github.com/FranziskusWiesnet/MinlogArith/blob/1c63c302aa52e8972b594a3ef8753fa11a0435eb/minlog/examples/arith/prime_nat.scm\#L143}{NatLeastFactor}, \Href{https://github.com/FranziskusWiesnet/MinlogArith/blob/1c63c302aa52e8972b594a3ef8753fa11a0435eb/minlog/examples/arith/prime_pos.scm\#L566}{PosLeastFactor}]
\label{Def:LeastFactor}
For a natural number $n$ we define the least factor (greater than 1) of $n$ by
\begin{align*}
\lf n := \mu_{i<n}\left(1<i\ \andb\ \exbn{\Suc n}_j j\cdot i=n\right)
\end{align*}
For efficiency reasons the definition of $\lf p$ is different. We first define $\lf_{aux} p$ by
\begin{align*}
\lf_{aux} p:=\mu_{q<\lfloor \sqrt{p} \rfloor + 1}\left(1< q\ \andb\ \gcd(q,p)=q \right),
\end{align*}
then $\lf p$ is given by
\begin{align*}
\lf p := \begin{cases} 
 1 & \text{ if }  p = 1, \\
 2 & \text{ if }  p = \SZero p', \\
 \lf_{aux} p & \text{ if }   p = \SOne p'  \text{ and } \lf_{aux} p \leq \lfloor \sqrt{p} \rfloor, \\
 p & \text{ otherwise.}
 \end{cases}
\end{align*}
\end{definition}

Note that both definitions ultimately produce the same result, as shown by the theorem \href{https://github.com/FranziskusWiesnet/MinlogArith/blob/1c63c302aa52e8972b594a3ef8753fa11a0435eb/minlog/examples/arith/prime_pos.scm#L616}{\texttt{PosToNatLeastFactor}} in Minlog.
However, der definition for binary numbers is more complex in order to minimise runtime. First, the last digit is inspected to determine whether the smallest divisor is 2. In that case, 2 is returned immediately. Otherwise, the function $\lf_{aux}$ searches upward from below for the smallest divisor. Here, the search is bounded by the square root, and divisibility is checked using the $\gcd$ algorithm (as divisibility is defined for binary numbers).
Although the similar improvements could be done for natural numbers, in our setting we prioritise simplicity of definitions and proofs. For positive numbers, by contrast, the goal is to minimise the runtime of the extracted term.

Furthermore, note that the \(\mu\)-operator is defined only for natural numbers. As a result, we implicitly rely on the transformation between positive and natural numbers. In the Minlog implementation, this transformation must be made explicit using \(\operatorname{PosToNat}\) and \(\operatorname{NatToPos}\).
\begin{lemma} \label{Lem:LeastFactorProp}
 For $p>1$ the following properties about $\lf$ hold: 
\begin{equation*}
\begin{array}{ll}
\href{https://github.com/FranziskusWiesnet/MinlogArith/blob/1c63c302aa52e8972b594a3ef8753fa11a0435eb/minlog/examples/arith/prime_nat.scm#L174}{\mathtt{LeastFactorProp0}},
    \ \href{https://github.com/FranziskusWiesnet/MinlogArith/blob/1c63c302aa52e8972b594a3ef8753fa11a0435eb/minlog/examples/arith/prime_pos.scm#L960}{\mathtt{PosOneLtLeastFactor}}: & 1< \lf p\\
    \href{https://github.com/FranziskusWiesnet/MinlogArith/blob/1c63c302aa52e8972b594a3ef8753fa11a0435eb/minlog/examples/arith/prime_nat.scm#L182}{\mathtt{LeastFactorProp1}},\ \href{https://github.com/FranziskusWiesnet/MinlogArith/blob/1c63c302aa52e8972b594a3ef8753fa11a0435eb/minlog/examples/arith/prime_pos.scm#L969}{\mathtt{PosLeastFactorDiv}}: & (\lf p) \mid p
    
\end{array}
\end{equation*}
\end{lemma}
\begin{proof}
This follows from the definition above and Lemma \ref{Lem:LeastNatProp}. For details, we refer to the Minlog code.
\end{proof}

The key property of the least divisor is that it is a prime number. Note that we do not formally prove that \(\lf p\) is indeed the smallest divisor greater than 1, as this will not be required in the subsequent theorems. However, this fact is implicitly shown and used in the proof of the following lemma.
\begin{lemma}[\Href{https://github.com/FranziskusWiesnet/MinlogArith/blob/1c63c302aa52e8972b594a3ef8753fa11a0435eb/minlog/examples/arith/prime_nat.scm\#L216}{NatLeastFactorPrime}, \Href{https://github.com/FranziskusWiesnet/MinlogArith/blob/1c63c302aa52e8972b594a3ef8753fa11a0435eb/minlog/examples/arith/prime_pos.scm\#L942}{PosLeastFactorPrime}]
\label{Lem:LeastFactorPrime}
\begin{align*}
\forall_{p>1}\,\Prime{\lf p}
\end{align*}
\end{lemma}
\begin{proof}
Let $p > 1$ be given. If $p=\SZero p'$, $\lf p=2$ which is obviously prime. Therefore, let $p=\SOne p'$ and assume $q \cdot r = \lf p$. If $\lf p = p$ then, because of  \texttt{NatLeastUpLeIntro} from Lemma~\ref{Lem:LeastNatProp}, there is no $q > 1$ and $q<p$ such that $q\mid p$. Hence $p = \lf p$ is prime.

Therefore, we may assume that $\lf p \neq p$, and thus $(\lf p) \leq \lfloor\sqrt{p}\rfloor$ by the definition of $\lf p$.

Let $\lf p = q \cdot r$ for some $q$ and $r$. Our goal is to show that $q = 1$ or $r = 1$.
By Lemma \ref{Lem:LeastFactorProp}, we have $(\lf p) \mid p$, so there exists some $p_0$ such that $p_0 \cdot \lf p = p$, and also $1 < \lf p=q\cdot r$. It follows that $p_0 \cdot q \cdot r = p$, and either $q > 1$ or $r > 1$.

Without loss of generality, assume $r > 1$. Then $r$ satisfies $1 < r$ and $r\mid p$. By the statement \texttt{NatLeastUpLeIntro} from Lemma \ref{Lem:LeastNatProp}, we have $\lf p \leq r$. Since $\lf p = q \cdot r$ and hence $r\leq \lf p$, it follows that $r = \lf p$, and thus $q = 1$.
\end{proof}

\subsection{Infinitude of Prime Numbers}
In this section, we prove the infinitude of prime numbers by following the classical approach of Euclid -- the most well-known proof of this statement.

While we base our formalisation on Euclid’s argument, it is worth noting that numerous alternative proofs exist. For example, \cite[Chapter 1]{aigner2010das} presents six distinct proofs of the infinitude of primes. Analysing the computational content of these proofs and possibly applying formal program extraction techniques to them could be an interesting direction for future work.
Notably, Ulrich Kohlenbach has examined various proof strategies for the infinitude of primes from a quantitative perspective \cite[Proposition 2.1]{kohlenbach2008applied}. His approach provides valuable insight into how different logical approaches can yield computational information. 

In contrast to the standard methodology in proof mining -- where existing proofs are analysed to extract their quantitative content -- we take a slightly different approach in this article. We do not begin with a completed textbook proof but rather formulate both the theorem and its proof in a way that is tailored to the extraction of a desired computational term.
For our purposes, we refine the classical statement as follows: for any finite set of prime numbers, there exists a new prime number. We define what it means for a number to be ``new'' using a program constant, which will be crucial for the subsequent formal program extraction.

\begin{definition}[\Href{https://github.com/FranziskusWiesnet/MinlogArith/blob/1c63c302aa52e8972b594a3ef8753fa11a0435eb/minlog/examples/arith/prime_nat.scm\#L419}{NatNewNumber}, \Href{https://github.com/FranziskusWiesnet/MinlogArith/blob/1c63c302aa52e8972b594a3ef8753fa11a0435eb/minlog/examples/arith/prime_pos.scm\#L1168}{PosNewNumber}]
Let $ps$ be a sequence of positive or natural numbers and let $p$ be a positive or natural number. We define the notion $p\notin \{ps(i)\mid i < n\}$ by the rules
\begin{align*}
\forall_{p,ps}\bigl(&p\notin \{ps(i)\mid i < 0\}\bigr) \\
\forall_{p,n,ps}\bigl(p\notin \{ps(i)\mid i < n\} \to  p \neq ps(n) \to& p\notin \{ps(i)\mid i < \Suc n\}\bigr)
\end{align*}
\end{definition}
Note that the use of set notation here is merely a notation. Formally, it is an inductively defined predicate that takes $p$, $ps$ and $n$ as arguments. Sets themselves do not appear as objects in our consideration.

\begin{lemma}[\Href{https://github.com/FranziskusWiesnet/MinlogArith/blob/1c63c302aa52e8972b594a3ef8753fa11a0435eb/minlog/examples/arith/prime_nat.scm\#L497}{NatNotDivProdToNewNumber}, \Href{https://github.com/FranziskusWiesnet/MinlogArith/blob/1c63c302aa52e8972b594a3ef8753fa11a0435eb/minlog/examples/arith/prime_pos.scm\#L1248}{PosNotDivProdToNewNumber}]
\label{Lem:NatDivProdToNewNumber}
\begin{align*}
\forall_{p,ps,n}\left(\neg \left(p\Bigm| \prod_{i<n}ps(i)\right) \to p\notin \{ps(i) \mid i<n \}\right).
\end{align*}
\end{lemma}
\begin{proof}
For given $p$ and $ps$, this follows quite directly by induction on $n$. For details, we refer to the Minlog file. \end{proof}

\begin{theorem}[\Href{https://github.com/FranziskusWiesnet/MinlogArith/blob/1c63c302aa52e8972b594a3ef8753fa11a0435eb/minlog/examples/arith/prime_nat.scm\#L595}{NatListToNewPrime}, \Href{https://github.com/FranziskusWiesnet/MinlogArith/blob/1c63c302aa52e8972b594a3ef8753fa11a0435eb/minlog/examples/arith/prime_pos.scm\#L1316}{PosListToNewPrime}]
\label{Thm:PrimesToNewPrimes}
\begin{align*}
\forall_{ps,n}\exl_p\left(\Prime{p} \andnc p \notin \{ps(i) \mid i<n\}\right).
\end{align*}
\end{theorem}

\begin{proof}
Let $ps$ and $n$ be given. We consider
\[
N := \prod_{i<n}ps(i) + 1.
\]  
Since $N \geq 1+1 >1$, by Lemma \ref{Lem:LeastFactorProp} we obtain
\[
(\lf N) \mid N \quad \text{and} \quad 1 < \lf N.
\]
By Lemma \ref{Lem:LeastFactorPrime}, $\lf N$ is prime.  
We set 
\[
p := \lf N.
\]  
Then, clearly
\[
\neg \left(p\Bigm| \prod_{i<n}ps(i)\right),
\]
as otherwise we would obtain $p\mid 1$ by $\mathtt{PosDivPlusRev}$ from Lemma \ref{Lem:DivProp}, which is impossible.
Therefore, Lemma \ref{Lem:NatDivProdToNewNumber} completes the proof.
\end{proof}

Note that we have not assumed that $ps(0),\dots,ps(n-1)$ only consist of prime numbers. Therefore, this theorem is actually a generalisation of Euclid's lemma. However, for the proof on natural numbers we assumed, for simplicity, that every list entry is positive. This is not a genuine restriction as one can just omit those entries of the original sequence that are equal to $0$.

The computational content of this theorem is so compact this time that we can even express it in Minlog notation:
\begin{verbatim}
[m,ps]PosLeastFactor(PosProd Zero m ps+1)
\end{verbatim}
As expected, the extracted term takes as input a natural number $m$ and a sequence $ps$ of positive binary numbers, and returns the smallest proper factor of $\prod_{i<m}ps(i) +1$.
\subsection{Euclid's Lemma}
In the next section, we will prove the fundamental theorem of arithmetic. A key lemma in this context is that prime numbers are also irreducible -- a result commonly known in the literature as Euclid’s Lemma. As this is a well-known and frequently used statement, we dedicate a separate section to this result.

\begin{lemma}[\Href{https://github.com/FranziskusWiesnet/MinlogArith/blob/1c63c302aa52e8972b594a3ef8753fa11a0435eb/minlog/examples/arith/prime_nat.scm\#L647}{NatPrimeToIrred}, \Href{https://github.com/FranziskusWiesnet/MinlogArith/blob/1c63c302aa52e8972b594a3ef8753fa11a0435eb/minlog/examples/arith/prime_pos.scm\#L1345}{PosPrimeToIrred}]
\label{Lem:PrimeToIrred}
\begin{align*}
\forall_{p,q_0,q_1}\left(\Prime{p} \to p\mid q_0\cdot q_1 \to p\mid q_0 \orb p\mid q_1\right).
\end{align*}
\end{lemma}
\begin{proof}
From $p\mid q_0\cdot q_1$, there exists some $p_0$ such that  
\[
p\cdot p_0 = q_0\cdot q_1.
\]
By applying Theorem \ref{Thm:PosGcdToLinComb} to $p$ and $q_0$, we obtain four possible cases:
\begin{align*}
\exl_q &\ q\cdot p = q_0,  \\
\exl_q &\ q\cdot q_0 = p, \\ 
\exd_{r_0} \exl_{r_1} &\ \gcd(p,q_0) + r_0\cdot p = r_1\cdot q_0,  \\
\exd_{r_0} \exl_{r_1} &\ \gcd(p,q_0) + r_1\cdot q_0 = r_0\cdot p.
\end{align*}

\noindent \subparagraph{Case 1} If $q\cdot p = q_0$, then $p\mid q_0$ follows directly.

\noindent \subparagraph{Case 2}  If $q\cdot q_0 = p$, then, since $p$ is prime, either $q=1$ or $q_0=1$. Thus, we conclude that $p\mid q_0$ or $p\mid q_1$.

\noindent \subparagraph{Case 3}  Suppose there exist $r_0, r_1$ such that  
  \[
  \gcd(p,q_0) + r_0\cdot p = r_1\cdot q_0.
  \]
  Since $p$ is prime and $\gcd(p,q_0) \mid p$, we have either $\gcd(p,q_0) = 1$ or $\gcd(p,q_0) = p$.  
If $\gcd(p,q_0) = p$, then $p\mid q_0$ follows immediately.
If $\gcd(p,q_0) = 1$, then we obtain  
    \[
    1 + r_0\cdot p = r_1\cdot q_0.
    \]
    We aim to show that $p\mid q_1$.  
    Using \texttt{PosDivPlusRev} from Lemma \ref{Lem:DivProp}, we consider  
    \[
    q_1\cdot r_0\cdot p + q_1 = q_1(1 + r_0\cdot p) = q_1\cdot r_1\cdot q_0.
    \]
    Since $p \mid q_0\cdot q_1$ and $p\mid q_1\cdot r_0\cdot p$, the result follows.
    
\subparagraph{Case 4} Again, if $\gcd(p,q_0) = p$, then $p\mid q_0$. Otherwise, we have $\gcd(p,q_0) = 1$, so we obtain  
  \[
  1 + r_0\cdot q_0 = r_0\cdot p.
  \]
  We aim to show that $p\mid q_1$ by using $\mathtt{PosDivPlusRev}$ from Lemma \ref{Lem:DivProp}.  
  As obviously $p\mid r_1\cdot p_0\cdot p$, it suffices to prove  
  \[
  p\mid r_1\cdot p_0\cdot p + q_1.
  \]
  Since $p_0\cdot p = q_0\cdot q_1$, we obtain  
  \[
  r_1\cdot p_0\cdot p + q_1 = r_1\cdot q_0\cdot q_1 + q_1 = q_1(1 + r_1\cdot q_0) = q_1\cdot r_0\cdot p.
  \]
  Thus, $p\mid r_1\cdot p_0\cdot p + q_1$, completing the proof.
\qed\end{proof}

Since the conclusion of the lemma is formulated as a Boolean disjunction, the statement carries no computational content. If, however, one were to replace the Boolean disjunction by an inductively defined (computational) disjunction, program extraction would yield a term which, given $p, q_0, q_1$, decides whether $p$ divides $q_0$ or $p$ divides $q_1$. 
For positive binary numbers, the predicates $p \mid q_0$ and $p \mid q_1$ are decidable and can be checked very efficiently using Stein's algorithm. The extracted term, however, would depend on an extension of this algorithm, namely the term extracted from Theorem~\ref{Thm:PosGcdToLinComb}. As discussed after the proof of that theorem, this extension is not quite as efficient. 
For natural numbers, by contrast, divisibility is not computable as efficiently as it is for positive binary numbers. Therefore, \texttt{NatPrimeToIrred} indeed makes use of the inductively defined disjunction.

In Euclid’s Lemma, the premise is that a prime divides the product of two numbers. However, for the proof of uniqueness in the fundamental theorem of arithmetic, we require this statement for products of arbitrary length. 
Therefore, we proceed to prove a more general version:
\begin{theorem}[\Href{https://github.com/FranziskusWiesnet/MinlogArith/blob/1c63c302aa52e8972b594a3ef8753fa11a0435eb/minlog/examples/arith/prime_nat.scm\#L820}{NatPrimeDivProdPrimesToInPrimes},  \Href{https://github.com/FranziskusWiesnet/MinlogArith/blob/1c63c302aa52e8972b594a3ef8753fa11a0435eb/minlog/examples/arith/prime_pos.scm\#L1529}{PosPrimeDivProdPrimesToInPrimes}]
\label{Thm:PrimeDivProdPrimesToInPrimes}
\begin{align*}
\forall_{ps,p,n}\left( \Prime{p} \to \Primes{ps}{n} \to p\Bigm| \prod_{i<n}ps(i) \to \exl_{i<n} ps(i) = p\right).
\end{align*}
\end{theorem}
\begin{proof}
We proceed by induction on $n$ for the given $ps$ and $p$.
For the base case, since $\prod_{i<0}ps(i) = 1$ and $p$ is prime, the statement $p \mid 1$ is impossible. Thus, there is nothing to prove.

For the induction step, let $n$ be given with $\Primes{ps}{\Suc n}$ (i.e.~$\Primes{ps}{n}$ and $\Prime{ps\ n}$) and assume that
\[
p \Bigm| \prod_{i<\Suc n}ps(i) = \left(\prod_{i<n}ps(i)\right) \cdot ps(n).
\]
By Lemma \ref{Lem:PrimeToIrred}, it follows that either  
$p \mid \prod_{i<n}ps(i)$, or  
$p \mid ps(n)$.
In the first case, the claim follows directly by the induction hypothesis.  
In the second case, there exists some $q$ such that  
\[
p \cdot q = ps(n).
\]  
Since $ps(n)$ is prime, it must be that either $p = 1$ or $q = 1$.  
Because $p$ is prime, $p = 1$ is impossible. Hence, $q = 1$ and therefore $p = ps(n)$, as required.  
\end{proof}
The proof is essentially an iterative application of Lemma \ref{Lem:PrimeToIrred}. 
Therefore, the extracted term essentially performs a bounded search from above for some $i<n$ with $p\mid ps(i)$. 
\section{Fundamental Theorem of Arithmetic}
The existence and uniqueness of the prime factorisation of an integer are known as the fundamental theorem of arithmetic and constitutes one of the central results in elementary number theory. A computational analysis of this theorem will likewise be one of the main results of this article.
Instead of working with integers, we are dealing with natural numbers and positive binary numbers. Furthermore, we will divide the fundamental theorem into two parts: first, we prove the existence of a prime factorisation, and then its uniqueness.

\subsection{Existence of the Prime Factorisation}
The proof of uniqueness is very simple, but its computational content is extremely inefficient, as we will see.
\begin{theorem}[\Href{https://github.com/FranziskusWiesnet/MinlogArith/blob/1c63c302aa52e8972b594a3ef8753fa11a0435eb/minlog/examples/arith/fta_nat.scm\#L13}{NatExPrimeFactorisation}, \Href{https://github.com/FranziskusWiesnet/MinlogArith/blob/1c63c302aa52e8972b594a3ef8753fa11a0435eb/minlog/examples/arith/fta_pos.scm\#L17}{PosExPrimeFactorisation}]
\label{Thm:ExPrimeFac}
\begin{align*}
\forall_p \exd_{ps} \exl_m\left( \Primes{ps}{m} \andnc \prod_{i<m}ps(i)=p\right)
\end{align*}
\end{theorem}

\begin{proof}
We prove the equivalent statement:
\begin{align*}
\forall_{l,p}\left( p<l \to \exd_{ps} \exl_m\left( \Primes{ps}{m} \andnc \prod_{i<m}ps(i)=p\right)\right).
\end{align*}
by induction on $l$. 
For $l=0$, there is nothing to show. 

For the induction step, let $l$ and $p<\Suc l$ be given. If $p=1$, we choose $ps=\lambda_n 1$ and $m=0$. 

If \( p > 1 \), we know from Lemma~\ref{Lem:LeastFactorPrime} that $\lf p$ is prime. By Lemma~\ref{Lem:LeastFactorProp}, we have \( \lf p \mid p \) and \( \lf p > 1 \), so there exists some \( q < p \) such that \( (\lf p) \cdot q = p \). Since also \( q < l \), we may apply the induction hypothesis to \( q \), yielding a sequence \( qs \) and natural number \( n \) such that \( \Primes{qs}{n} \) and \( \prod_{i < n} qs(i) = q \).

We define \( m := \Suc n \) and construct \( ps \) as follows:
\begin{align*}
ps(i) := \begin{cases}
\lf p & \text{if } i=n, \\
qs(i) & \text{otherwise.}
\end{cases}
\end{align*}
From \( \Prime{\lf p} \) and \( \Primes{qs}{n} \), it follows that \( \Primes{ps}{m} \). Furthermore, we compute:
\begin{align*}
\prod_{i<m}ps(i) = \left(\prod_{i<n}qs(i)\right) \cdot \lf p = q \cdot \lf p = p.
\end{align*}
This completes the proof.
\end{proof}

\paragraph{The extracted term.}
The extracted term takes a positive binary number as input, and is essentially an iterated application of \(\lf\).  Its output is a pair \((f, n)\), where $f: \NN\to \PP$ is a sequence and $n$ a natural number, such that $f(0),\dots,f(n-1)$ are exactly the prime divisors of the input with multiplicity.

Its tests are documented in \href{https://github.com/FranziskusWiesnet/MinlogArith/blob/1c63c302aa52e8972b594a3ef8753fa11a0435eb/test-files/fta_pos_test.txt#L1}{\texttt{fta\_pos\_test.txt}} and summarised in Appendix \ref{Sec:AppPrimeFactor}. We tested the factorisation of prime numbers, powers of primes, and products of twin primes.

While the extracted term verified that 10\,000\,019 is prime in 2.39 seconds, it took 24.87 seconds for 120\,000\,007. Hence, there is a factor of about 12 between the two primes and also between the two runtimes. The other measurements also suggest a linear relationship between the runtime and the prime given as input, which means that the runtime grows exponentially with the number of digits of the prime. Since the extracted term is essentially an iterative application of $\lf$, and $\lf$ is basically trial division from below, this is what we expect.

Numbers with small prime factors can be factored much faster by this algorithm, as also confirmed by the tests on prime powers. For instance, even the number $3^{1000}$ can be decomposed into its 1000 prime factors within 28 seconds. $2^{1000}$ is even better handled, since $\lf$ detects this factor immediately via the case distinction for $p=\SZero p'$ (see  Definition \ref{Def:LeastFactor}).
For all other prime powers -- for example, in our test with 47 -- the runtime seems to increases between quadratic and cubic with the exponent. This is consistent with the fact that the number of applications of $\lf$ is exactly linear in the exponent, that in each step divisibility is checked only a constant number of times (exactly the number of base prime), and that this is done using Stein's algorithm, whose runtime is quadratic in the number of digits and thus quadratic in the exponent. However, the measured data together with the polynomial approximation does not yield any clear results here, so substantially more tests, also with other prime powers, would be necessary for a clear verification. Our tests with powers of $47$ actually suggest a quadratic runtime w.r.t.~the exponent, whereas the theoretical considerations rather point to a cubic runtime.

The algorithm takes particularly long when the prime factors are close to the square root of the input. To investigate this, we in particular ran tests on products of twin primes. Here we found that $100\,160\,063 =10\,007\cdot 10\,009$ takes almost twice as long to decompose as the larger prime $120\,000\,007$. This is because $\lf$ does divisibility tests only up to the square root of the input and, if no divisor is found, $\lf$ returns the input itself as the factor. In the case of products of primes, it then finds a divisor and still has to factor the complementary cofactor. Here, $\lf$ starts its search again from below, even though it would suffice to start from the factor that has already been found. In principle, this suggests a possible way to improve the algorithm further. Nevertheless, trial division is inherently very inefficient, so it probably makes more sense to focus entirely on other factorisation algorithms, which we do in Section \ref{Sec:Fermat}.

\subsection{Permutations}
Having established the existence of a prime factorisation, we now turn to its uniqueness. For this purpose, we require a precise formalisation of the notion of ``uniqueness''.
In the context of the fundamental theorem of arithmetic, uniqueness means that for any two factorisations of the same number, there exists a permutation mapping one factor sequence to the other. Therefore, we must first develop the necessary theory of permutations.

As we will see, permutations are a subject whose precise formalisation in a computer proof environment is significantly more complex than the somewhat informal treatment typically found in textbook proofs. This complexity becomes apparent already in the definition of a permutation -- specifically, in the choice of its type:

In standard textbooks, a permutation is defined as a bijective mapping on a finite set of the form \(\{1, \dots, n\}\). However, this approach leads to complications in our formal setting, as the type of the permutation would then depend on the natural number \(n\). Unlike proof assistants such as Agda or Lean, Minlog avoids the use of dependent types. As a result, such a definition becomes difficult to manage, particularly when reasoning across varying values of \(n\), which will be necessary in our formalisation.

For this reason, we define permutations as functions on all natural numbers that act as identity functions from a certain point onward. Furthermore, we explicitly include the inverse function in the definition, so that bijectivity can be stated directly via the existence of an explicit inverse.
\begin{definition}[\Href{https://github.com/FranziskusWiesnet/MinlogArith/blob/1c63c302aa52e8972b594a3ef8753fa11a0435eb/minlog/examples/arith/fta_nat.scm\#L122}{Pms}]
Given a natural number \( m \) and functions \( f, g : \NN \to \NN \), the predicate \( \Pms_m(f, g) \) expresses that \( f \) is a permutation of the numbers \( 0, \dots, m - 1 \) with inverse \( g \). 
Formally, \( \Pms \) is defined as an inductive predicate with the following introduction rule:
\begin{align*}
\Pms^+:\quad\forall_{m,f,g}\left( \forall_ng(f\, n)=n \to \forall_n f(g\, n)=n\to \forall_{n\geq m}f\, n = n\to \Pms_m(f,g)\right)
\end{align*}
The predicate $\Pms$ itself is defined as non-computational.
\end{definition}

In addition to the introduction rule above, there is also an elimination axiom, derived from the introduction rule, which is given by:
\begin{align*}
\Pms^-(X):\ \ &\forall_{m,f,g}\big(\forall_n g(f\,n)=n \to \forall_n f(g\,n)=n \to \forall_{n\geq m} f\,n = n \to X(m,f,g)\big) \to \\
 &\forall_{m,f,g}\left(\Pms_m(f,g) \to X(m,f,g)\right),
\end{align*}
where \( X \) is any non-computational predicate with the same arity as \( \Pms \).
For further details on inductively defined predicates, we refer the reader to \cite[Section 7.1.2]{schwichtenberg2012proofs}.
\begin{lemma}\label{Lem:PmsProps}
The universal closures of the following statements hold:
\renewcommand{\arraystretch}{1.2}
\begin{equation*}
\begin{array}{ll}
    \href{https://github.com/FranziskusWiesnet/MinlogArith/blob/1c63c302aa52e8972b594a3ef8753fa11a0435eb/minlog/examples/arith/fta_nat.scm#L148}{\mathtt{PmsCirc}}:
    & \Pms_m(f,g) \to \forall_n g(f\,n) = n \\
    \href{https://github.com/FranziskusWiesnet/MinlogArith/blob/1c63c302aa52e8972b594a3ef8753fa11a0435eb/minlog/examples/arith/fta_nat.scm#L157}{\mathtt{PmsCircInv}}:
     & \Pms_m(f,g) \to \forall_n f(g\,n) = n \\
    \href{https://github.com/FranziskusWiesnet/MinlogArith/blob/1c63c302aa52e8972b594a3ef8753fa11a0435eb/minlog/examples/arith/fta_nat.scm#L166}{\mathtt{PmsIdOut}}:
    & \Pms_m(f,g) \to \forall_{n\geq m} f\,n = n \\
    \href{https://github.com/FranziskusWiesnet/MinlogArith/blob/1c63c302aa52e8972b594a3ef8753fa11a0435eb/minlog/examples/arith/fta_nat.scm#L223}{\mathtt{PmsIdOutInv}}: & \Pms_m(f,g) \to \forall_{n\geq m} g\,n = n \\
    \href{https://github.com/FranziskusWiesnet/MinlogArith/blob/1c63c302aa52e8972b594a3ef8753fa11a0435eb/minlog/examples/arith/fta_nat.scm#L175}{\mathtt{PmsSucc}}: & \Pms_n(f,g) \to \Pms_{\Suc n}(f,g) \\
    \href{https://github.com/FranziskusWiesnet/MinlogArith/blob/1c63c302aa52e8972b594a3ef8753fa11a0435eb/minlog/examples/arith/fta_nat.scm#L196}{\mathtt{PmsSuccInv}}: & \Pms_{\Suc n}(f,g) \to f\, n = n \to \Pms_n(f,g) \\
    \href{https://github.com/FranziskusWiesnet/MinlogArith/blob/1c63c302aa52e8972b594a3ef8753fa11a0435eb/minlog/examples/arith/fta_nat.scm#L253}{\mathtt{PmsConcat}}: & \Pms_n(f_0,f_1) \to \Pms(g_0,g_1) \to \Pms_n(f_0 \circ g_0, g_1\circ f_1) \\
\end{array}
\end{equation*}
\end{lemma}
\begin{proof}
These statements follow naturally from the axioms $\Pms^+$ and $\Pms^-$, among other theorems. For further details, we refer to the Minlog code.
\end{proof}

A special class of permutations are transpositions, which essentially generate all permutations. We will therefore first prove the lemma concerning the invariance of the general product under permutations for transpositions (specifically, those involving the last element), and then generalise the result to arbitrary permutations.
This is one of those cases where a typical textbook proof is considerably easier to carry out than a fully precise formalisation.
\begin{definition}[\Href{https://github.com/FranziskusWiesnet/MinlogArith/blob/1c63c302aa52e8972b594a3ef8753fa11a0435eb/minlog/examples/arith/fta_nat.scm\#L282}{Transp}]
For natural numbers $n,m$ we define the transposition $\tau_{n,m}:\NN\to\NN$ by the rule
\begin{align*}
\tau_{n,m}(i) :=
\begin{cases}
m & \text{ if }i=n\\
n & \text{ if }i=m\\
i & \text{ otherwise.}
\end{cases}
\end{align*}
\end{definition}
\begin{lemma}[\Href{https://github.com/FranziskusWiesnet/MinlogArith/blob/1c63c302aa52e8972b594a3ef8753fa11a0435eb/minlog/examples/arith/fta_nat.scm\#L297}{PmsTransp}]
\label{Lem:PmsTransp}
\begin{align*}
\forall_m\forall_{n_0,n_1<m}\Pms_m(\tau_{n_0,n_1},\tau_{n_0,n_1})
\end{align*}
\end{lemma}
\begin{proof}
This follows directly from the introduction axiom of $\Pms$.
\end{proof}
\begin{lemma}[\Href{https://github.com/FranziskusWiesnet/MinlogArith/blob/1c63c302aa52e8972b594a3ef8753fa11a0435eb/minlog/examples/arith/fta_nat.scm\#L367}{NatProdInvTranspAux}, \Href{https://github.com/FranziskusWiesnet/MinlogArith/blob/1c63c302aa52e8972b594a3ef8753fa11a0435eb/minlog/examples/arith/fta_pos.scm\#L103}{PosProdInvTranspAux}]
\label{Lem:ProdInvTranspAux}
\begin{align*}
\forall_{m,n,ps}\left( n<m \to \prod_{i<\Suc m}ps(i) = \prod_{i<\Suc m}ps(\tau_{n,m}(i))\right)
\end{align*}
\end{lemma}
\begin{proof}
We proceed by induction on $m$. 
For $m=0$, there is no $n<m$, so the statement holds trivially.

For the induction step let $m, n$, and $ps$ be given with $n < \Suc m$. Our goal is to show that
\[
\prod_{i<\Suc (\Suc m)}ps(i) = \prod_{i<\Suc (\Suc m)}ps(\tau_{n,\Suc m}(i)).
\]
For the right-hand side, we have:
\begin{align*}
\prod_{i<\Suc (\Suc m)}ps(\tau_{n,\Suc m}(i)) = \left(\prod_{i<m}ps(\tau_{n,\Suc m}(i))\right) \cdot ps(m) \cdot ps(n).
\end{align*}
For the left-hand side, we obtain:
\[
\prod_{i<\Suc (\Suc m)}ps(i) = \left(\prod_{i<m}ps(i)\right) \cdot ps(m) \cdot ps(\Suc m).
\]
If $m=n$, then $\tau_{n,\Suc m}(i) = i$ for all $i<m$, and thus both sides are equal.  
Hence, we assume that $n<m$. We define a modified sequence $qs$ by
\[
qs(i) :=
\begin{cases}
ps(\Suc m), & \text{if } i = m, \\
ps(i), & \text{otherwise}.
\end{cases}
\]
Applying the induction hypothesis to $n$ and $qs$, we get:
\begin{align*}
\left(\prod_{i<m}ps(i)\right) \cdot ps(\Suc m) &= \prod_{i<\Suc m}qs(i) \\
&= \prod_{i<\Suc m}qs(\tau_{n,m}(i)) \\
&= \left(\prod_{i<m}qs(\tau_{n,m}(i))\right) \cdot ps(n).
\end{align*}
Thus, it remains to show that:
\[
\prod_{i<m}ps(\tau_{n,\Suc m}(i)) = \prod_{i<m}qs(\tau_{n,m}(i)).
\]
This follows from the fact that $ps(\tau_{n,\Suc m}(i)) = qs(\tau_{n,m}(i))$ for all $i<m$, which can be seen by the definition of $qs$, $\tau$, and by considering the cases $i=n$ and $i\neq n$.
\end{proof}

\begin{lemma}[\Href{https://github.com/FranziskusWiesnet/MinlogArith/blob/1c63c302aa52e8972b594a3ef8753fa11a0435eb/minlog/examples/arith/fta_nat.scm\#L497}{NatProdInvPms}, \Href{https://github.com/FranziskusWiesnet/MinlogArith/blob/1c63c302aa52e8972b594a3ef8753fa11a0435eb/minlog/examples/arith/fta_pos.scm\#L239}{PosProdInvPms}]
\begin{align*}
\forall_{n,f,g,ps}\left( \Pms_n(f,g) \to \prod_{i<n}ps(i) = \prod_{i<n}ps(f\,i)\right)
\end{align*}
\end{lemma}
\begin{proof}
We proceed by induction on $n$.  

For the base case $n=0$, there is nothing to prove.  

For the induction step, assume that $f,g,ps$, and $n$ are given such that $\Pms_{\Suc n}(f,g)$ holds. If $f(n) = n$, it follows that $\Pms_n(f,g)$ holds by \texttt{PmsSuccInv} from Lemma \ref{Lem:PmsProps}, and the claim follows directly by the induction hypothesis.  

Thus, we assume that $f\,n \neq n$, and therefore $f\,n<n$. By Lemma \ref{Lem:ProdInvTranspAux}, we obtain  
\begin{align}
\label{Eq:ProdInvPms0}
\prod_{i<\Suc n}ps(i) = \prod_{i<\Suc n}ps(\tau_{f\, n,n}(i)) = \left(\prod_{i<n}ps(\tau_{f\,n,n}(i))\right) \cdot ps(f\,n).
\end{align}
Furthermore, by \texttt{PmsConcat} from Lemma \ref{Lem:PmsProps} and Lemma \ref{Lem:PmsTransp}, we have  
\[
\Pms_{\Suc n}(\tau_{f\,n,n} \circ f, g \circ \tau_{f\,n,n}),
\]
and since $(\tau_{f\,n,n} \circ f)(n) = n$, applying \texttt{PmsSuccInv} from Lemma \ref{Lem:PmsProps} yields  
\[
\Pms_n(\tau_{f\,n,n} \circ f, g \circ \tau_{f\,n,n}).
\]
Thus, we apply the induction hypothesis to $\tau_{f\,n,n} \circ f$, $g \circ \tau_{f\,n,n}$, and $ps\circ \tau_{f\,n,n}$, obtaining  
\begin{align*}
\prod_{i<n}(ps\circ\tau_{f\,n,n})(i)) = \prod_{i<n}(ps\circ\tau_{f\,n,n} \circ \tau_{f\,n,n} \circ f)(i) = \prod_{i<n}ps(f\,i).
\end{align*}
Continuing from Equation \eqref{Eq:ProdInvPms0}, we conclude  
\begin{align*}
\prod_{i<\Suc n}ps(i) &= \left(\prod_{i<n}ps(\tau_{f\,n,n}(i))\right) \cdot ps(f\,n) \\
&= \left(\prod_{i<n}ps(f\,i)\right) \cdot ps(f\,n) \\
&= \prod_{i<\Suc n}ps(f\,i),
\end{align*}
which completes the proof.
\end{proof}

We also need the predicate \(\Primes{\cdot} {\cdot}\) to be invariant under permutations. However, it suffices to establish this property for the specific case of transpositions involving the last element:
\begin{lemma}[\Href{https://github.com/FranziskusWiesnet/MinlogArith/blob/1c63c302aa52e8972b594a3ef8753fa11a0435eb/minlog/examples/arith/fta_nat.scm\#L602}{NatPrimesInvTranspAux}, \Href{https://github.com/FranziskusWiesnet/MinlogArith/blob/1c63c302aa52e8972b594a3ef8753fa11a0435eb/minlog/examples/arith/fta_pos.scm\#L354}{PosPrimesInvTranspAux}]
\label{Lem:PrimesInvTranspAux}
\begin{align*}
\forall_{m,n,ps}\left( n<m \to \Primes{ps}{m} \to \Primes{ps\circ\tau_{n,m}}{m}\right)
\end{align*}
\end{lemma}
\begin{proof}
The proof is similar to the proof of Lemma \ref{Lem:ProdInvTranspAux}.
\end{proof}
\subsection{Uniqueness of the Prime Factorisation}
With this theory of permutations in place, we are now ready to formulate and prove the uniqueness part of the fundamental theorem of arithmetic:
\begin{theorem}[\Href{https://github.com/FranziskusWiesnet/MinlogArith/blob/1c63c302aa52e8972b594a3ef8753fa11a0435eb/minlog/examples/arith/fta_nat.scm\#L746}{NatPrimeFactorisationsToPms}, \Href{https://github.com/FranziskusWiesnet/MinlogArith/blob/1c63c302aa52e8972b594a3ef8753fa11a0435eb/minlog/examples/arith/fta_pos.scm\#L504}{PosPrimeFactorisationsToPms}]
\label{Thm:PrimeFactorisationsToPms}
\begin{align*}
\forall_{n,m,ps,qs}\biggl(\ & \Primes{ps}{n} \to \Primes{qs}{m} \to \prod_{i<n}ps(i)=\prod_{i<m}qs(i)\to\\ & n=m\ \andr\ \exd_f \exl_g\left(\Pms_n(f,g) \andnc \forall_{i<n} ps(f\,i) = qs(i)\right)\biggl).
\end{align*}
\end{theorem}

\begin{proof}
We proceed by induction on $n$.

For $n=0$, we have $1=\prod_{i<n}ps(i)=\prod_{i<m}qs(i)$, which is only possible if $m=0$. Hence, choosing $f=g=\lambda_n n$ proves the claim.

Assume the statement holds for some $n$ and let $m', ps, qs$ be given such that
\begin{align*}
\Primes{ps}{\Suc n},\quad \Primes{qs}{m'}\quad \text{and}\quad \prod_{i<\Suc n}ps(i)=\prod_{i<m'}qs(i).
\end{align*}
If $m'=0$, then $1 < ps(n) \leq \prod_{i<\Suc n}ps(i) = \prod_{i<m'}qs(i) = 1$, which is a contradiction. Thus, $m' = \Suc m$ for some $m$.
In particular, we have
\begin{align*}
qs(m) \Bigm| \prod_{i<\Suc m}qs(i)=\left(\prod_{i<n}ps(i)\right) \cdot ps(n).
\end{align*}
By Lemma \ref{Lem:PrimeToIrred}, we conclude that either
\begin{align}\label{Form:DivCases}
qs(m) \Bigm| \prod_{i<n}ps(i)\quad \text{or}\quad qs(m) \mid ps(n).
\end{align}
In the latter case, since both numbers are prime, it follows that $qs(m) = ps(n)$. Thus,
\begin{align*}
\prod_{i<n}ps(i) = \prod_{i<m}qs(i),
\end{align*}
which allows us to apply the induction hypothesis to $m, ps, qs$ and thus complete the proof in this case.

Hence we may assume $qs(m) \bigm| \prod_{i<n}ps(i)$. By Theorem \ref{Thm:PrimeDivProdPrimesToInPrimes}, there exists some $l<n$ such that $ps(l)=qs(m)$. We define
\begin{align}
\label{Form:rs}
rs:=ps\circ\tau_{l,n},
\end{align}
and by Lemma \ref{Lem:ProdInvTranspAux} we obtain
\begin{align*}
\left(\prod_{i<n}rs(i)\right) \cdot ps(l) = \prod_{i<\Suc n}rs(i) = \prod_{i<\Suc n}ps(i) = \prod_{i<\Suc m}qs(i) = \left(\prod_{i<m}qs(i)\right) \cdot qs(m).
\end{align*}
Since $ps(l) = qs(m)$, we conclude that
\begin{align*}
\prod_{i<n}rs(i) = \prod_{i<m}qs(i).
\end{align*}
Furthermore, $\Primes{qs}{m}$ follows from $\Primes{qs}{m'}$, and $\Primes{rs}{n}$ follows from $\Primes{ps}{n}$ and Lemma~\ref{Lem:PrimesInvTranspAux}. Thus, applying the induction hypothesis to $m$, $rs$, and $qs$, we obtain $m=n$ and functions $f,g$ with
\begin{align}
\label{Eq:PrimeFactorisationsToPms1}
\Pms_n(f,g) \quad \text{and} \quad \forall_{i<n} rs(f\,i) = qs(i).
\end{align}
Since $\Suc n = \Suc m = m'$, it remains to construct functions $f', g'$ such that $\Pms_{\Suc n}(f',g')$ and $\forall_{i<\Suc n} ps(f'\,i) = qs(i)$. For this we define
\begin{align*}
f':= \tau_{l,n} \circ f, \qquad g':= g \circ \tau_{l,n}.
\end{align*}
By $\Pms_n(f,g)$ and Lemmas \ref{Lem:PmsProps} and \ref{Lem:PmsTransp}, we obtain $\Pms_{\Suc n}(f',g')$. Finally, for $i<\Suc n$, we show that $ps(f'\,i) = qs(i)$.
If $i<n$, then
\begin{align*}
ps(f'i) = ps(\tau_{l,n}(f\,i)) = rs(f\, i) = qs(i)
\end{align*}
by Formula \eqref{Form:rs} and Formula \eqref{Eq:PrimeFactorisationsToPms1}.
Furthermore, by $\Pms_n(f,g)$ and Lemma \ref{Lem:PmsProps} (specifically \texttt{PmsIdOut}), we have $f(n) = n$. Thus,
\begin{align*}
ps(f'\,n) = ps(\tau_{l,n}(n)) = ps(l) = qs(m).
\end{align*}
This completes the proof.
\end{proof}
\paragraph{The extracted term.}
The type of the extracted term is given by
\begin{align*}
\operatorname{genPms} : \NN \to \NN \to (\NN \to \PP) \to (\NN \to \PP) \to 
           (\NN \to \NN)\times(\NN \to \NN).
\end{align*}
For the natural numbers \(n\) and \(m\), and two prime factorisations given by \(ps\) and \(qs\) with length \(n\) and \(m\), respectively, as in the theorem above, the extracted term returns \(f\) and \(g\) with $\Pms_n(f,g)$. If $n=m$ and $ps(0),\dots,ps(n-1)$ and $qs(0),\dots,qs(m-1)$ contains the same numbers, then $ps(f(i)) = qs(i)$ for all $i\in\{0,\dots, n-1\}$.
In \href{https://github.com/FranziskusWiesnet/MinlogArith/blob/1c63c302aa52e8972b594a3ef8753fa11a0435eb/test-files/fta_pos_test.txt#L219}{\texttt{fta\_pos\_test.txt}}, we documented several tests of the extracted term. 

Essentially, the algorithm starts with $qs(m-1)$ and searches for some $i < n$ such that $qs(m-1)$ divides $ps(i)$. In doing so, the search begins almost from the end, as we explain in the following example:  For test reasons we had a look at the case $n=m=10$ and $ps=qs=\lambda_i2$. Therefore we entered 
\begin{verbatim}
let ps n = 2
let (f,g) = genPms 10 10 ps ps
map f [0..9]
\end{verbatim}
and got 
\begin{verbatim}
[1,2,3,4,5,6,7,8,9,0]
\end{verbatim}
as output. In particular, $f(n)=n+1$ for $i<9$ and $f(9)=0$. This is due to the case distinction in Formula (\ref{Form:DivCases}). First, the algorithm considers whether $qs(m)$ divides $\prod_{i<n} ps(i)$, and only if this is not the case, it checks whether $qs(m)$ divides $ps(n)$ instead. In the first case, the extracted term from Theorem \ref{Thm:PrimeDivProdPrimesToInPrimes} is applied, which indeed performs a search for an $i$ with $qs(m)\mid ps(i)$. Note that only $qs(m)\mid ps(i)$ is checked, not whether $qs(m)=ps(i)$. Hence, the following input leads to the same output as before:
\begin{verbatim}
let ps n = 4
let qs n = 2
let (f,g) = genPms 10 10 ps qs
map f [0..9]
\end{verbatim}
However, if we swap $ps$ and $qs$ by entering 
\begin{verbatim}
let ps n = 2
let qs n = 4
let (f,g) = genPms 10 10 ps qs
map f [0..9]
\end{verbatim}
we obtain a completely different output:
\begin{verbatim}
[0,2,3,4,5,6,7,8,9,1]
\end{verbatim}
Although the first case in Formula (\ref{Form:DivCases}) holds and therefore the extracted term of Theorem \ref{Thm:PrimeDivProdPrimesToInPrimes} is applied, that term does not find an $i<n$ with $qs(m)\mid ps(i)$ and returns $1$ in this case as it is the smallest $n$ with $qs(m)=4\Bigm|  \prod_{i<\Suc n}ps(i)$. 
If $qs(m)$ is instead a number that does not divide $\prod_{i<\Suc m}ps(i)$ at all, the algorithm stops immediately. Therefore, entering 
\begin{verbatim}
let ps n = 2
let qs n = 3
let (f,g) = genPms 10 10 ps qs
map f [0..9]
\end{verbatim}
leads to the output
\begin{verbatim}
[0,1,2,3,4,5,6,7,8,9].
\end{verbatim}
Even though these were cases where the assumptions were not satisfied, the examples illustrate very clearly how the algorithm operates and how it is linked to the proof. To analyse the runtime a bit further, we also carried out tests in which we considered a random permutation of the first $n\in\{0,50,100,\dots,500\}$ prime numbers in each case. The results are presented in tabular form in Appendix \ref{Sec:AppUniquess}. Since the extracted term is essentially a iterative search algorithm, a polynomial runtime is to be expected --  probably quadratic or cubic to compute a single value $f(i)$, and one degree higher to compute $f(0),\dots,f(n-1)$ as in our tests. Since a number with $n$ binary digits can have at most $n$ prime factors (with multiplicity), the algorithm is presumably efficient relative to the size of the factored number.  However, the measured results do not provide a clear conclusion here. The runtime depends in particular on the sizes of the individual primes and, in particular, on the specific permutation. Therefore, substantially more tests would be required to make a statistically and numerically significant statement here.
\section{Applications of the Fundamental Theorem of Arithmetic}
The following theorem is an important application of the fundamental theorem of arithmetic, which appears in many proofs -- often in the form of Corollary \ref{Cor:FTAapp}.
In our context, this is meant simply as an application of the fundamental theorem of arithmetic, demonstrating how its computational content can be handled in practice. 
\begin{theorem}[\Href{https://github.com/FranziskusWiesnet/MinlogArith/blob/1c63c302aa52e8972b594a3ef8753fa11a0435eb/minlog/examples/arith/fta_nat.scm\#L1422}{NatProdEqProdSplit}, \Href{https://github.com/FranziskusWiesnet/MinlogArith/blob/1c63c302aa52e8972b594a3ef8753fa11a0435eb/minlog/examples/arith/fta_pos.scm\#L1171}{PosProdEqProdSplit}]
\label{Thm:PosProdEqProdSplit}
\begin{align*}
\forall_{p_0,p_1,q_0,q_1}\bigl(\ &p_0\cdot p_1 = q_0\cdot q_1 \to \\ 
&\exd_{r_0,r_1,r_2,r_3}\left( p_0=r_0\cdot r_1 \andnc p_1=r_2\cdot r_3 \andnc q_0=r_0\cdot r_2 \andnc q_1=r_1\cdot r_3\right)\bigr).
\end{align*}
\end{theorem}
\begin{proof}
Let \((ps_0, m_0), (ps_1, m_1), (qs_0, n_0)\), and \((qs_1, n_1)\) be the prime factorisations of \(p_0, p_1, q_0\), and \(q_1\), respectively. In particular, we have
\begin{align*}
\prod_{i<m_0} ps_0(i) &= p_0, & \prod_{i<m_1} ps_1(i) &= p_1, \\
\prod_{i<n_0} qs_0(i) &= q_0, & \prod_{i<n_1} qs_1(i) &= q_1.
\end{align*}
We now define the sequence \(ps\) by
\[
ps(i) = 
\begin{cases}
ps_0(i) & \text{for } i < m_0, \\
ps_1(i - m_0) & \text{for } m_0 \leq i < m_0 + m_1.
\end{cases}
\]
Similarly, we define the sequence \(qs\) as
\[
qs(i) = 
\begin{cases}
qs_0(i) & \text{for } i < n_0, \\
qs_1(i - n_0) & \text{for } n_0 \leq i < n_0 + n_1.
\end{cases}
\]
Then, we obtain  
\[
\prod_{i<m_0+m_1} ps(i) = p_0\cdot p_1 = q_0\cdot q_1 = \prod_{i<n_0+n_1} qs(i).
\]
Since these are two prime factorisations of the same number, Theorem~\ref{Thm:PrimeFactorisationsToPms} implies that \(m_0 + m_1 = n_0 + n_1 =: N\) and that there exist permutations \(f\) and \(g\) such that \(\Pms_N(f, g)\) holds and
\[
\forall_{i < N} \quad ps(f(i)) = qs(i).
\]
Now, we define the sequences $rs_0$ and $rs_1$ by
\begin{align*}
rs_0(i) &:= 
\begin{cases}
ps(i) & \text{if } g(i) < n_0, \\
1 & \text{otherwise},
\end{cases} \\
rs_1(i) &:= 
\begin{cases}
ps(i) & \text{if } n_0 \leq g(i), \\
1 & \text{otherwise}.
\end{cases}
\end{align*}
Then, using formally \texttt{PosProdSeqFilterProp0}, we obtain:
\begin{align*}
\prod_{i<N} rs_0(i) &= \prod_{\substack{i<N \\ g(i) < n_0}} ps(f(g\ i)) = \prod_{\substack{i<N \\ g(i) < n_0}} qs(g\, i) \\
&= \prod_{i<n_0} qs(i) = q_0,
\end{align*}
and similarly, using formally \texttt{PosProdSeqFilterProp1}, 
\begin{align*}
\prod_{i<N} rs_1(i) &= \prod_{\substack{i<N \\ n_0 \leq g(i)}} ps(f(g\, i)) = \prod_{\substack{i<N \\ n_0 \leq g(i)}} qs(g\, i) \\
&= \prod_{n_0 \leq i < n_0+n_1} qs(i) = q_1.
\end{align*}
Furthermore, we have:
\begin{align*}
\prod_{i<m_0} rs_0(i) \cdot \prod_{i<m_0} rs_1(i) &= \prod_{\substack{i<m_0 \\ g(i)<n_0}} ps(i) \cdot \prod_{\substack{i<m_0 \\ n_0 \leq g(i)}} ps(i) \\
&= \prod_{i<m_0} ps(i) = p_0,
\end{align*}
and analogously:
\begin{align*}
\prod_{\substack{m_0\leq i \\ i<m_0+m_1}} rs_0(i) \cdot \prod_{\substack{m_0\leq i \\ i<N}} rs_1(i) &= \prod_{\substack{m_0\leq i \\ i<N \\ g(i)<n_0}} ps(i) \cdot \prod_{\substack{m_0\leq i \\ i<N \\ n_0 \leq g(i)}} ps(i) \\
&= \prod_{\substack{m_0\leq i \\ i<N}} ps(i) = p_1.
\end{align*}
We now define: 
\begin{align*}
r_0 &:= \prod_{i<m_0} rs_0(i), & r_1 &:= \prod_{i<m_0} rs_1(i), \\
r_2 &:= \prod_{\substack{m_0\leq i \\ i<N}} rs_0(i), & r_3 &:= \prod_{\substack{m_0\leq i \\ i<N}} rs_1(i),
\end{align*}
which satisfy the desired properties as established by the four equations above.
\end{proof}

\paragraph{The extracted term.}
The proof of the theorem required a prime factorisation four times which then are compared sorted by the computational content of the uniqueness of the prime factorisation. Therefore, the computational content of the theorem is quite inefficient. Tests given in \href{https://github.com/FranziskusWiesnet/MinlogArith/blob/1c63c302aa52e8972b594a3ef8753fa11a0435eb/test-files/fta_pos_test.txt#L403}{\texttt{fta\_pos\_test.txt}}, but they also show that the extracted term is quite usable for small examples.
The function is defined as  
\begin{align*}
\operatorname{prodSplit}: \PP \to \PP \to \PP \to  \PP \to \PP\times\PP\times\PP\times\PP,
\end{align*}  
and takes four numbers \( p_0, p_1, q_0, q_1 \) as input, returning four numbers \( r_0, r_1, r_2, r_3 \) that satisfy the properties stated in the theorem above.
When executing  
\begin{verbatim}
prodSplit 7921 676 2314 2314
\end{verbatim}  
in the generated Haskell program, the output \texttt{((89,(89,(26,26)))} is returned almost instantaneously (0.10 seconds in our test).
However, if we use numbers with a few more digits, obtaining a result takes noticeably longer. For example executing 
\begin{verbatim}
prodSplit 37921088150 671104993 22439775070 1134103685
\end{verbatim} 
takes already 78.09 seconds in our test to produce \texttt{(211690,(179135,(106003,6331)))} as output.

But even though this theorem itself has inefficient computational content, it is still very useful, especially when we use it to prove a statement that has no computational content, such as the following corollary:
\begin{corollary}[\Href{https://github.com/FranziskusWiesnet/MinlogArith/blob/1c63c302aa52e8972b594a3ef8753fa11a0435eb/minlog/examples/arith/fta_nat.scm\#L1572}{NatGcdEqOneDivProdToDiv}, \Href{https://github.com/FranziskusWiesnet/MinlogArith/blob/1c63c302aa52e8972b594a3ef8753fa11a0435eb/minlog/examples/arith/fta_pos.scm\#L1309}{PosGcdEqOneDivProdToDiv}]
\label{Cor:FTAapp}
\begin{align*}
\forall_{p,q_0,q_1}\left( \gcd(p,q_0)=1 \to p\mid q_0\cdot q_1 \to p\mid q_1\right)
\end{align*}
\end{corollary}
\begin{proof}
Since $p\mid q_0\cdot q_1$, there exists some $p_0$ such that  
$$ p\cdot p_0 = q_0\cdot q_1. $$  
By Theorem \ref{Thm:PosProdEqProdSplit}, there exist positive numbers $r_0, r_1, r_2, r_3$ such that  
\begin{equation*}
\begin{array}{@{\;\;}r@{\;}l@{\;\;}r@{\;}l@{\;\;}}
p   & = r_0\cdot r_1,\qquad  & p_0 & = r_2\cdot r_3, \\  
q_0 & = r_0\cdot r_2,\qquad  & q_1 & = r_1\cdot r_3.  
\end{array}
\end{equation*}
Since $\gcd(p, q_0) = 1$ is given, it follows that $r_0 = 1$.  
Thus, we obtain $p = r_1$, and consequently,  
$
 p\mid r_1\cdot r_3 = q_1. \hfill 
$
\end{proof}

As one can easily verify, the corollary above is a generalisation of Lemma \ref{Lem:PrimeToIrred}.
\section{Fermat's Factorisation Method}
\label{Sec:Fermat}
In the existence proof of prime factorisation and the definition of a prime number itself, we have already seen that factoring a number is very inefficient. In fact, finding an efficient factorisation algorithm remains a well-known open problem.

In this article, by defining the smallest factor via $\lf$, we have searched for a factor by simply testing from the bottom up. This method is known as trial division.

To demonstrate the potential of formalising number theory in proof assistants such as Minlog, we would like to present another factorisation algorithm: Fermat's factorisation method. Although this method is not more efficient than trial division, it starts the search from a different point, i.e.~the integral square root. Therefore, this method is especially effective for numbers with factors close to each other.
Furthermore, in the next section, we will outline that this method can be generalised to the quadratic sieve method, which is one of the most efficient factorisation algorithms currently known.

Fermat's factorisation method is based on the formula $x^2-y^2=(x+y)\cdot (x-y)$. If one finds a representation of a number as a difference of squares, it directly yields a factorisation. If neither of the two factors is 1, this results in the desired decomposition. As we have seen, the computation of the integer square root function is already very efficient, which therefore also applies to checking whether a number is a perfect square. In particular, the property for a positive binary number $p$ being a perfect square is defined as follows:
\begin{align*}
\operatorname{IsSq}\ p\quad  :=\quad  (\mathtt{FastSqrt}\ {p} )^2 = p
\end{align*}
Note that we use the function \href{https://github.com/FranziskusWiesnet/MinlogArith/blob/1c63c302aa52e8972b594a3ef8753fa11a0435eb/minlog/examples/arith/factor_pos.scm#L17}{\texttt{FastSqrt}} from Definition \ref{Def:FastSqrt} in this section instead of $\lfloor\sqrt{\cdot} \rfloor$, since it has a shorter runtime and our extracted algorithm relies on many calls to the square-root function.

We first show that we do not search in vain when dealing with odd numbers. That means, if an odd number is factorable, this search will always yield a result. Conversely, this means that if the search is unsuccessful, the number must be a prime.
\begin{lemma}[\Href{https://github.com/FranziskusWiesnet/MinlogArith/blob/1c63c302aa52e8972b594a3ef8753fa11a0435eb/minlog/examples/arith/factor_nat.scm\#L47}{NatOddProdToDiffSq}]
\begin{align*}
\forall_{n}\forall_{l_0,l_1>1}\biggl(\ &\operatorname{NatEven} (\Suc n) \to n=l_0\cdot l_1 \to\\ &\exnc_{m_0,m_1}\left(\ m_0<m_1<\frac{n-1}{2}\andnc n=m_1^2-m_0^2\right) \biggr)
\end{align*}
\end{lemma}
\begin{proof}
The proof of this lemma is similar to the proof of the next lemma. For more details we refer to the Minlog code. 
\end{proof}

\begin{lemma}[\Href{https://github.com/FranziskusWiesnet/MinlogArith/blob/1c63c302aa52e8972b594a3ef8753fa11a0435eb/minlog/examples/arith/factor_pos.scm\#L428}{PosOddProdToDiffSq}]
\label{PosOddProdToDiffSqNc}
\begin{align*}
\forall_{p}\forall_{q_0,q_1>1}\bigl(\ &\SOne p = q_0\cdot q_1 \to q_0 \neq q_1 \to\\ 
&\exnc_{p_0,p_1}\left(\ p_0<p_1<p \andnc \SOne p = p_1^2-p_0^2\right)\bigr)
\end{align*}
\end{lemma}

\begin{proof}
We assume without loss of generality that $q_0<q_1$, as the case $q_1<q_0$ follows analogously.  
Since $\SOne p = 2p+1 = q_0\cdot q_1$ is odd, it follows that both $q_0$ and $q_1$ must also be odd. In particular, we can write  
$$ q_0 = 2r_0+1 \quad \text{and} \quad q_1 = 2r_1+1 $$  
for some $r_0, r_1$ of type $\PP$.  
We now define  
\begin{align*}
p_0 := \frac{q_1-q_0}{2} = r_1 - r_0 > 0, \qquad  
p_1 := \frac{q_1+q_0}{2} = r_1 + r_0 +1 > 0.
\end{align*}
Then, we obtain  
$$ p_1^2 - p_0^2 = q_0\cdot q_1 = \SOne p, $$  
and clearly, $p_0 < p_1$.  

Furthermore, since $q_0, q_1$ are odd and $q_0, q_1 > 1$, it follows that $q_0, q_1 > 2$. Thus, we also have $q_0, q_1 < p$, which implies $p_1 < p$.
\end{proof}

Since, in the following theorem, we only use the above lemma to prove a negative statement, the computational content of the lemma is not relevant. However, one could also consider formulating the lemma with a cr existential quantifier. In that case, the extracted term would essentially be given by the definitions of $p_0$ and $p_1$ as in the proof above.

We now aim to obtain Fermat's factorisation method as an extracted term from the proof of the following theorem. The statement of the theorem is, in essence, weaker than the existence part of the fundamental theorem of arithmetic (Theorem~\ref{Thm:ExPrimeFac}). However, this is a case where the specific proof of the theorem is crucial, as it directly determines the extracted algorithm.

\begin{theorem}[\Href{https://github.com/FranziskusWiesnet/MinlogArith/blob/1c63c302aa52e8972b594a3ef8753fa11a0435eb/minlog/examples/arith/factor_nat.scm\#L239}{NatPrimeOrComposedFermat}, \Href{https://github.com/FranziskusWiesnet/MinlogArith/blob/1c63c302aa52e8972b594a3ef8753fa11a0435eb/minlog/examples/arith/factor_pos.scm\#L631}{PosPrimeOrComposedFermat}]
\label{Thm:Fermat}
\begin{align*}
\forall_{p>1}\left( \Prime{p} \orr \exd_{q_0>1} \exl_{q_1>1}p=q_0\cdot q_1\right)
\end{align*}
\end{theorem}
\begin{proof}
For $p>1$, we have $p=\SZero q$ or $p=\SOne q$. In the first case, either $p=2$, which implies $\Prime{p}$, or $p = 2\cdot q$ with $q>1$. In both cases, the statement holds.  

Now, let us consider the case where $p=\SOne q = 2q +1$. If $p=(\mathtt{FastSqrt}\ p)^2$, the proof is complete. Therefore, we assume $(\mathtt{FastSqrt}\ p)^2 < p$.

Furthermore, if $q\leq 2$, then $p$ is either $3$ or $5$, both of which are prime. Thus, we also assume $q>2$, which implies $\lfloor\sqrt p \rfloor<q$ (proven as \texttt{FastSqrtSOneBound} in Minlog).  

Next, we define  
$$
l:= \mu_{(\mathtt{FastSqrt}\ p)   \leq i < q}\left(\IsSquare{i^2-p}\right).
$$
If $l=q$, then there is no $i<q$ such that $\mathtt{FastSqrt}\ p\leq i$ and $\IsSquare{i^2-p}$. Since $\mathtt{FastSqrt}\ p\leq i$ is a necessary condition for $i^2-p\geq 0$, it follows that no $i<q$ exists with $\IsSquare{i^2-p}$.  
By Lemma \ref{PosOddProdToDiffSqNc} and the assumption that $p$ is not a square, there exist no $q_0,q_1>1$ such that $p=q_0\cdot q_1$, which implies that $\Prime p$.  

Thus, we assume that $l<q$. Then there exists some $r$ such that $l^2-p=r^2$, i.e.,  
$$p=l^2-r^2=(l+r)\cdot (l-r).$$
As $p$ is not zero, we have $l-r \neq 0$. It remains to show that $l-r\neq 1$. Suppose, for contradiction, that $l-r=1$. Then we have $l=r+1$ and consequently $2q+1=p=2r+1$. This implies $q=r$, which leads to $l=r+1>q$, contradicting to the definition of $l$.
\end{proof}
\paragraph{The extracted term.}
Generally, the extracted algorithm, which we call \emph{Fermat algorithm} in the following, is inefficient, as it again relies on a bounded search, one that examines even more numbers than trial division. However, its runtime is short for numbers whose factors are close to each other. 
Thus, Fermat’s algorithm is essentially complementary to trial division, which is done by the function $\lf$ and the computational content of Theorem \ref{Thm:ExPrimeFac}.

Tests using the extracted Haskell program given in \href{https://github.com/FranziskusWiesnet/MinlogArith/blob/main/test-files/factor_pos.hs}{\texttt{factor\_pos.hs}} are documented in \href{https://github.com/FranziskusWiesnet/MinlogArith/blob/main/test-files/factor_pos_test.txt}{\texttt{factor\_pos\_test.txt}}.

While products of twin primes had the longest runtime for the extracted term from Theorem 5, as shown in Appendix A.5, Fermat’s algorithm needs only a single attempt for such inputs and can therefore factor them instantly. Even  the product of the two large twin primes  $299\,686\,303\,457$ and $299\,686\,303\,459$ were decomposed in 0.01 seconds.

To verify that those two numbers are prime, however, the Fermat algorithm takes far too long, so we did not execute this computation.
Even verifying that the number $89\,917$ is a prime took 101.20 seconds, whereas the extracted term of Theorem \ref{Thm:ExPrimeFac} only needed 0.06 seconds.

A noteworthy result, however, is that the Fermat algorithm is actually able to factorise the number 
$$810\,450\,000\,160\,224\,500\,006\,321=900\,500\,000\,129\cdot 900\,000\,000\,049$$
in about one minute (62.08 seconds in our test)  even though the individual factors are quite far apart.
Using trial division, we had no chance to factorise this number within a reasonable time.

To determine whether using \texttt{FastSqrt} instead of \texttt{PosSqrt} makes a difference, we replaced \texttt{FastSqrt} by \texttt{PosSqrt} in the definition of the extracted Haskell term fermat, thereby obtaining a new term \texttt{fermatPosSqrt}. The resulting Haskell program is saved as \href{https://github.com/FranziskusWiesnet/MinlogArith/blob/main/test-files/factor_pos_added.hs}{\texttt{factor\_pos\_added.hs}}. We then carried out all tests of the Fermat algorithm using \texttt{fermatPosSqrt} as well, and documented the results it in the same file.
It turns out that \texttt{fermat} is indeed faster than \texttt{fermatPosSqrt}, but the difference is very small. 

Factoring the number 
$810\,450\,000\,160\,224\,500\,006\,321$  was only marginally slower with the function \texttt{fermatPosSqrt}, taking 62.98 seconds. Likewise, determining that 89\,917 is prime did not take substantially longer either, with the modified algorithm requiring 102.40 seconds.
Nevertheless, in every test the Fermat algorithm was consistently faster, so using \texttt{FastSqrt} instead of \texttt{PosSqrt} is a clear improvement. Such small improvements can also accumulate quickly, so building on this article one may be able to find further optimisations for the algorithms extracted here. This brings us to the concluding section of this article.
\section{Outlook}
In this article, we have explored how program extraction from proofs in elementary number theory can be effectively realised within the Minlog system. To conclude, we outline a few directions  in which this foundational work might be extended. These arise in particular also with regard to work in other proof assistants, as shown in Section \ref{Sec:Novelties}.
\subsection{Subexponential Factoring Algorithms}
The two factorisation algorithms we have considered -- trial division and Fermat's algorithm -- have exponential runtime but are conceptually simple.  A natural next step would be to extend our approach to more complex yet more efficient factorisation algorithms. A well-structured overview of subexponential factorisation methods can be found in \cite{crandall2005prime, forster2015algorithmische, menezes2018handbook}. To the best of our knowledge, no subexponential factorisation method has yet been implemented in a proof assistant.

One particularly interesting method is the quadratic sieve, introduced by Carl Pomerance \cite{pomerance1982analysis}, as it builds upon Fermat's factorisation method. Consequently, a logical extension of this work would be to formalise a theorem in Minlog such that the extracted term corresponds to the quadratic sieve algorithm. In addition to Theorem \ref{Thm:Fermat}, Corollary \ref{Cor:FTAapp} also provides a solid foundation and is expected to play a role in the proof.

However, a key challenge lies in the fact that the quadratic sieve requires solving linear systems of equations. While this is not particularly difficult in itself, the necessary theoretical framework has not yet been implemented in Minlog. Therefore, this groundwork must be established before the quadratic sieve can be implemented.
\subsection{Primality Tests}
Primality tests are closely related to factorisation algorithms and are also of significant practical interest, and in contrast to factorisation algorithms, primality tests are present in some proof assistants, e.g.~\cite{caprotti2001formal,hurd2003verification, thery2007primality}. These tests determine only whether a number is prime or composite; however, in the case of a composite number, they do not necessarily reveal its factors.

Due to their speed, probabilistic primality tests (\cite[Section 3.5]{crandall2005prime}, \cite[Chapter 12]{forster2015algorithmische}) are widely used in practice. A number that passes these tests is only prime with high probability. On the other hand, a number that fails these tests is certainly composite.
One of the  most well-known examples is the Miller-Rabin test \cite{hurd2003verification,miller1975riemann,rabin1980probabilistic}.
From a computational perspective, expressing the statement that a given number is prime with a certain probability as an extracted term is challenging yet intellectually stimulating. In particular, the realisability predicate should be extended to incorporate a notion of probability.

\subsection{Development of a Computer Algebra System with Verified Algorithms}
Throughout this article, we have also examined the extracted terms as Haskell programs. Rather than considering these programs individually, it would be particularly interesting to integrate them into a modular computer algebra system that, unlike other such systems, operates exclusively using verified algorithms.

To automatically generate a Haskell file containing all extracted programs, it would likely be necessary to implement an additional function in Minlog that directly stores the extracted term in a Haskell file during its creation.

In the context of arithmetic, a comparison with Aribas \cite{forster2024aribas}, an interactive interpreter for big integer arithmetic and multi-precision floating-point arithmetic developed by Otto Forster, would be valuable. This system includes implementations of various factorisation algorithms and primality tests, albeit without program extraction from proofs. Once the program extraction approach presented in this article is further refined, a runtime comparison would be worthwhile.
Of course, many other computer algebra systems could also serve as suitable candidates for comparison.

\subsection{Connection between Formal and Textbook Proofs}
Each of the proofs presented above is based on an implementation in the proof assistant Minlog. The primary purpose of the formalised proofs is to ensure that all proofs have been verified by a computer and that we can apply formal program extraction. Internally, Minlog stores proofs as proof terms.

The textbook-style proofs provided on paper are primarily intended to enhance the reader's understanding.  However, in many cases, we have referred to the Minlog proofs directly and presented only a proof sketch in written form.
Additionally, we have pointed out that, in case of any uncertainties regarding details, the Minlog implementation can serve as a useful reference. This is particularly valuable in Minlog, as its tactic script aims to reconstruct proofs in natural language.

In general, it would be highly interesting to further integrate these two dimensions of proofs, effectively adding depth to the proof structure. Proofs could be stored in a custom file format that, at the lowest level, contain a machine-verifiable proof, while higher levels provide a proof sketch that users can refine as needed by exploring deeper layers of the proof in specific sections.

\appendix
\vspace{1cm}

\bibliographystyle{plainurl}
\bibliography{EuklidMinlog}
\newpage

\section{Runtime Measurements}
\label{Sec:Runtime}
To generate numbers with that many digits, we used the random number generator provided on the following webpage:  \url{https://numbergenerator.org/random-100000-digit-number-generator}.
Note that the number of digits is given in the decimal representation.
Larger prime numbers are taken from \url{https://www.walter-fendt.de/html5/mde/primenumbers_de.htm}.

The red line in each graphic shows the least-squares approximating polynomial of degree at most $9$ that satisfies $f(0)=0$.  If degree 9 did not produce a meaningful result, we used  a smaller degree bound instead. This is noted separately in each case. Please note that this does not mean that the runtime is actually given by this polynomial. It is only the best polynomial approximation that can be determined from the measurement data.

\subsection{Square Root Algorithm by Interval Bisection}
\label{Sec:AppSqrt}
 \vspace{5mm}
\begin{minipage}[c]{0.6\linewidth}

  \centering
  \includegraphics[width=\linewidth]{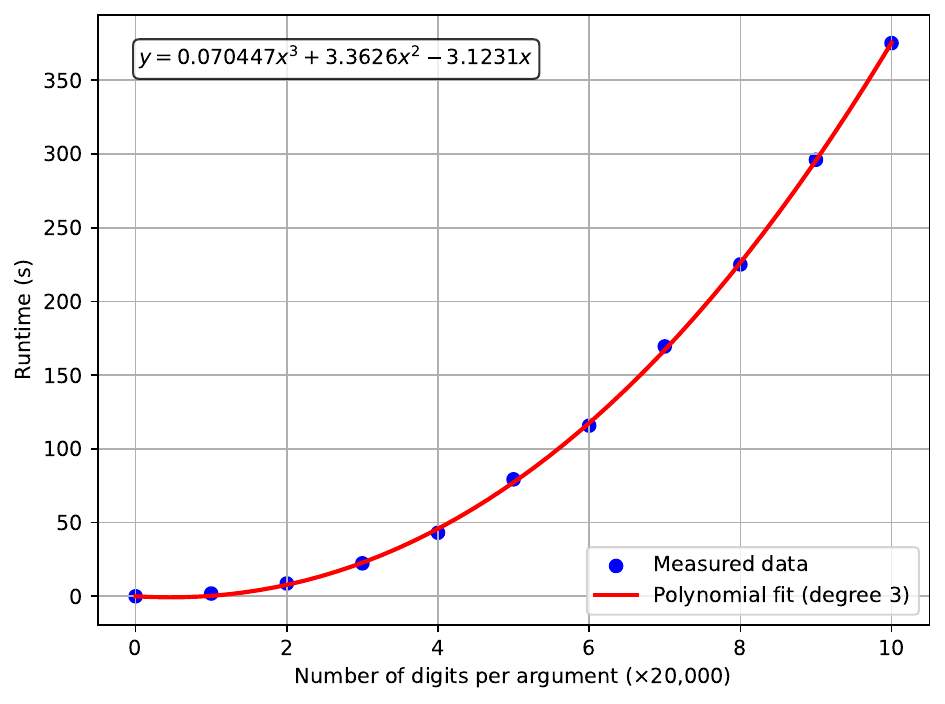}
\end{minipage}\hfill
\begin{minipage}[c]{0.35\linewidth}
  \centering
  \small
  \begin{tabular}{r r}
\hline
Digits & Runtime (s) \\
\hline
20\,000  & 1.88 \\
40\,000   & 8.69 \\
60\,000  & 22.37 \\
80\,000  & 43.12 \\
100\,000  & 79.33 \\
120\,000  & 115.76  \\
140\,000  & 169.47 \\
160\,000  & 225.10 \\
180\,000  & 296.10 \\
200\,000 & 375.22 \\
\hline
\end{tabular}
\end{minipage}
\subsection{Fast Square Root Algorithm}
\label{Sec:AppFastSqrt}
 \vspace{5mm}
\begin{minipage}[c]{0.6\linewidth}

  \centering
  \includegraphics[width=\linewidth]{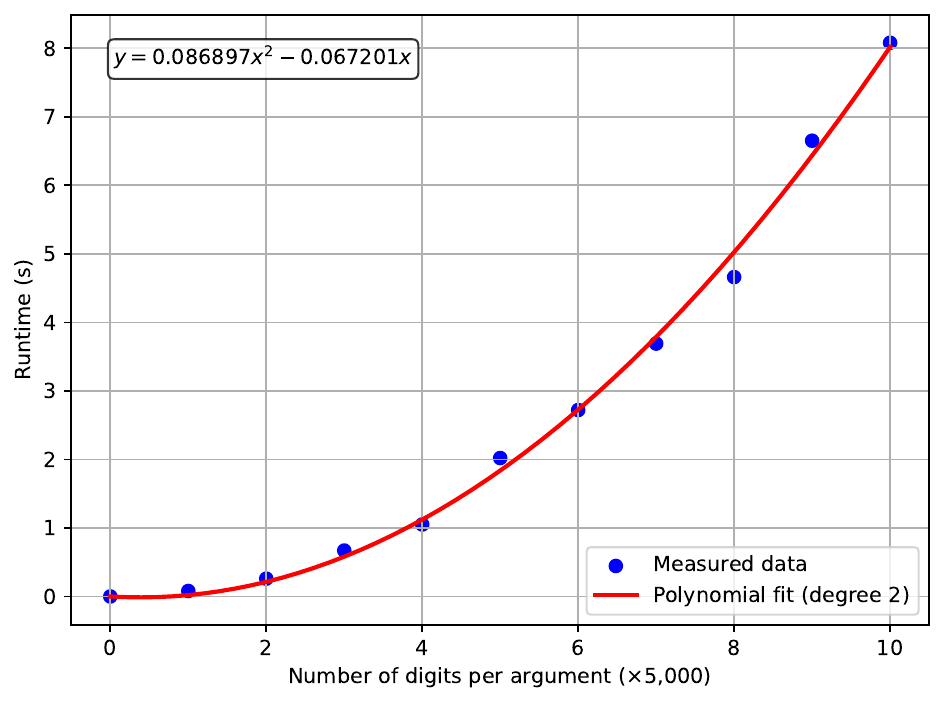}\flushleft
  \small The approximating polynomial was capped at degree~6.
\end{minipage}\hfill
\begin{minipage}[c]{0.35\linewidth}
  \centering
  \small
  \begin{tabular}{r r}
\hline
Digits & Runtime (s) \\
\hline
5\,000  & 0.08 \\
10\,000   & 0.26 \\
15\,000  & 0.67 \\
20\,000  & 1.05 \\
25\,000  & 2.02 \\
30\,000  & 2.72  \\
35\,000  & 3.69 \\
40\,000  & 4.66 \\
45\,000  & 6.65 \\
50\,000 & 8.08\\
\hline
\end{tabular}
\end{minipage}

\subsection{Euclidean Algorithm}
\label{Sec:AppEuclid}
\vspace{5mm}
\begin{minipage}[c]{0.6\linewidth}

  \centering
  \includegraphics[width=\linewidth]{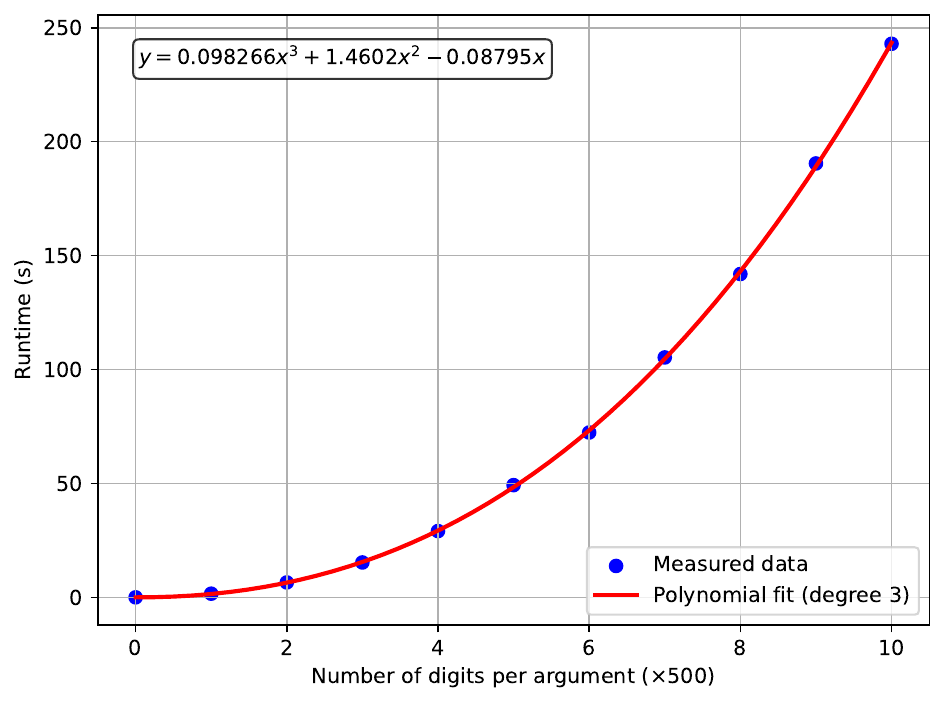}
\end{minipage}\hfill
\begin{minipage}[c]{0.35\linewidth}
  \centering
  \small
  \begin{tabular}{r r}
\hline
Digits & Runtime (s) \\
\hline
500  & 1.61\\
1\,000   & 6.55\\
1\,500 & 15.31 \\
2\,000  & 29.09 \\
2\,500  & 49.25 \\
3\,000  & 72.34\\
3\,500  & 105.28\\
4\,000  & 141.89\\
4\,500 & 190.42\\
5\,000 & 242.94\\
\hline
\end{tabular}
\end{minipage}
\subsection{Stein's Algorithm}
\label{Sec:AppStein}
\vspace{5mm}
\begin{minipage}[c]{0.6\linewidth}

  \centering
  \includegraphics[width=\linewidth]{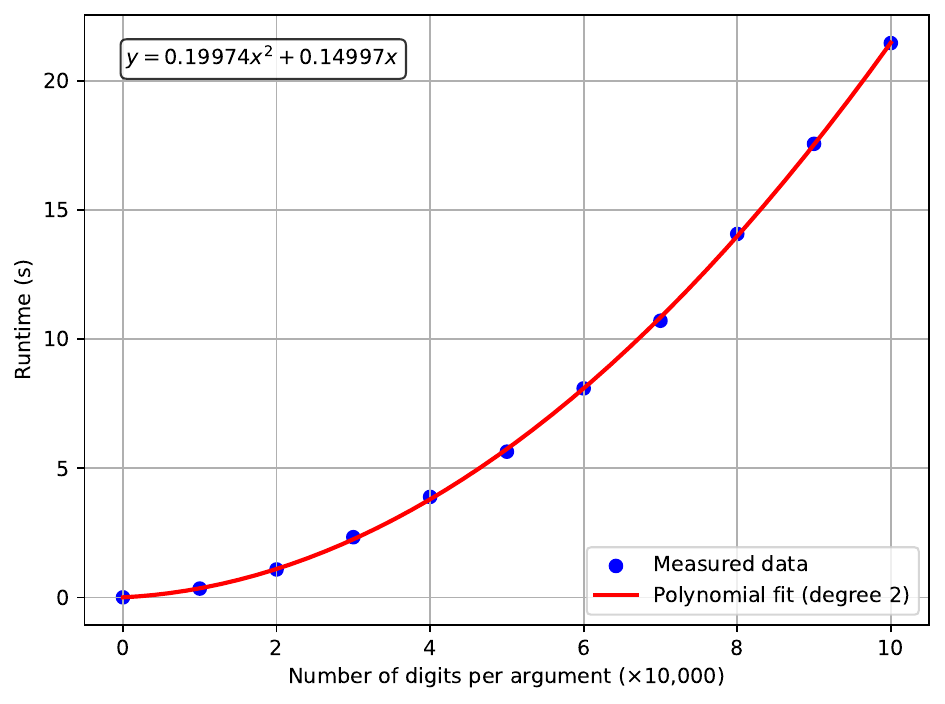}
\end{minipage}\hfill
\begin{minipage}[c]{0.35\linewidth}
  \centering
  \small
  \begin{tabular}{r r}
\hline
Digits & Runtime (s) \\
\hline
10\,000  & 0.34 \\
20\,000   & 1.08 \\
30\,000  & 2.33 \\
40\,000  & 3.89 \\
50\,000  & 5.64 \\
60\,000  & 8.09 \\
70\,000  & 10.71 \\
80\,000  & 14.07 \\
90\,000  & 17.56 \\
100\,000 & 21.46 \\
\hline
\end{tabular}
\end{minipage}
\subsection{Extended Stein's Algorithm}
\label{Sec:AppExtStein}

\vspace{5mm}
\begin{minipage}[c]{0.6\linewidth}

  \centering
  \includegraphics[width=\linewidth]{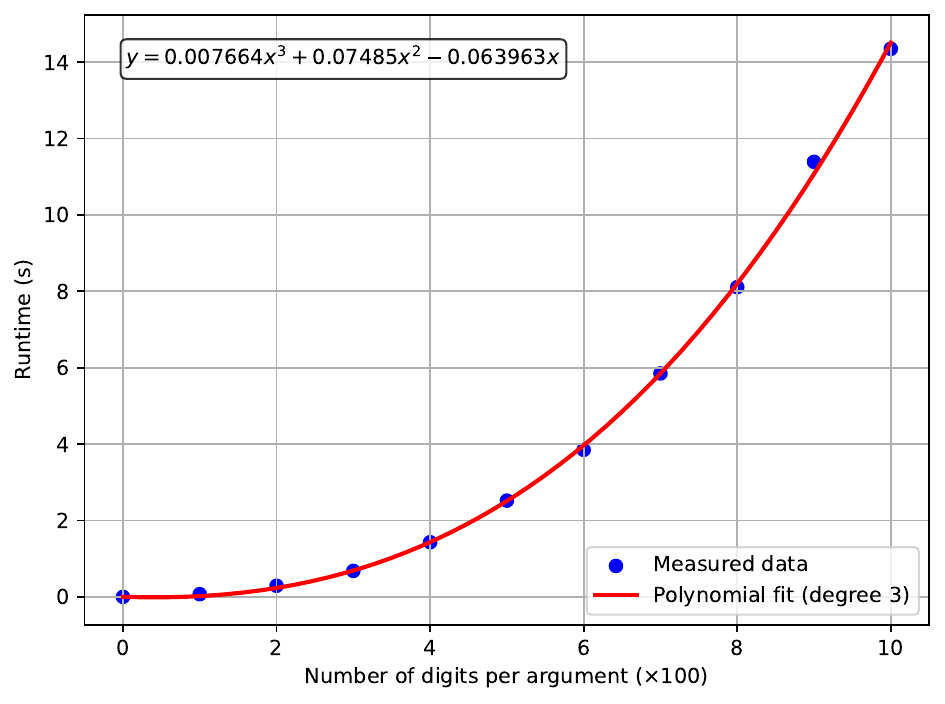}
  \flushleft
  \small The approximating polynomial was capped at degree~3.
\end{minipage}\hfill
\begin{minipage}[c]{0.35\linewidth}
  \centering
  \small
  \begin{tabular}{r r}
\hline
Digits & Runtime (s) \\
\hline
100   & 0.07 \\
200   & 0.29  \\
300   & 0.68 \\
400   & 1.43 \\
500   & 2.52 \\
600   & 3.85 \\
700   & 5.85 \\
800   & 8.11 \\
900   & 11.39 \\
1\,000 & 14.35 \\
\hline
\end{tabular}
\end{minipage}
\subsection{Extended Euclid's Algorithm (for binary numbers)}
\label{Sec:AppExtEuclid}
\vspace{5mm}
\begin{minipage}[c]{0.6\linewidth}

  \centering
  \includegraphics[width=\linewidth]{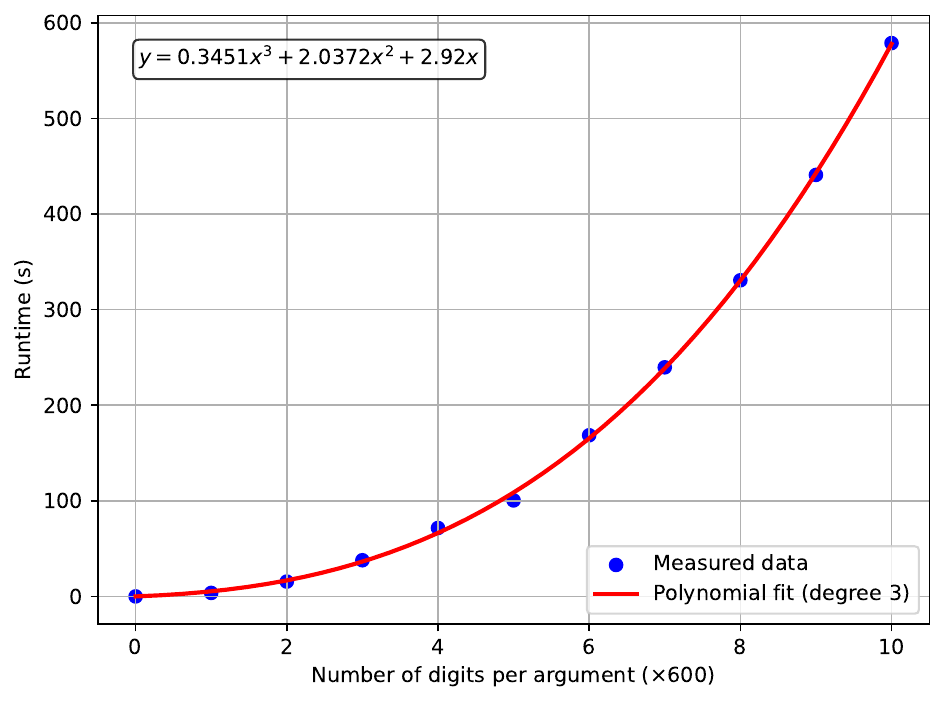}
\end{minipage}\hfill
\begin{minipage}[c]{0.35\linewidth}
  \centering
  \small
\begin{tabular}{r r}
\hline
Digits & Runtime (s) \\
\hline
600   & 3.65   \\
1\,200 & 15.51  \\
1\,800 & 37.83  \\
2\,400 & 71.56  \\
3\,000 & 100.44 \\
3\,600 & 168.59\\
4\,200 & 239.67\\
4\,800 & 330.71\\
5\,400 & 440.97 \\
6\,000 & 578.88 \\
\hline
\end{tabular}
\end{minipage}

\subsection{Existence of the Prime Factorisation}
\label{Sec:AppPrimeFactor}
\vspace{5mm}
\begin{minipage}[c]{0.4\linewidth}
\begin{tabular}{r r}
\hline
Prime number & Runtime (s) \\
\hline
10\,000\,019 & 2.39   \\
20\,000\,003 & 4.54  \\
30\,000\,001 & 6.64   \\
40\,000\,003 & 8.87  \\
50\,000\,017 & 10.74 \\
60\,000\,011 & 12.79\\
70\,000\,027 & 15.21 \\
80\,000\,023 & 17.14 \\
90\,000\,049 & 18.78 \\
100\,000\,007 & 20.85 \\
110\,000\,017 & 23.23  \\
120\,000\,007 &	24.87\\
\hline
\end{tabular}
\end{minipage}\hfill
\begin{minipage}[c]{0.4\linewidth}
\begin{tabular}{r r}
\hline
Prime power & Runtime (s) \\
\hline
$2^{1000}$ & 0.06\\
$3^{1000}$ & 28.01\\
$5^{1000}$ & 58.76\\
$47^{40}$ & 2.63\\
$47^{80}$ & 10.54\\
$47^{120}$ & 24.34\\
$47^{160}$ &  42.70 \\
$47^{200}$ & 66.32 \\
$101^{20}$ & 1.80\\
$131^{20}$ & 2.49\\
$233^{20}$ & 5.21\\
$1013^{20}$ & 34.90\\
\hline
\end{tabular}
\end{minipage}\vspace{5mm}
\noindent
\begin{minipage}[c]{0.4\linewidth}
\begin{tabular}{r r}
\hline
Twin primes & Runtime (s) \\
\hline
$1\,019\cdot 1\,021$ & 0.67\\
$2\,027\cdot 2\,029$ & 2.18\\
$3\,167\cdot 3\,169$ & 4.83\\
$4\,001\cdot 4\,003$ & 7.58\\
$5\,009\cdot 5\,011$ & 11.52\\
$6\,089\cdot 6\,091$ & 16.67\\
$7\,127\cdot 7\,129$ & 22.35\\
$8\,009\cdot 8\,011$ & 27.52\\
$8\,999\cdot 9\,001$ & 34.45 \\
$10\,007\cdot 10\,009$ &  42.93\\
\hline
\end{tabular}
\end{minipage}

\subsection{Uniqueness of the Prime Factorisation}
\label{Sec:AppUniquess}
\vspace{5mm}
\begin{tabular}{r r}
\hline
Prime factors & Runtime (s) \\
\hline
0 & 0.01 \\
50   & 0.28   \\
100 & 2.32 \\
150 & 10.79  \\
200 & 28.68  \\
250 & 61.74 \\
300 & 114.82\\
350 & 203.70\\
400 & 310.26\\
450 & 489.55 \\
500 & 699.24 \\
\hline
\end{tabular}
\end{document}